\documentclass[12pt]{amsart}

\usepackage[dvips]{graphicx}
\usepackage{amsmath,amssymb,amsthm,wrapfig,amsfonts,enumerate,latexsym}
\usepackage{bm}
\usepackage{tikz}
\usepackage{here}
\usepackage{url}

\usepackage{comment,color}

\makeatletter

\@addtoreset{figure}{section}
\makeatother

\newtheorem{theorem}{Theorem}[section]
\newtheorem{lemma}[theorem]{Lemma}
\newtheorem{corollary}[theorem]{Corollary}
\newtheorem{proposition}[theorem]{Proposition}

\theoremstyle{definition}

\newtheorem{definition}[theorem]{Definition}
\newtheorem{remark}[theorem]{Remark}
\newtheorem{condition*}[theorem]{Condition}
\newtheorem*{acknowledgement}{Acknowledgements}       
 
\numberwithin{equation}{section}

\usepackage{caption}
\usepackage{subcaption}

\begin{document}

\title[A non-geodesic analogue of Reshetnyak's majorization theorem]{A non-geodesic analogue of \\Reshetnyak's majorization theorem}
\author[T. Toyoda]{Tetsu Toyoda}
\thanks{This work is supported in part by JSPS KAKENHI Grant Number JP21K03254}
\email[Tetsu Toyoda]{toyoda@cc.kogakuin.ac.jp}
\address[Tetsu Toyoda]
{\endgraf Kogakuin University, \endgraf 2665-1, Nakano, Hachioji, Tokyo, 192-0015 Japan}

\keywords{Reshetnyak's majorization theorem, $\mathrm{Cycl}_n (\kappa )$ space, $\mathrm{CAT}(\kappa )$ space, $\mathrm{Wir}_n$ space, 
the $\boxtimes$-inequalities, the weighted quadruple inequalities}
\subjclass[2010]{Primary 53C23; Secondary 51F99}

\begin{abstract}
For any real number $\kappa$ and any integer $n\geq 4$, the $\mathrm{Cycl}_n (\kappa )$ condition introduced by Gromov (2001) 
is a necessary condition for a metric space to 
admit an isometric embedding into a $\mathrm{CAT}(\kappa )$ space. 
It is known that for geodesic metric spaces, the $\mathrm{Cycl}_4 (\kappa )$ condition is equivalent to being $\mathrm{CAT}(\kappa )$. 
In this paper, we prove an analogue of Reshetnyak's majorization theorem 
for (possibly non-geodesic) metric spaces that satisfy the $\mathrm{Cycl}_4 (\kappa )$ condition. 
It follows from our result that for general metric spaces, the $\mathrm{Cycl}_4 (\kappa )$ condition 
implies the $\mathrm{Cycl}_n (\kappa )$ conditions for all integers $n\geq 5$, 
although Gromov stated that this implication is apparently not true. 
\end{abstract}

\maketitle

\section{Introduction}\label{intro-sec}
To find a characterization of those metric spaces that admit an isometric embedding into a $\mathrm{CAT}(\kappa )$ space 
is a longstanding open problem posed by Gromov 
(see \cite[Section 1.4]{ANN}, \cite[Section 1.19+]{Gr1} and \cite[\S15]{Gr2}). 
On the other hand, it is known that 
various conditions for a general metric space become equivalent to being $\mathrm{CAT}(\kappa )$ 
under the assumption that the metric space is geodesic. 
The $\mathrm{Cycl}_4 (\kappa )$ condition defined by Gromov \cite{Gr2} is one of such conditions (see \cite[\S 15]{Gr2}). 
In this paper, we prove an analogue of Reshetnyak's majorization theorem (see \cite{R}) 
for (possibly non-geodesic) metric spaces that satisfy the $\mathrm{Cycl}_4 (\kappa )$ condition. 
Our result shows that metric spaces with the $\mathrm{Cycl}_4 (\kappa )$ condition have more 
properties in common with $\mathrm{CAT}(\kappa )$ spaces than expected. 
In particular, it follows from our result that every metric space with the $\mathrm{Cycl}_4 (\kappa )$ condition 
satisfies the $\mathrm{Cycl}_n (\kappa )$ conditions for all integers $n\geq 5$, 
although Gromov stated that this is apparently not true (see Subsection \ref{gromov-remark-subsec}).

For a real number $\kappa$, we denote by $M_{\kappa}^{2}$ the complete, simply-connected, two-dimensional Riemannian manifold of 
constant Gaussian curvature $\kappa$, and by $d_{\kappa}$ the distance function on $M_{\kappa}^2$. 
Let $D_{\kappa}$ be the diameter of $M_{\kappa}^2$. Thus we have 
\begin{equation*}
D_{\kappa}
=
\begin{cases}
\frac{\pi}{\sqrt{\kappa}}\quad &\textrm{if }\kappa >0,\\
\infty\quad &\textrm{if }\kappa\leq 0.
\end{cases}
\end{equation*}
For a positive integer $n$ and an integer $m$, we denote by $\lbrack m\rbrack_n$ the element of 
$\mathbb{Z}/n\mathbb{Z}$ represented by $m$. 
We recall the definition of the $\mathrm{Cycl}_n (\kappa )$ conditions introduced in \cite{Gr2}. 

\begin{definition}\label{Cycl-def}
Fix $\kappa\in\mathbb{R}$ and an integer $n \ge 4$. 
Let $(X,d_X )$ be a metric space. 
We say that $X$ is 
a {\em $\mathrm{Cycl}_n (\kappa )$ space} or that 
$X$ satisfies the {\em $\mathrm{Cycl}_n (\kappa )$ condition} 
if for any map $f:\mathbb Z / n\mathbb Z\to X$ with 
\begin{equation}\label{Cycl-def-diameter-restriction-ineq}
\sum_{i\in\mathbb{Z}/n\mathbb{Z}}d_X \left(f(i),f(i+\lbrack 1\rbrack_n )\right) <2 D_{\kappa},
\end{equation}
there exists a map $g:\mathbb Z / n\mathbb Z\to M_{\kappa}^2$ such that 
\begin{equation*}
d_{\kappa}(g(i),g(i+\lbrack 1\rbrack_n ))\leq d_X (f(i),f(i+\lbrack 1\rbrack_n )),\quad
d_{\kappa}(g(i),g(j))\geq d_X (f(i),f(j))
\end{equation*}
for any $i,j\in\mathbb{Z}/n\mathbb{Z}$ with $j\neq i+\lbrack 1\rbrack_n$ and $i\neq j+\lbrack1\rbrack_n$. 
\end{definition}

\begin{remark}
The original definition of $\mathrm{Cycl}_n (\kappa )$ spaces  
in \cite[\S 7]{Gr2} requires the existence of a map 
$g':\mathbb Z / n\mathbb Z\to M_{\kappa '}^2$ for some $\kappa' \leq\kappa$ 
such that 
\begin{equation*}
d_{\kappa'}(g'(i),g'(i+\lbrack 1\rbrack_n ))\leq d_X (f(i),f(i+\lbrack 1\rbrack_n )),\quad
d_{\kappa'}(g'(i),g'(j))\geq d_X (f(i),f(j))
\end{equation*}
for any $i,j\in\mathbb{Z}/n\mathbb{Z}$ with $j\neq i+\lbrack 1\rbrack_n$ and $i\neq j+\lbrack1\rbrack_n$, 
instead of the existence of a map $g$ as in Definition \ref{Cycl-def}. 
As was mentioned in \cite{Gr2}, this definition is equivalent to Definition \ref{Cycl-def}. 
In fact, the existence of such a map $g'$ implies 
the existence of a map $g$ as in Definition \ref{Cycl-def} by Reshetnyak's majorization theorem. 
\end{remark}

\begin{remark}
In  \cite[\S 7]{Gr2}, the assumption \eqref{Cycl-def-diameter-restriction-ineq} was not stated explicitly. 
It was just remarked  that we have to consider only maps $f: \mathbb{Z}/n\mathbb{Z}\to X$ 
with``small" images $f(\mathbb{Z}/n\mathbb{Z})$ when $\kappa >0$. 
\end{remark}

Concerning $\mathrm{Cycl}_n (\kappa )$ spaces and $\mathrm{CAT}(\kappa )$ spaces, 
Gromov \cite{Gr2} established the following fact. 
For the definition of $\mathrm{CAT}(\kappa )$ spaces, see Definition \ref{CAT-def} in this paper. 

\begin{theorem}[Gromov \cite{Gr2}]\label{Cycl-facts-th}
Fix $\kappa\in\mathbb{R}$. The following two assertions hold true. 
\begin{enumerate}
\item[$(1)$]
A metric space $(X,d_X )$ is $\mathrm{CAT}(\kappa )$ if and only if $X$ is $\mathrm{Cycl}_4 (\kappa )$ and $D_{\kappa}$-geodesic. 
Here, we say $X$ is $D_{\kappa}$-geodesic if any $x,y\in X$ with $d_X (x,y)<D_{\kappa}$ can be joined by a geodesic segment in $X$. 
\item[$(2)$]
Every $\mathrm{CAT}(\kappa )$ space is $\mathrm{Cycl}_n (\kappa )$ for all integers $n\geq 4$. 
\end{enumerate}
\end{theorem}

On the other hand, 
the $\mathrm{Cycl}_4 (\kappa )$ condition generally does not imply the isometric embeddability into a 
$\mathrm{CAT}(\kappa )$ space without assuming the metric space is $D_{\kappa}$-geodesic. 
In fact, Nina Lebedava constructed a $6$-point $\mathrm{Cycl}_4 (0)$ space 
that does not admit an isometric embedding into any $\mathrm{CAT}(0)$ space (see \cite[\S 7.2]{AKP}). 
Moreover, it follows from the result of Eskenazis, Mendel and Naor \cite{EMN} that 
there also exists a $\mathrm{Cycl}_4 (0)$ space that does not admit a coarse embedding into any $\mathrm{CAT}(0)$ space (see \cite[p.116]{toyoda-five}).

The following theorem is our main result, which can be viewed as an analogue of Reshetnyak's majorization theorem for 
$\mathrm{Cycl}_{4}(\kappa )$ spaces.

\begin{theorem}\label{nongeodesic-majorization-th}
Let $\kappa\in\mathbb{R}$. 
If $X$ is a $\mathrm{Cycl}_4 (\kappa )$ space, 
then for any integer $n\geq 3$, and 
for any map $f:\mathbb{Z}/n\mathbb{Z}\to X$ that satisfies 
\begin{equation*}
\sum_{i\in\mathbb{Z}/n\mathbb{Z}}d_X \left( f(i),f(i+\lbrack 1\rbrack_n )\right)<2 D_{\kappa},\quad
f(j)\neq f(j+\lbrack 1\rbrack_n )
\end{equation*}
for every $j\in\mathbb{Z}/n\mathbb{Z}$, 
there exists a map $g:\mathbb{Z}/n\mathbb{Z}\to M_{\kappa}^2$ 
that satisfies the following two conditions: 
\begin{enumerate}
\item[$(1)$]
For any $i,j\in\mathbb{Z}/n\mathbb{Z}$, we have 
\begin{equation*}
d_{\kappa}(g(i),g(i+\lbrack 1\rbrack_n )) =d_X (f(i),f(i+\lbrack 1\rbrack_n )),\quad 
d_{\kappa}(g(i),g(j))\geq d_X (f(i),f(j)).
\end{equation*}
\item[$(2)$]
For any $i,j\in\mathbb{Z}/n\mathbb{Z}$ with $i\neq j$, we have 
$\lbrack g(i),g(j)\rbrack\cap\lbrack g(i-\lbrack 1\rbrack_n ),g(i+\lbrack 1\rbrack_n )\rbrack\neq\emptyset$, 
where we denote by $\lbrack a,b\rbrack$ the line segment in $M_{\kappa}^2$ with endpoints $a$ and $b$. 
\end{enumerate}
\end{theorem}

Note that 
when the polygon with vertices $g(i)$, $i\in\mathbb{Z}/n\mathbb{Z}$ is non-degenerate, 
the condition $(2)$ in Theorem \ref{nongeodesic-majorization-th} 
means that this polygon is convex.

\subsection{Gromov's remark about the $\mathrm{Cycl}_n (\kappa )$ conditions}\label{gromov-remark-subsec}
Theorem \ref{Cycl-facts-th} tells us that 
the $\mathrm{Cycl}_4 (\kappa )$ condition implies the $\mathrm{Cycl}_n (\kappa )$ conditions for all integers $n\geq 5$ 
under the assumption that the metric space is $D_{\kappa}$-geodesic. 
In the study of upper curvature bound for general metric spaces, 
it is natural to ask whether this implication is true without assuming that the metric space is $D_{\kappa}$-geodesic. 
Concerning this question, Gromov \cite[\S 15, Remarks.(b)]{Gr2} stated 
``We shall see later on that $\mathrm{Cycl}_4 \Rightarrow \mathrm{Cycl}_k$ for all $k\geq 5$ in the geodesic case 
but this is apparently not so in general." 
However, we prove as a direct consequence of Theorem \ref{nongeodesic-majorization-th} that 
this implication actually holds true without assuming that the metric space is $D_{\kappa}$-geodesic:

\begin{theorem}\label{Cycl-th}
Let $\kappa\in\mathbb{R}$. 
Every $\mathrm{Cycl}_4 (\kappa )$ space 
is $\mathrm{Cycl}_n (\kappa )$ for all integers $n\geq 5$. 
\end{theorem}

\subsection{Gromov's question about the Wirtinger inequalities}
In \cite[\S 6]{Gr2}, Gromov also intoduced the following 
conditions for metric spaces. 

\begin{definition}
Fix an integer $n\geq 4$. 
We say that a metric space $(X,d_{X})$ is a {\em $\mathrm{Wir}_n$ space}  
if any map $f: \mathbb Z / n\mathbb Z \to X$ satisfies 
\begin{equation}\label{W_nj}
0\leq
\sin^2 \frac{j \pi}{n}\sum_{i \in \mathbb Z / n\mathbb Z} d_X (f(i), f(i+\lbrack 1\rbrack_n ))^2
-\sin^2 \frac{\pi}{n}\sum_{i \in \mathbb Z / n\mathbb Z} d_X (f(i), f(i+\lbrack j\rbrack_n ))^2
\end{equation}
for every $j\in\mathbb{Z}\cap\lbrack 2,n-2\rbrack$. 
\end{definition}

The inequalities \eqref{W_nj} can be thought of as a discrete and nonlinear analogue of 
classical Wirtinger's inequality for functions on $S^1$. 
Every Euclidean space is $\mathrm{Wir}_n$ for all integers $n\geq 4$, 
which was first proved by Pech \cite{Pe} 
before Gromov introduced the notion of $\mathrm{Wir}_n$ spaces. 
Therefore, it follows from the definition of $\mathrm{Cycl}_n (0)$ spaces that every 
$\mathrm{Cycl}_n (0)$ space is $\mathrm{Wir}_n$ for each integer $n\geq 4$. 
Thus, for general metric spaces,  
the following implications 
are true for each integer $n\geq 4$: 
\begin{equation*}
\mathrm{CAT}(0)
\Longrightarrow
\mathrm{Cycl}_n (0)
\Longrightarrow
\mathrm{Wir}_n .
\end{equation*}
In \cite[p.133, \S 25, Question]{Gr2}, Gromov posed the question of 
whether the implication $\mathrm{Cycl}_4 (0)\Rightarrow\mathrm{Wir}_n$ holds true for every integer $n\geq 5$ 
without assuming that the metric space is geodesic. 
Kondo, Toyoda and Uehara \cite{KTU} answered this question affirmatively. 

\begin{theorem}[\cite{KTU}]\label{KTU-th}
Every $\mathrm{Cycl}_4 (0)$ space 
is $\mathrm{Wir}_n$ for all integers $n\geq 4$. 
\end{theorem}
Since Theorem \ref{Cycl-th} implies the stronger implication 
$\mathrm{Cycl}_4 (0)\Rightarrow\mathrm{Cycl}_n (0)$ for every integer $n\geq 4$,  
Theorem \ref{Cycl-th} strengthens Theorem \ref{KTU-th} and gives another proof of it.

\subsection{The $\boxtimes$-inequalities}
It was remarked in \cite[\S 7]{Gr2} that the $\mathrm{Cycl}_4 (0)$ condition is equivalent to the validity of 
a certain family of inequalities, defined as follows. 

\begin{definition}
We say that a metric space $(X,d_X )$ satisfies 
the {\em $\boxtimes$-inequalities} if
for any $t, s \in\lbrack0,1\rbrack$ and any $x,y,z,w\in X$, we have 
\begin{multline*}
0
\leq
(1-t)(1-s) d_X (x,y)^2 +t(1-s) d_X (y,z)^2 +ts d_X (z,w)^2 \\
+(1-t)s d_X (w,x)^2 -t(1-t) d_X (x,z)^2 -s(1-s) d_X (y,w)^2.
\end{multline*}
\end{definition}
Gromov \cite{Gr2} and Sturm \cite{St} proved independently that every $\mathrm{CAT}(0)$ space satisfies the 
$\boxtimes$-inequalities. 
The name ``$\boxtimes$-inequalities" is based on a notation used by Gromov \cite{Gr2}, and was used in \cite{KTU} and \cite{toyoda-five}.  
Sturm \cite{St} called these inequalities the {\em weighted quadruple inequalities}. 
Gromov stated the following fact in \cite[\S 7]{Gr2}. 
\begin{theorem}[\cite{Gr2}]\label{Cycl4-boxtimes-th}
A metric space is 
$\mathrm{Cycl}_4 (0)$ if and only if it satisfies the $\boxtimes$-inequalities. 
\end{theorem}
For a proof of Theorem \ref{Cycl4-boxtimes-th}, 
see Section \ref{boxtimes-sec} of this paper. 
The following two corollaries follow from 
Theorem \ref{nongeodesic-majorization-th} and Theorem \ref{Cycl-th} immediately. 

\begin{corollary}\label{boxtimes-nongeodesic-majorization-coro}
If a metric space $X$ satisfies the $\boxtimes$-inequalities, 
then for any integer $n\geq 3$, and 
for any map $f:\mathbb{Z}/n\mathbb{Z}\to X$ that satisfies 
$f(j)\neq f(j+\lbrack 1\rbrack_n )$ for every $j\in\mathbb{Z}/n\mathbb{Z}$, 
there exists a map $g:\mathbb{Z}/n\mathbb{Z}\to\mathbb{R}^2$ 
that satisfies the following two conditions: 
\begin{enumerate}
\item[$(1)$]
For any $i,j\in\mathbb{Z}/n\mathbb{Z}$, we have 
\begin{equation*}
\| g(i)-g(i+\lbrack 1\rbrack_n )\| =d_X (f(i),f(i+\lbrack 1\rbrack_n )),\quad 
\| g(i)-g(j)\|\geq d_X (f(i),f(j)).
\end{equation*}
\item[$(2)$]
For any $i,j\in\mathbb{Z}/n\mathbb{Z}$ with $i\neq j$, we have 
$\lbrack g(i),g(j)\rbrack\cap\lbrack g(i-\lbrack 1\rbrack_n ),g(i+\lbrack 1\rbrack_n )\rbrack\neq\emptyset$, 
where we denote by $\lbrack a,b\rbrack$ the line segment in $\mathbb{R}^2$ with endpoints $a$ and $b$. 
\end{enumerate}
\end{corollary}

\begin{corollary}\label{boxtimes-main-coro}
If a metric space satisfies the $\boxtimes$-inequalities, then 
it is $\mathrm{Cycl}_n (0)$ for every integer $n\geq 4$. 
\end{corollary}

\subsection{Graph comparison}

We can say that the $\mathrm{Cycl}_n (\kappa )$ condition is defined by 
comparing embeddings of the cycle graph with $n$ vertices into a given metric space 
with embeddings of the same graph into $M_{\kappa}^2$ 
(see Definition 1.4 in \cite{toyoda-five}). 
By replacing the cycle graph with another graph, and $M_{\kappa}^2$ with 
another space, we can define a large number of new conditions. 
In recent years, such graph comparison conditions have been used by Lebedeva, Petrunin and Zolotov \cite{LPZ} and 
the present author \cite{toyoda-five}, among others, and play an interesting role in the study of the geometry of metric spaces.

\subsection{Outline of the proof of Theorem \ref{nongeodesic-majorization-th}}

Our proof of Theorem \ref{nongeodesic-majorization-th} is based on the idea used by Ballmann in his lecture note \cite{B} 
for proving Reshetnyak's majorization theorem. 
To use Ballmann's argument in our setting, we prove the following lemma in Section \ref{quadruple-sec}. 

\begin{lemma}\label{quadruple-p-lemma}
Let $\kappa\in\mathbb{R}$, and 
let $(X, d_X )$ be a $\mathrm{Cycl}_4 (\kappa )$ space. 
Suppose $x,y,z,w\in X$ and $x',y',z',w'\in M_{\kappa}^2$ are points such that 
\begin{align*}
&d_{\kappa}(x',y')+d_{\kappa}(y',z')+d_{\kappa}(z',w')+d_{\kappa}(w',x')<2D_{\kappa},\\
&d_X (x,y)\leq d_{\kappa}(x' ,y' ),\quad d_X (y,z)\leq d_{\kappa}(y' ,z' ),\quad d_X (z,w)\leq d_{\kappa}(z' ,w' ),\\
&d_X (w,x)\leq d_{\kappa}(w' ,x' ),\quad d_{\kappa}(x',z')\leq d_{X}(x,z).
\end{align*}
Then we have 
$d_X (y,w)\leq d_{\kappa}(y',p)+d_{\kappa}(p,w')$ for every $p\in\lbrack x',z' \rbrack$. 
\end{lemma}

For $\kappa\leq 0$, Lemma \ref{quadruple-p-lemma} follows from a straightforward computation (see Remark \ref{key-lemma-for-nonpositive}). 
In order to prove Lemma \ref{quadruple-p-lemma} for general $\kappa\in\mathbb{R}$, 
we prove a generalization of Alexandrov's lemma \cite[p.25]{BH} in Section \ref{alexandrov-sec}. 
Since this requires a delicate treatment of angle measure in $M_{\kappa}^2$, 
we recall and organize some facts about angle measure in $M_{\kappa}^2$ in Section \ref{angle-sec}.

\subsection{Organization of the paper}
The paper is organized as follows. 
In Section \ref{preliminaries-sec}, we recall some definitions and results from 
metric geometry and the geometry of $M_{\kappa}^2$. 
In Section \ref{angle-sec}, we recall and organize some facts about angle measure in $M_{\kappa}^2$. 
In Section \ref{alexandrov-sec}, we prove a generalized Alexandrov's lemma. 
In Section \ref{quadruple-sec}, 
we prove Lemma \ref{quadruple-p-lemma}. 
In Section \ref{convex-polygon-sec}, 
we discuss a certain property of convex polygons, which we use to prove Theorem \ref{nongeodesic-majorization-th}. 
In Section \ref{quadruple-Cycln-sec}, we prove Theorem \ref{nongeodesic-majorization-th} and Theorem \ref{Cycl-th}. 
In Section \ref{boxtimes-sec}, 
we present a proof of Theorem \ref{Cycl4-boxtimes-th} for completeness.

\section{Preliminaries}\label{preliminaries-sec}

In this section, we recall some definitions and results 
from metric geometry and the geometry of $M_{\kappa}^2$.  

Let $(X,d_X )$ be a metric space. 
A {\em geodesic} in $X$ is an isometric embedding of an interval of the real line 
into $X$. 
For $x,y\in X$, a {\em geodesic segment with endpoints $x$ and $y$} 
is the image of a geodesic $\gamma :\lbrack 0,d_X (x,y)\rbrack\to X$ with 
$\gamma (0)=x$, $\gamma (d_X (x,y))=y$. 
If there exists a unique geodesic segment with endpoints $x$ and $y$, 
we denote it by $\lbrack x,y\rbrack$. 
We also denote the sets $\lbrack x,y\rbrack\setminus\{ x,y\}$, 
$\lbrack x,y\rbrack\setminus\{ x\}$ and $\lbrack x,y\rbrack\setminus\{ y\}$ by $(x,y)$, $(x,y\rbrack$ and $\lbrack x,y)$, respectively. 
A metric space $X$ is called {\em geodesic} if 
for any $x,y\in X$, there exists a geodesic 
segment with endpoints $x$ and $y$. 
Let $D\in (0,\infty )$. 
We say that $X$ is {\em $D$-geodesic} 
if for any $x,y\in X$ with $d_{X}(x,y)<D$, 
there exists a geodesic segment with endpoints $x$ and $y$. 
A subset $S$ of $X$ is called {\em convex} 
if for any $x,y\in S$, every geodesic segment in $X$ with endpoints $x$ and $y$ 
is contained in $S$. 
If this condition holds for any $x,y\in S$ with $d_{X}(x,y)<D$, 
then $S$ is called {\em $D$-convex}.

\subsection{The geometry of $M_{\kappa}^2$}
Let $\kappa\in\mathbb{R}$. 
For any $x,y\in M_{\kappa}^2$ with $d_{\kappa}(x,y)<D_{\kappa}$, 
there exists a unique geodesic segment $\lbrack x,y\rbrack$ with endpoints $x$ and $y$. 
We mean by a {\em line} in $M_{\kappa}^2$ 
the image of an isometric embedding of $\mathbb{R}$ into $M_{\kappa}^2$ if $\kappa\leq 0$, and 
a great circle in $M_{\kappa}^2$ if $\kappa >0$. 
Then for any two distinct points $x,y\in M_{\kappa}^2$ with $d_{\kappa}(x,y)<D_{\kappa}$, there exits a unique 
line through $x$ and $y$, which we denote by $\ell (x,y)$. 
For any line $\ell$ in $M_{\kappa}^2$, $M_{\kappa}^2 \setminus\ell$ consists of exactly two connected components. 
We call each connected component of $M_{\kappa}^2 \setminus\ell$ a {\em side} of $\ell$. 
One side of $\ell$ is called the {\em opposite} side of the other. 
If $x,y\in M_{\kappa}^2$ lie on the same side of a line, then $d_{\kappa}(x,y)<D_{\kappa}$. 
Each side of a line is a convex subset of $M_{\kappa}^2$, and 
the union of a line and any one of its side is a $D_{\kappa}$-convex subset of $M_{\kappa}^2$.

For $x,y,z\in M_{\kappa}^2$ with 
$0<d_{\kappa}(x,y)<D_{\kappa}$ and 
$0<d_{\kappa}(y,z)<D_{\kappa}$, 
we denote by $\angle xyz\in\lbrack 0,\pi\rbrack$ the interior angle measure at $y$ of the (possibly degenerate) triangle with vertices $x$, $y$ and $z$. 
By the law of cosines for $M_{\kappa}^2$ (see \cite[p.24]{BH}), $\angle xyz\in\lbrack 0,\pi\rbrack$ 
satisfies the following formula: 
\begin{equation*}
\cos\angle xyz
=
\begin{cases}
\dfrac{d_{\kappa}(x,y)^2+d_{\kappa}(y,z)^2-d_{\kappa}(z,x)^2}{2d_{\kappa}(x,y) d_{\kappa}(y,z)},\quad &\textrm{if }\kappa =0,\vspace{1mm}\\
\dfrac{\cosh\left(\sqrt{-\kappa}d_{\kappa}(x,y)\right)\cosh\left(\sqrt{-\kappa}d_{\kappa}(y,z)\right)
-\cosh\left(\sqrt{-\kappa}d_{\kappa}(z,x)\right)}{
\sinh\left(\sqrt{-\kappa}d_{\kappa}(x,y)\right)\sinh\left(\sqrt{-\kappa}d_{\kappa}(y,z)\right)},\quad &\textrm{if }\kappa <0,\vspace{1mm}\\
\dfrac{\cos\left(\sqrt{\kappa}d_{\kappa}(z,x)\right)
-\cos\left(\sqrt{\kappa}d_{\kappa}(x,y)\right)\cos\left(\sqrt{\kappa}d_{\kappa}(y,z)\right)}{
\sin\left(\sqrt{\kappa}d_{\kappa}(x,y)\right)\sin\left(\sqrt{\kappa}d_{\kappa}(y,z)\right)} ,\quad &\textrm{if }\kappa >0.
\end{cases}
\end{equation*}

For $x\in M_{\kappa}^2$, a {\em ray from $x$} is the image of a geodesic 
$\gamma :\lbrack 0,D_{\kappa})\to M_{\kappa}^2$ with $\gamma (0)=x$. 
Two rays from the same point $x$ are said to be {\em opposite} if they are contained in the same line, and their intersection is $\{ x\}$. 
For any two distinct points $x,y\in M_{\kappa}^2$ with $d_{\kappa}(x,y)<D_{\kappa}$, there exists a unique ray from $x$ through $y$, 
and we denote it by $R_{xy}$. 
For any ray $R$, there exists a unique ray opposite $R$, and we denote it by $\overline{R}$. 

Suppose $o,x,y\in M_{\kappa}^2$ are points such that 
$0<d_{\kappa}(o,x)<D_{\kappa}$ and 
$0<d_{\kappa}(o,y)<D_{\kappa}$. 
Then we have $\angle xoy=\angle x'oy'$ 
for any $x'\in R_{ox}\setminus\{ o\}$ and any $y' \in R_{oy}\setminus\{ o\}$. 
We have $\angle xoy=0$ (resp. $\angle xoy=\pi$) 
if and only if $y\in R_{ox}$ (resp. $y\in\overline{R_{ox}}$). 
If $\ell (o,x)\neq\ell (o,y)$, then 
we have 
$R_{ox}\setminus\{ o\} =\ell (o,x)\cap S$, where 
$S$ is the side of $\ell (o,y)$ containing $x$. 
Let $p\in R_{ox}$ and $q\in\overline{R_{ox}}$. 
If $d_{\kappa}(p,o)+d_{\kappa}(o,q)<D_{\kappa}$, then 
$o\in\lbrack p,q\rbrack$, and 
$d_{\kappa}(p,o)+d_{\kappa}(o,q)=d_{\kappa}(p,q)$. 
If $D_{\kappa}\leq d_{\kappa}(p,o)+d_{\kappa}(o,q)$, then $\kappa >0$, and 
$d_{\kappa}(p,o)+d_{\kappa}(o,q)+d_{\kappa}(q,p)=2D_{\kappa}$.

The following lemma follows immediately from the law of cosines. 

\begin{lemma}\label{law-of-cos-lemma}
Let $\kappa\in\mathbb{R}$. 
Suppose $x,y,z,x',y',z'\in M_{\kappa}^2$ are points such that 
\begin{equation*}
0<d_{\kappa}(x,y)=d_{\kappa}(x' ,y')<D_{\kappa},\quad
0<d_{\kappa}(y,z)=d_{\kappa}(y' ,z')<D_{\kappa} .
\end{equation*}
Then $d_{\kappa}(x,z)\leq d_{\kappa}(x' ,z' )$ if and only if 
$\angle xyz\leq\angle x'y'z'$. 
Moreover, $d_{\kappa} (x,z)=d_{\kappa}(x' ,z' )$ if and only if 
$\angle xyz=\angle x'y'z'$. 
\end{lemma}

The following lemma also follows from the law of cosines.

\begin{lemma}\label{from-law-of-cos-lemma}
Let $\kappa\in\mathbb{R}$. 
Suppose 
$x,y,z\in M_{\kappa}^2$ are three distinct points such that 
\begin{equation*}
d_{\kappa} (x,y)+d_{\kappa} (y,z)<D_{\kappa},\quad
d_{\kappa}(x,y)\leq d_{\kappa}(y,z).
\end{equation*}
Suppose $x',y',z'\in M_{\kappa}^2$ are points such that 
\begin{equation*}
d_{\kappa} (x,y)=d_{\kappa}(x',y'),\quad d_{\kappa} (y,z)=d_{\kappa} (y',z'),\quad d_{\kappa}(z,x)\leq d_{\kappa}(z',x' ).
\end{equation*}
Then $\angle y'x'z' \leq\angle yxz$. 
\end{lemma}

\begin{proof}
We consider three cases. 

\textsc{Case 1}: 
{\em $\kappa =0$.}
In this case, we have $\angle yzx\leq\angle yxz$ and 
$\angle y'z'x'\leq\angle y'x'z'$ by hypothesis, and therefore 
$\angle yzx\leq\pi /2$ and $\angle y'z'x'\leq \pi /2$. 
It follows that 
\begin{align*}
0&\leq d_{\kappa}(x,z)^2+d_{\kappa}(y,z)^2-d_{\kappa}(x,y)^2,\\
0&\leq d_{\kappa}(x',z')^2+d_{\kappa}(y',z')^2-d_{\kappa}(x',y')^2 
=d_{\kappa}(x',z')^2+d_{\kappa}(y,z)^2-d_{\kappa}(x,y)^2 ,
\end{align*}
and therefore we have 
\begin{equation}\label{from-law-phthagoas-ineq}
0\leq\alpha^2+d_{\kappa}(y,z)^2-d_{\kappa}(x,y)^2 
\end{equation}
for any $\alpha\in\lbrack d_{\kappa}(x,z),d_{\kappa}(x',z')\rbrack$. 
Define a function $f_1 :(0,\infty )\to\mathbb{R}$ by 
\begin{equation*}
f_1 (\alpha )
=
\frac{d_{\kappa}(x,y)^2+\alpha^2-d_{\kappa}(y,z)^2}{2 d_{\kappa}(x,y)\alpha} .
\end{equation*}
Then we have 
\begin{equation*}
\frac{d}{d\alpha}f_1 (\alpha )
=
\frac{1}{2d_{\kappa}(x,y)\alpha^2}
\left(
\alpha^2 +d_{\kappa}(y,z)^2 -d_{\kappa}(x,y)^2
\right)
\geq 0
\end{equation*} 
for any $\alpha\in\lbrack d_{\kappa}(x,z),d_{\kappa}(x',z')\rbrack$ by \eqref{from-law-phthagoas-ineq}, 
which implies that $\angle y'x'z' \leq\angle yxz$ 
because 
\begin{equation*}
f_1 (d_{\kappa}(x,z))
=
\cos\angle yxz,\quad
f_1 (d_{\kappa}(x',z'))
=
\cos\angle y'x'z'
\end{equation*}
by the law of cosines. 

\textsc{Case 2}: 
{\em $\kappa <0$.} 
In this case, we define a function $f_2 :(0,\infty )\to\mathbb{R}$ by 
\begin{equation*}
f_2 (\alpha )
=
\frac{\cosh\left(\sqrt{-\kappa}d_{\kappa}(x,y)\right)\cosh\left(\sqrt{-\kappa}\alpha\right)
-\cosh\left(\sqrt{-\kappa}d_{\kappa}(y,z)\right)}{
\sinh\left(\sqrt{-\kappa}d_{\kappa}(x,y)\right)\sinh\left(\sqrt{-\kappa}\alpha\right)}. 
\end{equation*}
Then we have 
\begin{align*}
\frac{d}{d\alpha}f_2 (\alpha )
&=
\frac{\sqrt{-\kappa}\left(
\cosh\left(\sqrt{-\kappa}d_{\kappa}(y,z)\right)
\cosh\left(\sqrt{-\kappa}\alpha\right)
-\cosh\left(\sqrt{-\kappa}d_{\kappa}(x,y)\right)
\right)}{\sinh\left(\sqrt{-\kappa}d_{\kappa}(x,y)\right)\sinh^2 \left(\sqrt{-\kappa}\alpha\right)}\\
&\geq
\frac{\sqrt{-\kappa}\left(
\cosh\left(\sqrt{-\kappa}d_{\kappa}(y,z)\right)
-\cosh\left(\sqrt{-\kappa}d_{\kappa}(x,y)\right)
\right)}{\sinh\left(\sqrt{-\kappa}d_{\kappa}(x,y)\right)\sinh^2 \left(\sqrt{-\kappa}\alpha\right)}
\geq 0
\end{align*}
for any $\alpha\in (0,\infty )$ by hypothesis. 
This implies that 
$\angle y'x'z' \leq\angle yxz$ 
because 
\begin{equation*}
f_2 (d_{\kappa}(x,z))
=
\cos\angle yxz,\quad
f_2 (d_{\kappa}(x',z'))
=
\cos\angle y'x'z'
\end{equation*}
by the law of cosines. 

\textsc{Case 3}: 
{\em $\kappa >0$.} 
In this case, we define a function $f_3 :(0,D_{\kappa})\to\mathbb{R}$ by 
\begin{equation*}
f_3 (\alpha )
=
\frac{\cos\left(\sqrt{\kappa}d_{\kappa}(y,z)\right)
-\cos\left(\sqrt{\kappa}d_{\kappa}(x,y)\right)\cos\left(\sqrt{\kappa}\alpha\right)}{
\sin\left(\sqrt{\kappa}d_{\kappa}(x,y)\right)\sin\left(\sqrt{\kappa}\alpha\right)}. 
\end{equation*}
Then 
\begin{equation*}
\frac{d}{d\alpha}f_3 (\alpha )
=
\frac{\sqrt{\kappa}\left(
\cos\left(\sqrt{\kappa}d_{\kappa}(x,y)\right)
-\cos\left(\sqrt{\kappa}d_{\kappa}(y,z)\right)\cos\left(\sqrt{\kappa}\alpha\right)
\right)}{\sin\left(\sqrt{\kappa}d_{\kappa}(x,y)\right)\sin^2 \left(\sqrt{\kappa}\alpha\right)}.
\end{equation*}
If $d_{\kappa}(y,z)\leq D_{\kappa}/2$, then 
\begin{equation*}
0<\sqrt{\kappa}d_{\kappa}(x,y)\leq\sqrt{\kappa}d_{\kappa}(y,z)\leq\frac{\pi}{2}
\end{equation*}
by hypothesis, 
and therefore 
\begin{align*}
\cos\left(\sqrt{\kappa}d_{\kappa}(x,y)\right)
&-\cos\left(\sqrt{\kappa}d_{\kappa}(y,z)\right)\cos\left(\sqrt{\kappa}\alpha\right) \\
&\geq
\cos\left(\sqrt{\kappa}d_{\kappa}(x,y)\right)
-
\cos\left(\sqrt{\kappa}d_{\kappa}(y,z)\right)
\geq 0
\end{align*}
for any $\alpha\in (0,D_{\kappa})$. 
If $d_{\kappa}(y,z)>D_{\kappa}/2$, then 
\begin{equation*}
0<\sqrt{\kappa}d_{\kappa}(x,y)<\frac{\pi}{2}<\sqrt{\kappa}d_{\kappa}(y,z)<\pi ,\quad
\sqrt{\kappa}d_{\kappa}(x,y)<\pi-\sqrt{\kappa}d_{\kappa}(y,z)
\end{equation*}
by hypothesis, and therefore 
\begin{align*}
\cos\left(\sqrt{\kappa}d_{\kappa}(x,y)\right)
&-\cos\left(\sqrt{\kappa}d_{\kappa}(y,z)\right)\cos\left(\sqrt{\kappa}\alpha\right) \\
&\geq
\cos\left(\sqrt{\kappa}d_{\kappa}(x,y)\right)
+\cos\left(\sqrt{\kappa}d_{\kappa}(y,z)\right) \\
&=
\cos\left(\sqrt{\kappa}d_{\kappa}(x,y)\right)
-\cos\left(\pi-\sqrt{\kappa}d_{\kappa}(y,z)\right) \\
&>
\cos\left(\sqrt{\kappa}d_{\kappa}(x,y)\right)
-\cos\left(\sqrt{\kappa}d_{\kappa}(x,y)\right) =0
\end{align*}
for any $\alpha\in (0,D_{\kappa})$. 
Thus we always have 
\begin{equation*}
0\leq
\frac{d}{d\alpha}f_3 (\alpha )
\end{equation*}
for any $\alpha\in (0,D_{\kappa})$. 
This implies that $\angle y'x'z' \leq\angle yxz$ 
because 
\begin{equation*}
f_3 (d_{\kappa}(x,z))
=
\cos\angle yxz,\quad
f_3 (d_{\kappa}(x',z'))
=
\cos\angle y'x'z'
\end{equation*}
by the law of cosines.

The above three cases exhaust all possibilities. 
\end{proof}

The following formulas follow from straightforward computation. 
See \cite[Chapter I.2]{BH} for a guide to such computations concerning the distance function on $M_{\kappa}^2$. 

\begin{lemma}\label{naibunten-lem}
Let $\kappa\in\mathbb{R}$. 
Suppose 
$x,y,z\in M_{\kappa}^2$ are points such that $x\neq z$ and 
$d_{\kappa}(a,b)<D_{\kappa}$ for any $a,b\in\{ x,y,z\}$. 
Suppose $\gamma :\lbrack 0,d_{\kappa}(x,z)\rbrack\to M_{\kappa}^2$ is the geodesic such that $\gamma (0)=x$ and $\gamma (d_{\kappa}(x,z))=z$. 
Let $t\in\lbrack 0,1\rbrack$, and let $p=\gamma\left( td_{\kappa}(x,z)\right)$. 
If $\kappa =0$, then 
\begin{equation*}
d_{\kappa}(y,p)^2
=
(1-t)d_{\kappa}(x,y)^2 +td_{\kappa}(y,z)^2 -t(1-t)d_{\kappa}(x,z)^2 .
\end{equation*}
If $\kappa <0$, then 
\begin{multline*}
\cosh\left(\sqrt{-\kappa}d_{\kappa}(y,p)\right)\\
=
\frac{1}{\sinh\left(\sqrt{-\kappa}d_{\kappa} (x,z)\right)}
\bigg(\sinh\left(\sqrt{-\kappa}(1-t)d_{\kappa}(x,z)\right)\cosh\left(\sqrt{-\kappa}d_{\kappa}(x,y)\right)\\
+\sinh\left(\sqrt{-\kappa}td_{\kappa}(x,z)\right)\cosh\left(\sqrt{-\kappa}d_{\kappa}(y,z)\right)\bigg) .
\end{multline*}
If $\kappa >0$, then 
\begin{multline*}
\cos\left(\sqrt{\kappa}d_{\kappa}(y,p)\right)\\
=
\frac{\sin\left(\sqrt{\kappa}(1-t)d_{\kappa}(x,z)\right)\cos\left(\sqrt{\kappa}d_{\kappa}(x,y)\right)
+\sin\left(\sqrt{\kappa}td_{\kappa}(x,z)\right)\cos\left(\sqrt{\kappa}d_{\kappa}(y,z)\right)}{\sin\left(\sqrt{\kappa}d_{\kappa}(x,z)\right)}.
\end{multline*}
\end{lemma}

The next corollary follows immediately from Lemma \ref{naibunten-lem}

\begin{corollary}\label{naibunten-coro}
Let $\kappa\in\mathbb{R}$. 
Suppose $x,y,z,\tilde{y}\in M_{\kappa}^2$ are points 
such that 
\begin{equation*}
d_{\kappa} (x,y)\leq d_{\kappa}(x,\tilde{y})<D_{\kappa},\quad 
d_{\kappa} (y,z)\leq d_{\kappa} (\tilde{y},z)<D_{\kappa},\quad
0<d_{\kappa}(x,z)<D_{\kappa}.
\end{equation*}
Then we have $d_{\kappa} (y,p)\leq d_{\kappa}(\tilde{y},p)$ for any $p\in\lbrack x,z\rbrack$. 
\end{corollary}

Although it is not necessary for our purpose, 
it is worth noting that 
for $\kappa\in (-\infty ,0\rbrack$, Lemma \ref{naibunten-lem} also implies the following corollary immediately. 

\begin{corollary}\label{naibunten-coro2}
Let $\kappa\in (-\infty ,0\rbrack$. 
Suppose $x,y,z,\tilde{x},\tilde{y},\tilde{z}\in M_{\kappa}^2$ are points such that 
\begin{equation*}
d_{\kappa}(x,y)\leq d_{\kappa}(\tilde{x},\tilde{y})<D_{\kappa},\quad 
d_{\kappa}(y,z)\leq d_{\kappa}(\tilde{y},\tilde{z})<D_{\kappa},\quad
0<d_{\kappa}(\tilde{x},\tilde{z})\leq d_{\kappa}(x,z)<D_{\kappa}.
\end{equation*}
Let $\gamma :\lbrack 0,d_{\kappa}(x,z)\rbrack\to M_{\kappa}^2$ and 
$\tilde{\gamma}:\lbrack 0,d_{\kappa}(\tilde{x},\tilde{z})\rbrack\to M_{\kappa}^2$ 
be the geodesics such that 
\begin{equation*}
\gamma (0)=x,\quad\gamma (d_{\kappa}(x,z))=z,\quad
\tilde{\gamma}(0)=\tilde{x},\quad\tilde{\gamma}(d_{\kappa}(\tilde{x},\tilde{z}))=\tilde{z}.
\end{equation*} 
Fix $t\in\lbrack 0,1\rbrack$, and set 
$p=\gamma (td_{\kappa}(x,z))$, $\tilde{p}=\tilde{\gamma}(td_{\kappa}(\tilde{x},\tilde{z}))$. 
Then 
$d_{\kappa} (y,p)\leq d_{\kappa}(\tilde{y},\tilde{p})$. 
\end{corollary}

\begin{remark}\label{key-lemma-for-nonpositive}
For the case in which $\kappa\in (-\infty ,0\rbrack$, we can prove Lemma \ref{quadruple-p-lemma} easily by using Corollary \ref{naibunten-coro2}. 
To prove Lemma \ref{quadruple-p-lemma} for general $\kappa\in\mathbb{R}$, we will develop another method 
in Section \ref{angle-sec}, \ref{alexandrov-sec} and \ref{quadruple-sec}
\end{remark}

\subsection{Polygons in $M_{\kappa}^2$}
Let $\kappa\in\mathbb{R}$. 
For a subset $S$ of $M_{\kappa}^2$, 
the {\em convex hull of $S$} is the intersection of all convex subsets of $M_{\kappa}^2$ containing $S$, or 
equivalently, the minimal convex subset of $M_{\kappa}^2$ containing $S$. 
We denote the convex hull of $S$ by $\mathrm{conv}(S)$. 
We recall the following well-known fact, which is 
trivial when $\kappa\leq 0$. 

\begin{proposition}\label{hemisphere-prop}
Let $\kappa\in\mathbb{R}$, and let $n\geq 3$ be an integer. 
Suppose $f:\mathbb{Z}/n\mathbb{Z}\to M_{\kappa}^2$ is a map such that 
$\sum_{i\in\mathbb{Z}/n\mathbb{Z}}d_{\kappa}(f(i),f(i+\lbrack 1\rbrack_n ))<2D_{\kappa}$. 
Then there exists a line $L$ in $M_{\kappa}^2$ such that $\mathrm{conv}(f(\mathbb{Z}/n\mathbb{Z}))$ 
is contained in one side of $L$. 
In particular, any $p,q\in\mathrm{conv}(f(\mathbb{Z}/n\mathbb{Z}))$ satisfy $d_{\kappa}(p,q)<D_{\kappa}$. 
\end{proposition}

\subsection{$\mathrm{CAT}(\kappa )$ spaces}
A \textit{geodesic triangle} in a metric space $X$ is a triple 
$\triangle =(\gamma_1 ,\gamma_2 ,\gamma_3 )$ 
of geodesics 
$\gamma_i :\lbrack a_i,b_i \rbrack\to X$ 
such that 
$\gamma_1 (b_1)=\gamma_2 (a_2)$, 
$\gamma_2 (b_2)=\gamma_3 (a_3)$ and 
$\gamma_3 (b_3)=\gamma_1 (a_1)$. 
Let $\kappa\in\mathbb{R}$. 
If the perimeter $\sum_{i=1}^3 |b_i -a_i |$ of the geodesic triangle $\triangle$ is less than $2 D_{\kappa}$, then 
there exists a geodesic triangle 
$\triangle^{\kappa}=
(\gamma^{\kappa}_1 ,\gamma^{\kappa}_2 ,
\gamma^{\kappa}_3 )$, 
$\gamma^{\kappa}_i :\lbrack a_i,b_i \rbrack\to M_{\kappa}^2$ 
in $M_{\kappa}^2$. 
Such a geodesic triangle $\triangle^{\kappa}$ is unique up to isometry of $M_{\kappa}^2$. 
The geodesic triangle $\triangle$ is said to be {\em $\kappa$-thin} 
if $d_X (\gamma_i (s),\gamma_j (t))\leq d_{\kappa}(\gamma^{\kappa}_i (s),\gamma^{\kappa}_j (t))$ for 
any $i,j\in\{ 1,2,3\}$, any $s\in\lbrack a_i ,b_i \rbrack$, 
and any $t\in\lbrack a_j ,b_j \rbrack$.

\begin{definition}\label{CAT-def}
Let $\kappa\in\mathbb{R}$. 
A metric space $X$ is called a {\em $\mathrm{CAT}(\kappa )$ space}  
if $X$ is $D_{\kappa}$-geodesic, 
and any geodesic triangle in $X$ with 
perimeter$< 2D_{\kappa}$ 
is $\kappa$-thin. 
\end{definition}

By definition, $M_{\kappa}^2$ is a $\mathrm{CAT}(\kappa )$ space. 
It is easily observed that 
if $(X,d_X )$ is a $\mathrm{CAT}(\kappa )$ space, then 
for any $x,y\in X$ with $d_{X}(x,y)<D_{\kappa}$, there exists the unique geodesic segment with endpoints $x$ and $y$. 
Every $D_{\kappa}$-convex subset of a $\mathrm{CAT}(\kappa )$ space equipped with the induced metric is 
a $\mathrm{CAT}(\kappa )$ space.

Suppose that $(X_1 ,d_1 )$ and $(X_2 ,d_2 )$ are metric spaces, and that 
$Z_1$ and $Z_2$ are closed subsets of $X_1$ and $X_2$, respectively. 
Suppose further that $Z_1$ and $Z_2$ are isometric via an isometry $f:Z_1 \to Z_2$. 
We denote by $X_1 \sqcup X_2$ the disjoint union of $X_1$ and $X_2$. 
For $x,y\in X_1 \sqcup X_2$, define $d_0 (x,y)\in\lbrack 0,\infty )$ by 
\begin{equation*}
d_0 (x,y)
=
\begin{cases}
&d_1 (x,y),\quad\textrm{if }x,y\in X_1 ,\\
&d_2 (x,y),\quad\textrm{if }x,y\in X_2 ,\\
&\min_{z\in Z_1}\{ d_1 (x ,z)+d_2 (f(z),y)\} ,\quad\textrm{if }x\in X_1 , y\in X_2 ,\\
&\min_{z\in Z_1}\{ d_1 (y ,z)+d_2 (f(z),x)\} ,\quad\textrm{if }x\in X_2 , y\in X_1 .
\end{cases}
\end{equation*}
Then $d_0$ is a semi-metric on $X_1 \sqcup X_2$. 
In other words, $d_0$ satisfies the axioms for a metric except 
the requirement that $d_0 (x,y)=0$ implies $x=y$. 
Define a relation $\sim$ on $X_1 \sqcup X_2$ by 
declaring $x\sim y$ if and only if $d_0 (x,y)=0$. 
Then $\sim$ is an equivalence relation on $X_1 \sqcup X_2$, 
and the projection $\overline{d}_0$ of $d_0$ onto the quotient space $X_0 =(X_1 \sqcup X_2 )/\sim$ is well defined. 
It is easily observed that $(X_0 ,\overline{d}_0 )$ is a metric space, 
which is called the {\em gluing of $X_1$ and $X_2$ along $f$}. 
When two geodesic segments $\lbrack a,b\rbrack\subseteq X_1$ and $\lbrack c,d\rbrack\subseteq X_2$ are isometric, 
we mean by ``the metric space obtained by gluing $X_1$ and $X_2$ {\em by identifying $\lbrack a,b\rbrack$ with $\lbrack c,d\rbrack$}"
the gluing of $X_1$ and $X_2$ along the isometry $f:\lbrack a,b\rbrack\to\lbrack c,d\rbrack$ with $f(a)=c$ and $f(b)=d$. 

If $X_1$ and $X_2$ are complete locally compact $\mathrm{CAT}(\kappa )$ spaces, 
$Z_1$ and $Z_2$ are closed $D_{\kappa}$-convex subsets of $X_1$ and $X_2$, respectively, 
and $f:Z_1 \to Z_2$ is an isometry, 
then by Reshetnyak's gluing theorem, 
the gluing of $X_1$ and $X_2$ along 
$f$ becomes a $\mathrm{CAT}(\kappa )$ space. 
For a proof of this fact, see \cite{R}, \cite[Theorem 9.1.21]{BBI} or \cite[ChapterII, Theorem 11.1]{BH}.

\section{Angle measure in $M_{\kappa}^2$}\label{angle-sec}

In this section, we recall and organize several facts about angle measure in $M_{\kappa}^2$. 
We will use these facts mainly in the next section to prove a generalization of Alexandrov's lemma. 

We start with the following three propositions, 
which are all well-known. 

\begin{proposition}\label{angle-triangle-inequality-prop}
Let $\kappa\in\mathbb{R}$. 
Suppose $o,x,y,z \in M_{\kappa}^2$ are points such that 
$0<d_{\kappa}(o,a)<D_{\kappa}$ 
for every $a\in\{ x,y,z\}$. 
Then $\angle xoz\leq\angle xoy +\angle yoz$. 
\end{proposition}

\begin{proposition}\label{angle-equality-necessary-conditions-prop}
Let $\kappa\in\mathbb{R}$. 
Suppose $o,x,y,z \in M_{\kappa}^2$ are points such that 
$0<d_{\kappa}(o,a)<D_{\kappa}$ 
for every $a\in\{ x,y,z\}$. 
Assume that $\angle xoz =\angle xoy +\angle yoz$. 
Then all of the following conditions are true: 
\begin{itemize}
\item
$y$ and $z$ do not lie on opposite sides of $\ell (o,x)$. 
\item
$x$ and $y$ do not lie on opposite sides of $\ell (o,z)$. 
\item
$x$ and $z$ do not lie on the same side of $\ell (o,y)$. 
\end{itemize}
\end{proposition}

\begin{proposition}\label{angle-equality-prop}
Let $\kappa\in\mathbb{R}$. 
Suppose $o,x,y,z\in M_{\kappa}^2$ are points such that 
$0<d_{\kappa}(o,a)<D_{\kappa}$ 
for every $a\in\{ x,y,z\}$. 
Then the identity 
$\angle xoz=\angle xoy+\angle yoz$ holds if and only if 
$y$ and $z$ do not lie on opposite sides of $\ell (o,x)$, and 
$\angle xoy\leq\angle xoz$. 
Assume in addition that $0<\angle xoz$. 
Then the identity 
$\angle xoz=\angle xoy+\angle yoz$ holds if and only if 
$y$ is neither on the opposite side of $\ell (o,z)$ from $x$, nor 
on the opposite side of $\ell (o,x)$ from $z$. 
\end{proposition}

The following corollary follows immediately from the first part of Proposition \ref{angle-equality-prop}. 

\begin{corollary}\label{angle-equality-coro}
Let $\kappa\in\mathbb{R}$. 
Suppose $o,x,y,z\in M_{\kappa}^2$ are points such that 
$0<d_{\kappa}(o,a)<D_{\kappa}$ 
for every $a\in\{ x,y,z\}$. 
If $\angle xoy =0$ or $\angle xoz=\pi$, then 
$\angle xoz=\angle xoy +\angle yoz$. 
\end{corollary}

In the rest of this section, we recall several more facts about angle measure in $M_{\kappa}^2$. 
Although all of them are also well-known, we will prove them by using the above three propositions for completeness. 

\begin{proposition}\label{diagonal-angle-prop}
Let $\kappa\in\mathbb{R}$. 
Suppose $o,x,y,z\in M_{\kappa}^2$ are points such that 
$0<d_{\kappa}(o,a)<D_{\kappa}$ for every $a\in\{ x,y,z\}$, and $d_{\kappa}(x,z)<D_{\kappa}$. 
If $\lbrack x,z\rbrack\cap\lbrack o,y\rbrack\neq\emptyset$, then 
\begin{equation*}
\angle xoz
=
\angle xoy +\angle yoz.
\end{equation*}
\end{proposition}

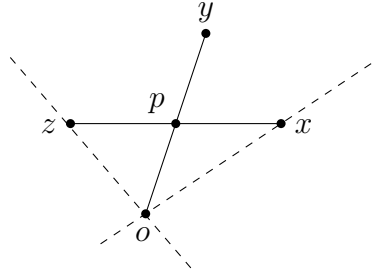
\begin{figure}[htbp]
\centering\begin{tikzpicture}[scale=0.4]
\draw[dashed] (-2,5.2) -- (4,-1.8);
\draw[dashed] (1,-1) -- (10,5);
\draw (0,3) -- (7,3);
\draw (2.5,0) -- (4.5,6);
\node [right] at (7.1,2.9) {$x$};
\node [above] at (4.5,6) {$y$};
\node [below] at (2.4,-0.1) {$o$};
\node [left] at (-0.1,2.9) {$z$};
\fill (3.5,3) circle [radius=0.15];
\fill (0,3) circle [radius=0.15];
\fill (7,3) circle [radius=0.15];
\fill (2.5,0) circle [radius=0.15];
\fill (4.5,6) circle [radius=0.15];
\node [above left] at (3.5,3) {$p$};
\end{tikzpicture}
\caption{Proof of Proposition \ref{diagonal-angle-prop}.}
\end{figure}

\begin{proof}
If $o\in\lbrack x,z\rbrack$, then $\angle xoz=\pi$, and therefore we have 
$\angle xoz
=
\angle xoy +\angle yoz$ 
by Corollary \ref{angle-equality-coro}. 
So henceforth we assume that $o\not\in\lbrack x,z\rbrack$. 
Then there exists a point $p\in\lbrack x,z\rbrack\cap (o,y\rbrack$ by hypothesis. 
Because 
the union of $\ell (o,z)$ and any one of its sides 
is $D_{\kappa}$-convex, 
the fact that $p\in\lbrack x,z\rbrack$ implies that 
$p$ is not on the opposite side of $\ell (o,z)$ from $x$. 
Similarly, $p$ is not on the opposite side of $\ell (o,x)$ from $z$. 
Therefore, if $0<\angle xoz$, then Proposition \ref{angle-equality-prop} implies that 
\begin{equation*}
\angle xoz
=\angle xop +\angle poz
=\angle xoy +\angle yoz.
\end{equation*}
So we assume in addition that $\angle xoz= 0$. 
Then $x,z\in R_{ox}\setminus\{ o\}$. 
Therefore, 
the hypothesis that $\lbrack x,z\rbrack\cap\lbrack o,y\rbrack\neq\emptyset$ implies 
$y\in R_{ox}\setminus\{ o\}$ clearly. 
It follows that 
$\angle xoy=\angle yoz=0$, and thus 
$\angle xoz=\angle xoy +\angle yoz$, which completes the proof. 
\end{proof}

The following corollary is an immediate consequence of Proposition \ref{diagonal-angle-prop}.

\begin{corollary}\label{diagonal-angle-coro}
Let $\kappa\in\mathbb{R}$. 
Suppose $o,x,y,z\in M_{\kappa}^2$ are points such that 
$0<d_{\kappa}(o,a)<D_{\kappa}$ for every $a\in\{ x,y,z\}$, and $d_{\kappa}(x,z)<D_{\kappa}$. 
If $\lbrack x,z\rbrack\cap\lbrack o,y\rbrack\neq\emptyset$, then 
$\angle xoy +\angle yoz\leq\pi$. 
\end{corollary}

We will continue to prove facts about angle measure in $M_{\kappa}^2$. 

\begin{proposition}\label{angle-betweenness-prop}
Let $\kappa\in\mathbb{R}$. 
Suppose $o,x,y,z,w\in M_{\kappa}^2$ are points such that 
$0<d_{\kappa}(o,a)<D_{\kappa}$ 
for every $a\in\{ x,y,z,w\}$. 
If we have 
\begin{align}
\angle xoz&=\angle xoy+\angle yoz,\label{angle-betweeness-xoz-hyp}\\
\angle xow&=\angle xoz +\angle zow,\label{angle-betweeness-xow-hyp}
\end{align}
then we have 
\begin{equation*}
\angle yow=\angle yoz +\angle zow,\quad
\angle xow=\angle xoy+\angle yow.
\end{equation*}
\end{proposition}

\begin{figure}[htbp]
\centering\begin{tikzpicture}[scale=0.4]
\draw[dashed] (-2,0) -- (12,0);
\draw[dashed] (9,2) -- (0,3);
\draw[dashed] (3.8,-1) -- (5.4,7);
\draw (4,0) -- (10,0);
\draw (4,0) -- (9,2);
\draw (4,0) -- (5,5);
\draw (4,0) -- (0,3);
\node [below left] at (4,0) {$o$};
\node [below] at (10,0) {$x$};
\node [right] at (9,2) {$y$};
\node [left] at (5,5) {$z$};
\node [left] at (0,3) {$w$};
\fill (4.5,2.5) circle [radius=0.15];
\fill (4,0) circle [radius=0.15];
\fill (10,0) circle [radius=0.15];
\fill (9,2) circle [radius=0.15];
\fill (5,5) circle [radius=0.15];
\fill (0,3) circle [radius=0.15];
\node [above left] at (4.5,2.5) {$p$};
\end{tikzpicture}
\caption{Proof of Proposition \ref{angle-betweenness-prop}.}
\end{figure}

\begin{proof}
First, assume that one of the angle measures 
$\angle xoy$, $\angle xoz$ or $\angle xow$ is $0$ or $\pi$. 
Then $\angle xoy=0$ or $\angle xow=\pi$ because 
\eqref{angle-betweeness-xoz-hyp} and 
\eqref{angle-betweeness-xow-hyp} imply 
$\angle xoy\leq\angle xoz\leq\angle xow$. 
Therefore, we have $\angle xow=\angle xoy+\angle yow$ by Corollary \ref{angle-equality-coro}. 
Together with \eqref{angle-betweeness-xoz-hyp} and \eqref{angle-betweeness-xow-hyp}, this implies that 
\begin{equation*}
\angle yow
=
\angle xow-\angle xoy
=
\angle xoz+\angle zow-\angle xoy
=
\angle yoz+\angle zow.
\end{equation*}
So henceforth we assume that 
\begin{equation}\label{angle-betweenness-prop-0pi-assumption}
\{\angle xoy,\angle xoz,\angle xow\}\cap\{ 0,\pi\} =\emptyset .
\end{equation}
Then neither $y$, $z$ nor $w$ is on $\ell (o,x)$. 
We denote by $S$ the side of $\ell (o,x)$ containing $z$. 
By Proposition \ref{angle-equality-necessary-conditions-prop}, 
\eqref{angle-betweeness-xoz-hyp} and \eqref{angle-betweeness-xow-hyp} imply that 
neither $y$ nor $w$ is on the opposite side of $\ell (o,x)$ from $z$, 
and therefore $y\in S$ and $w\in S$. 
In particular, it follows that $d_{\kappa}(y,w)<D_{\kappa}$. 
By Proposition \ref{angle-equality-necessary-conditions-prop}, 
\eqref{angle-betweeness-xoz-hyp} implies 
that $y$ is not on the opposite side of $\ell (o,z)$ from $x$, and 
\eqref{angle-betweeness-xow-hyp} implies 
that $w$ is not on the same side of $\ell (o,z)$ as $x$. 
Because $x\not\in\ell (o,z)$ by \eqref{angle-betweenness-prop-0pi-assumption}, 
it follows that $y$ and $w$ are not on the same side of $\ell (o,z)$. 
Therefore, there exits a point 
$p\in\lbrack y,w\rbrack\cap\ell (o,z)$. 
Because $R_{oz}\setminus\{ o\}=\ell (o,z)\cap S$, and $\lbrack y,w\rbrack\subseteq S$ by the convexity of $S$, 
we have $p\in R_{oz}\setminus\{ o\}$. 
Therefore, 
$\angle yop=\angle yoz$ and $\angle pow=\angle zow$. 
Because $\lbrack y,w\rbrack\cap\lbrack o,p\rbrack\neq\emptyset$, Proposition \ref{diagonal-angle-prop} implies that 
\begin{equation}\label{angle-betweenness-prop-yoz-zow-eq}
\angle yow=\angle yop+\angle pow=\angle yoz+\angle zow.
\end{equation}
Combining \eqref{angle-betweeness-xoz-hyp}, \eqref{angle-betweeness-xow-hyp} and \eqref{angle-betweenness-prop-yoz-zow-eq}, 
we obtain 
\begin{equation*}
\angle xow
=
\angle xoz+\angle zow
=
\angle xoy+\angle yoz+\angle zow
=
\angle xoy+\angle yow,
\end{equation*}
which completes the proof.  
\end{proof}

\begin{proposition}\label{triangle-convexhull-prop}
Let $\kappa\in\mathbb{R}$. 
Suppose $o,x,y,z\in M_{\kappa}^2$ are points such that $0<d_{\kappa}(o,a)<D_{\kappa}$ for every $a\in\{ x,y,z\}$, and 
$d_{\kappa}(y,o)+d_{\kappa}(o,z)<D_{\kappa}$. 
If $\pi\leq\angle yox+\angle xoz$, then 
$d_{\kappa}(y,o)+d_{\kappa}(o,z)\leq
d_{\kappa}(y,x)+d_{\kappa}(x,z)$.
\end{proposition}

\begin{figure}[htbp]
\centering\begin{tikzpicture}[scale=0.5]
\draw (0,0) -- (6,2);
\draw (0,0) -- (2,4);
\draw (2,4) -- (6,0);
\draw (3,1) -- (6,0);
\node [above] at (2,4) {$x$};
\node [left] at (0,0) {$y$};
\node [below] at (3,1) {$o$};
\node [right] at (6,0) {$z$};
\node [right] at (6,2) {$\tilde{z}$};
\end{tikzpicture}
\caption{Proof of Proposition \ref{triangle-convexhull-prop}.}\label{alexandrov-fig}
\end{figure}

\begin{proof}
Let $\tilde{z}\in M_{\kappa}^2$ be the unique point such that 
$d_{\kappa}(o,\tilde{z})=d_{\kappa}(o,z)$ and $\angle yo\tilde{z}=\pi$. 
Then 
\begin{equation*}
\angle yox +\angle xo\tilde{z}=\angle yo\tilde{z}=\pi
\end{equation*} 
by Corollary \ref{angle-equality-coro}. 
Combining this with the hypothesis that $\pi\leq\angle yox+\angle xoz$ yields 
\begin{equation*}
\angle xo\tilde{z}\leq\angle xoz,
\end{equation*}
which implies 
\begin{equation}\label{alexandrov-lemma-xtildez-xz-ineq}
d_{\kappa}(x,\tilde{z})\leq d_{\kappa}(x,z)
\end{equation}
by Lemma \ref{law-of-cos-lemma}. 
Because 
\begin{equation*}
d_{\kappa}(y,o)+d_{\kappa}(o,\tilde{z})=d_{\kappa}(y,o)+d_{\kappa}(o,z)<D_{\kappa}
\end{equation*}
by hypothesis, it follows from the definition of $\tilde{z}$ that $o\in\lbrack y,\tilde{z}\rbrack$, and therefore 
\begin{equation}\label{alexandrov-lemma-o-in-ytildez}
d_{\kappa}(y,o)+d_{\kappa}(o,\tilde{z})=d_{\kappa}(y,\tilde{z}). 
\end{equation}
Using \eqref{alexandrov-lemma-xtildez-xz-ineq}, \eqref{alexandrov-lemma-o-in-ytildez}, the definition of $\tilde{z}$, 
and the triangle inequality for $d_{\kappa}$, we obtain 
\begin{align*}
d_{\kappa}(y,o)+d_{\kappa}(o,z)
&=
d_{\kappa}(y,o)+d_{\kappa}(o,\tilde{z})
=
d_{\kappa}(y,\tilde{z})\\
&\leq
d_{\kappa}(y,x)+d_{\kappa}(x,\tilde{z})
\leq
d_{\kappa}(y,x)+d_{\kappa}(x,z),
\end{align*}
which completes the proof. 
\end{proof}

\begin{proposition}\label{triangle-convexhull-coro-prop}
Let $\kappa\in\mathbb{R}$. 
Suppose $o,x,y,z\in M_{\kappa}^2$ are points such that 
$o\not\in\{ x,y,z\}$, and 
$d_{\kappa}(x,y)+d_{\kappa}(y,o)+d_{\kappa}(o,z)+d_{\kappa}(z,x)<2D_{\kappa}$. 
If $\pi\leq\angle yox+\angle xoz$, 
then we have 
$d_{\kappa}(y,o)+d_{\kappa}(o,z)\leq
d_{\kappa}(y,p)+d_{\kappa}(p,z)$ 
for any $p\in\lbrack o,x\rbrack$. 
\end{proposition}

\begin{proof}
Let $\tilde{z}\in M_{\kappa}^2$ be the point as in the proof of Proposition \ref{triangle-convexhull-prop}. 
Then the inequality \eqref{alexandrov-lemma-xtildez-xz-ineq} in the proof of Proposition \ref{triangle-convexhull-prop} holds 
by the same argument as in the proof of Proposition \ref{triangle-convexhull-prop}. 
By definition of $\tilde{z}$, we have $\tilde{z}\in\overline{R_{oy}}\setminus\{ o\}$. 
It follows that 
\begin{equation}\label{alexandrov-lemma-yo-otildez-ineq}
d_{\kappa}(y,o)+d_{\kappa}(o,\tilde{z})<D_{\kappa}
\end{equation}
because otherwise $\kappa$ would be greater than $0$, and 
$d_{\kappa}(y,o)+d_{\kappa}(o,\tilde{z})+d_{\kappa}(\tilde{z},y)$ would 
be equal to $2D_{\kappa}$, 
which would imply that 
\begin{align*}
d_{\kappa}(x,y)+d_{\kappa}(y,o)+d_{\kappa}(o,z)+d_{\kappa}(z,x)
&\geq
d_{\kappa}(x,y)+d_{\kappa}(y,o)+d_{\kappa}(o,\tilde{z})+d_{\kappa}(\tilde{z},x) \\
&\geq
d_{\kappa}(y,o)+d_{\kappa}(o,\tilde{z})+d_{\kappa}(\tilde{z},y)
=
2 D_{\kappa},
\end{align*}
contradicting the hypothesis. 
It follows from \eqref{alexandrov-lemma-yo-otildez-ineq} and the definition of $\tilde{z}$ that 
\begin{equation}\label{triangle-convexhull-coro-yo-oz-ineq}
d_{\kappa}(y,o)+d_{\kappa}(o,z)<D_{\kappa}.
\end{equation}
If $p=o$, then the desired inequality holds trivially, 
so assume that $p\neq o$. 
Then it follows from the hypothesis and Proposition \ref{hemisphere-prop} that 
\begin{equation}\label{triangle-convexhull-coro-eda-ineq}
0<d_{\kappa}(o,a)<D_{\kappa}
\end{equation}
for every $a\in\{ p,y,z\}$. 
Furthermore, we have 
\begin{equation}\label{triangle-convexhull-coro-yoppoz-ineq}
\angle yop+\angle poz=\angle yox+xoz\geq\pi
\end{equation}
by hypothesis.
It follows from 
\eqref{triangle-convexhull-coro-yo-oz-ineq}, \eqref{triangle-convexhull-coro-eda-ineq}, \eqref{triangle-convexhull-coro-yoppoz-ineq} 
and Proposition \ref{triangle-convexhull-prop} that 
\begin{equation*}
d_{\kappa}(y,o)+d_{\kappa}(o,z)\leq
d_{\kappa}(y,p)+d_{\kappa}(p,z),
\end{equation*}
which completes the proof. 
\end{proof}

\begin{proposition}\label{ray-prop}
Let $\kappa\in\mathbb{R}$. 
Suppose $o,x,y,z\in M_{\kappa}^2$ are points such that $0<d_{\kappa}(o,a)<D_{\kappa}$ for every $a\in\{ x,y,z\}$, and 
$d_{\kappa}(y,o)+d_{\kappa}(o,z)<D_{\kappa}$. 
If $\pi\leq\angle yox+\angle xoz$, and $y$ and $z$ do not lie on the same side of $\ell (o,x)$, 
then $\lbrack y,z\rbrack\cap\overline{R_{ox}}\neq\emptyset$. 
\end{proposition}

\begin{proof}
By hypothesis, we have $d_{\kappa}(y,z)\leq d_{\kappa}(y,o)+d_{\kappa}(o,z)<D_{\kappa}$, and therefore 
there exists a unique geodesic segment $\lbrack y,z\rbrack$ with endpoints $y$ and $z$. 
If both $y$ and $z$ lie on $\ell (o,x)$, then $y\in\overline{R_{ox}}$ or $z\in\overline{R_{ox}}$ because otherwise both $y$ and $z$ would lie in 
$R_{ox}\setminus\{ o\}$, which would imply 
$\angle yox+\angle xoz=0$, contradicting the hypothesis. 
So henceforth we assume that $y\not\in\ell (o,x)$ or $z\not\in\ell (o,x)$. 
Then there exists a unique point $p\in\ell (o,x)\cap\lbrack y,z\rbrack$ because 
$y$ and $z$ are not on the same side of $\ell (o,x)$ by hypothesis. 
The point $p$ satisfies $p\in\overline{R_{ox}}$ because otherwise $p$ would lie in $R_{ox}\setminus\{ o\}$, and therefore  
$\angle yop+\angle poz$ would be equal to $\angle yox +\angle xoz\geq\pi$, which would imply 
$d_{\kappa}(y,o)+d_{\kappa}(o,z)\leq
d_{\kappa}(y,p)+d_{\kappa}(p,z)$ 
by Proposition \ref{triangle-convexhull-prop}, 
contradicting the uniqueness of the geodesic segment with endpoints $y$ and $z$. 
Thus we have $p\in\lbrack y,z\rbrack\cap\overline{R_{ox}}$, which completes the proof. 
\end{proof}

\begin{corollary}\label{ray-coro}
Let $\kappa\in\mathbb{R}$. 
Suppose $o,x,y,z\in M_{\kappa}^2$ are points such that $o\not\in\{ x,y,z\}$, and 
$d_{\kappa}(x,y)+d_{\kappa}(y,o)+d_{\kappa}(o,z)+d_{\kappa}(z,x)<2D_{\kappa}$. 
If $\pi\leq\angle yox+\angle xoz$, and $y$ and $z$ do not lie on the same side of $\ell (o,x)$, 
then 
$\lbrack y,z\rbrack\cap\overline{R_{ox}}\neq\emptyset$. 
\end{corollary}

\begin{proof}
By the same argument as in the proof of Proposition \ref{triangle-convexhull-coro-prop}, 
the hypothesis implies that 
$0<d_{\kappa}(o,a)<D_{\kappa}$ for every $a\in\{ x,y,z\}$, and 
$d_{\kappa}(y,o)+d_{\kappa}(o,z)<D_{\kappa}$. 
Therefore, we have $\lbrack y,z\rbrack\cap\overline{R_{ox}}\neq\emptyset$ by Proposition \ref{ray-prop}. 
\end{proof}

\begin{proposition}\label{convex-quadrilateral-prop}
Let $\kappa\in\mathbb{R}$. 
Suppose $x,y,z,w\in M_{\kappa}^2$ are points such that 
\begin{align*}
&d_{\kappa}(x,y)+d_{\kappa}(y,z)+d_{\kappa}(z,w)+d_{\kappa}(w,x)<2D_{\kappa},\\
&x\neq y,\quad x\neq z,\quad x\neq w,\quad y\neq z, \quad z\neq w.
\end{align*} 
Then $\lbrack x,z\rbrack\cap\lbrack y,w\rbrack\neq\emptyset$ if and only if 
$\angle yxz+\angle zxw\leq\pi$, $\angle yzx+\angle xzw\leq\pi$, and 
$y$ and $w$ are not on the same side of $\ell (x,z)$. 
\end{proposition}

\begin{figure}[htbp]
\centering\begin{tikzpicture}[scale=0.5]
\draw[dashed] (3,-1.5) -- (3,4.5);
\draw (0,1) -- (3,3);
\draw (3,3) -- (5,1.67);
\draw (5,1.67) -- (3,0);
\draw (3,0) -- (0,1);
\node [above left] at (3,3) {$x$};
\node [left] at (0,1) {$y$};
\node [below left] at (3,0) {$z$};
\node [right] at (5,1.67) {$w$};
\end{tikzpicture}
\caption{Proof of Proposition \ref{convex-quadrilateral-prop}.}\label{convex-quadrilateral-fig}
\end{figure}

\begin{proof}
It follows from the hypothesis and Proposition \ref{hemisphere-prop} that 
$\{ x,y,z,w\}$ is contained in a side $S$ of some line, and that 
we have 
\begin{equation*}
d_{\kappa}(y,w)<D_{\kappa},\quad
0<d_{\kappa}(x,a)<D_{\kappa},\quad
0<d_{\kappa}(z,b)<D_{\kappa}
\end{equation*}
for any $a\in\{ y,z,w\}$ and any $b\in\{ x,y,w\}$. 
First assume that $\lbrack x,z\rbrack\cap\lbrack y,w\rbrack\neq\emptyset$. 
Then 
$\angle yxz+\angle zxw\leq\pi$ and 
$\angle yzx+\angle xzw\leq\pi$ 
by Corollary \ref{diagonal-angle-coro}. 
Furthermore, $y$ and $w$ are not on the same side of $\ell (x,z)$ because otherwise 
$\lbrack y,w\rbrack$ would be contained in one side of $\ell (x,z)$ by the convexity of the side, 
which would imply $\ell (x,z)\cap\lbrack y,w\rbrack =\emptyset$, contradicting the assumption that $\lbrack x,z\rbrack\cap\lbrack y,w\rbrack\neq\emptyset$.

Conversely, assume that $\angle yxz+\angle zxw\leq\pi$, $\angle yzx+\angle xzw\leq\pi$, and 
$y$ and $w$ are not on the same side of $\ell (x,z)$. 
We consider two cases. 

\textsc{Case 1}: 
{\em $y\not\in\ell (x,z)$ or $w\not\in\ell (x,z)$.} 
In this case, there exists a unique point $p$ that lies in 
$\lbrack y,w\rbrack\cap\ell (x,z)$, and we have 
\begin{align}\label{convex-quadrilateral-perimeter-ineq}
d_{\kappa}(x,y)+d_{\kappa}(y,p)+d_{\kappa}(p,&w)+d_{\kappa}(w,x)
=
d_{\kappa}(x,y)+d_{\kappa}(y,w)+d_{\kappa}(w,x)\\
&\leq
d_{\kappa}(x,y)+d_{\kappa}(y,z)+d_{\kappa}(z,w)+d_{\kappa}(w,x)<2D_{\kappa}\nonumber
\end{align}
by hypothesis. 
Since $\lbrack x,z\rbrack =R_{xz}\cap R_{zx}$, 
it suffices to prove that $p\in R_{xz}$ and $p\in R_{zx}$. 
Suppose for the sake of contradiction that $p\not\in R_{xz}$. 
Then $p\in\overline{R_{xz}}\setminus\{ x\}$ and $\angle pxz=\pi$. 
Therefore, we have 
$\angle pxy+\angle yxz=\angle pxz=\pi$ and 
$\angle pxw+\angle wxz=\angle pxz=\pi$ 
by Corollary \ref{angle-equality-coro}. 
Combining these with the assumption that $\angle yxz+\angle zxw\leq\pi$, we obtain 
\begin{equation}\label{convex-quadrilateral-yxp-pxw}
\angle yxp+\angle pxw=(\pi -\angle yxz)+(\pi -\angle wxz)\geq\pi .
\end{equation}
Because 
$y$ and $w$ are not on the same side of the line $\ell (x,p)=\ell (x,z)$, 
\eqref{convex-quadrilateral-perimeter-ineq} and \eqref{convex-quadrilateral-yxp-pxw} imply 
$\lbrack y,w\rbrack\cap\overline{R_{xp}}\neq\emptyset$ by Corollary \ref{ray-coro}. 
Because $\overline{R_{xp}}=R_{xz}$ and $\lbrack y,w\rbrack\cap\ell (x,z)=\{ p\}$, 
it follows that $p\in R_{xz}$, contradicting the assumption that $p\not\in R_{xz}$. 
Thus we have $p\in R_{xz}$. 
Exactly the same argument shows that $p\in R_{zx}$. 

\textsc{Case 2}: 
{\em $y\in\ell (x,z)$ and $w\in\ell (x,z)$.} 
If $y\in\lbrack x,z\rbrack$, 
then $\lbrack x,z\rbrack\cap\lbrack y,w\rbrack\neq\emptyset$ clearly. 
So we assume that $y\not\in\lbrack x,z\rbrack$. 
Then $y\not\in R_{xz}$ or $y\not\in R_{zx}$. 
We may assume without loss of generality that $y\not\in R_{xz}$. 
Then $w\in R_{xz}$ because otherwise 
both $\angle yxz$ and $\angle zxw$ would be equal to $\pi$, contradicting the assumption that 
$\angle yxz+\angle zxw\leq \pi$. 
Thus we have 
\begin{equation}\label{convex-quadrilateral-yw-Rxz}
y\in\overline{R_{xz}},\quad
w\in R_{xz}.
\end{equation}
If $\kappa\leq 0$, then \eqref{convex-quadrilateral-yw-Rxz} clearly 
implies $x\in\lbrack y,w\rbrack$, and thus $\lbrack x,z\rbrack\cap\lbrack y,w\rbrack\neq\emptyset$. 
If $\kappa >0$, 
then $\lbrack y,w\rbrack$ contains either $x$ or the antipode of $x$ by \eqref{convex-quadrilateral-yw-Rxz}. 
On the other hand, 
the antipode of $x$ cannot lie in 
$\lbrack y,w\rbrack$ because $\lbrack y,w\rbrack\subseteq S$ by the convexity of $S$. 
Therefore, we have $x\in\lbrack y,w\rbrack$ in the case in which $\kappa >0$ as well, 
which completes the proof.  
\end{proof}

\begin{proposition}\label{plane-angle-prop}
Let $\kappa\in\mathbb{R}$. 
Suppose $o,x,y,z \in M_{\kappa}^2$ are points such that 
$0<d_{\kappa}(o,a)<D_{\kappa}$ 
for every $a\in\{ x,y,z\}$. 
If $\pi\leq\angle yox+\angle xoz$, and $y$ and $z$ do not lie on the same side of $\ell (o,x)$, then 
$\angle xoy+\angle yoz+\angle zox=2\pi$. 
\end{proposition}

\begin{figure}[htbp]
\centering\begin{tikzpicture}[scale=0.4]
\draw (2,5) -- (6,4.5);
\draw (0,4) -- (4,6);
\draw (4,6) -- (8,3);
\draw (4,6) -- (5,10);
\draw[dashed] (4,6) -- (2.5,0);
\fill (0,4) circle [radius=0.15];
\fill (4,6) circle [radius=0.15];
\fill (5,10) circle [radius=0.15];
\fill (8,3) circle [radius=0.15];
\fill (2,5) circle [radius=0.15];
\fill (3,2) circle [radius=0.15];
\fill (3,2) circle [radius=0.15];
\fill (6,4.5) circle [radius=0.15];
\fill (3.697,4.785) circle [radius=0.15];
\node [above] at (5,10) {$x$};
\node [left] at (0,4) {$y$};
\node [right] at (8,3) {$z$};
\node [above left] at (4,6) {$o$};
\node [above left] at (2,5) {$y'$};
\node [right] at (3,2) {$x'$};
\node [above right] at (6,4.5) {$z'$};
\node [below right] at (3.697,4.785) {$p$};
\end{tikzpicture}
\caption{Proof of Proposition \ref{plane-angle-prop}.}\label{plane-angle-fig}
\end{figure}
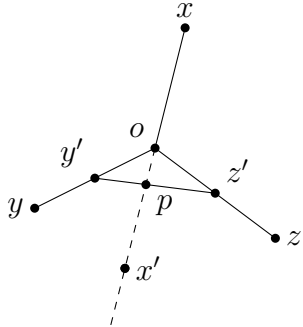

\begin{proof}
Choose a point $x'\in\overline{R_{ox}}\setminus\{ o\}$. 
Choose two points $y'\in (o,y\rbrack$ and $z'\in (o,z\rbrack$ such that 
$d_{\kappa}(y' ,o)+d_{\kappa}(o,z')<D_{\kappa}$. 
If $o\in\lbrack y',z' \rbrack$, then $\angle zoy=\angle z'oy' =\pi$, and therefore it follows from Corollary \ref{angle-equality-coro} that 
\begin{equation*}
\angle xoy+\angle yoz+\angle zox=\pi +(\angle zox+\angle xoy)=\pi +\angle zoy=2\pi .
\end{equation*}
So henceforth we assume that $o\not\in\lbrack y' ,z' \rbrack$. 
The hypothesis that $y$ and $z$ do not lie on the same side of $\ell (o,x)$ implies that 
$y'$ and $z'$ do not lie on the same side of $\ell (o,x)$ because 
the union of $\ell (o,x)$ and any one of its sides is a $D_{\kappa}$-convex subset of $M_{\kappa}^2$. 
We also have 
$\pi\leq\angle yox+\angle xoz=\angle y'ox+\angle xoz'$ 
by hypothesis. 
Therefore, there exists a point $p\in(\overline{R_{ox}}\setminus\{ o\} )\cap\lbrack y',z'\rbrack$ by Proposition \ref{ray-prop}. 
Since $\lbrack y',z'\rbrack\cap\lbrack o,p\rbrack\neq\emptyset$, 
we have $\angle y'oz' =\angle y'op+\angle poz'$ by 
Proposition \ref{diagonal-angle-prop}, and therefore 
\begin{equation}\label{plane-angle-ydashozdash}
\angle yoz
=\angle y'oz' =\angle y'op+\angle poz'
=\angle yox' +\angle x'oz.
\end{equation}
On the other hand, since $\angle xox'=\pi$, we have 
\begin{equation}\label{plane-angle-xoxdash}
\angle xoy+\angle yox'=\angle xox' =\pi ,\quad
\angle x'oz +\angle zox=\angle x'ox =\pi
\end{equation}
by Corollary \ref{angle-equality-coro}. 
Combining \eqref{plane-angle-ydashozdash} and \eqref{plane-angle-xoxdash}, we obtain 
\begin{equation*}
\angle xoy+\angle yoz+\angle zox
=
\angle xoy +\angle yox' +\angle x'oz +\angle zox=2\pi ,
\end{equation*}
which completes the proof. 
\end{proof}

Proposition \ref{plane-angle-prop} implies the following corollary. 

\begin{corollary}\label{plane-angle-coro}
Let $\kappa\in\mathbb{R}$. 
Suppose $o,x,y,z \in M_{\kappa}^2$ are points such that 
$0<d_{\kappa}(o,a)<D_{\kappa}$ 
for every $a\in\{ x,y,z\}$. 
Then the following are equivalent: 
\begin{enumerate}
\item[$(1)$]
$\pi\leq\angle yox+\angle xoz$, and 
$y$ and $z$ do not lie on the same side of $\ell (o,x)$. 
\item[$(2)$]
$\pi\leq\angle zoy+\angle yox$, and 
$z$ and $x$ do not lie on the same side of $\ell (o,y)$. 
\item[$(3)$]
$\pi\leq\angle xoz+\angle zoy$, and 
$x$ and $y$ do not lie on the same side of $\ell (o,z)$.
\end{enumerate}
\end{corollary}

\begin{proof}
It suffices to prove that $(1)$ implies $(2)$. 
Assume that $(1)$ is true. 
Then Proposition \ref{plane-angle-prop} implies that 
\begin{equation}\label{convexhull-separation-2pi}
\angle xoy+\angle yoz+\angle zox=2\pi ,
\end{equation}
and therefore 
\begin{equation*}
\angle zoy+\angle yox =2\pi -\angle zox\geq \pi .
\end{equation*}
Assume for the sake of contradiction that $z$ and $x$ lie on the same side of $\ell (o,y)$. 
Then neither $z$ nor $x$ is on $\ell (o,y)$, and therefore $\angle yoz<\pi$ and $\angle yox<\pi$. 
If $\angle yoz\leq\angle yox$, then $\angle yoz +\angle zox =\angle yox$ by Proposition \ref{angle-equality-prop}, and therefore 
\begin{equation*}
\angle xoy+\angle yoz+\angle zox
=
2\angle xoy
<
2\pi .
\end{equation*}
Similarly, if $\angle yox\leq\angle yoz$, then 
$\angle yox +\angle xoz =\angle yoz$ by Proposition \ref{angle-equality-prop}, and therefore 
\begin{equation*}
\angle xoy+\angle yoz+\angle zox
=
2\angle yoz
<
2\pi .
\end{equation*}
Thus we always have 
$\angle xoy+\angle yoz+\angle zox <2\pi$, 
contradicting \eqref{convexhull-separation-2pi}. 
Thus $(2)$ is true. 
\end{proof}

\begin{proposition}\label{interior-of-triangle-prop}
Let $\kappa\in\mathbb{R}$. 
Suppose $o,x,y,z \in M_{\kappa}^2$ 
are four distinct points such that $\{ o,x,y,z\}$ is contained in a side of some line. 
Assume that $\pi\leq\angle yox+\angle xoz$, and that 
$y$ and $z$ do not lie on the same side of $\ell (o,x)$. 
Then all of the following conditions are true: 
\begin{enumerate}
\item[$(1)$]
$o$ and $x$ do not lie on opposite sides of $\ell (y,z)$. 
\item[$(2)$]
$o$ and $y$ do not lie on opposite sides of $\ell (z,x)$. 
\item[$(3)$]
$o$ and $z$ do not lie on opposite sides of $\ell (x,y)$. 
\end{enumerate}
\end{proposition}

\begin{proof}
By Corollary \ref{plane-angle-coro}, it suffices to prove that $(1)$ is true. 
If $o\in\ell (y,z)$, then 
$o$ and $x$ do not lie on opposite sides of $\ell (y,z)$ clearly, 
so assume that $o\not\in\ell (y,z)$. 
Then we have $\angle yoz<\pi$ and $\ell (o,x)\neq\ell (y,z)$. 
Because $y$ and $z$ do not lie on the same side of $\ell (o,x)$ by hypothesis, 
there exists a point $p\in\ell (o,x)\cap\lbrack y,z\rbrack$. 
Suppose for the sake of contradiction that 
$o$ and $x$ lie on opposite sides of $\ell (y,z)$. 
Then there exists a point $q\in\lbrack o,x\rbrack\cap\ell (y,z)$. 
By hypothesis, $\{ o,x,y,z\}$ is contained in a side $S$ of some line, and both $p$ and $q$ 
lie in $S$ by the convexity of $S$. 
It follows that $p=q$ because otherwise 
the two distinct lines $\ell (o,x)$ and $\ell (y,z)$ would share the two distinct points $p$ and $q$, and 
therefore $\kappa$ would be greater than $0$, and $q$ would be the antipode of $p$, 
contradicting the fact that $\{ p,q\}\subseteq S$. 
Therefore, we have $\lbrack o,x\rbrack\cap\lbrack y,z\rbrack\neq\emptyset$, 
which implies 
$\angle yox+\angle xoz=\angle yoz<\pi$ 
by Proposition \ref{diagonal-angle-prop}, 
contradicting the hypothesis. 
\end{proof}

\begin{proposition}\label{triangle-boundary-prop}
Let $\kappa\in\mathbb{R}$. 
Suppose $o,x,y,z\in M_{\kappa}^2$ are four distinct points such that 
$\{ o,x,y,z\}$ is contained in a side of some line. 
Assume that $\pi\leq\angle yox+\angle xoz$, and 
that $y$ and $z$ do not lie on the same side of $\ell (o,x)$. 
If $o\in\ell (x,y)\cup\ell (y,z)\cup\ell (z,x)$, then $o\in\lbrack x,y\rbrack\cup\lbrack y,z\rbrack\cup\lbrack z,x\rbrack$. 
\end{proposition}

\begin{proof}
By hypothesis, $\{ o,x,y,z\}$ is contained in a side $S$ of some line. 
By Corollary \ref{plane-angle-coro}, we may assume without loss of generality that $o\in\ell (x,y)$. 
Because $\lbrack x,o\rbrack\cup\lbrack o,y\rbrack\subseteq S$ by the convexity of $S$, 
if $y\in\overline{R_{ox}}$, then we have $d_{\kappa}(x,o)+d_{\kappa}(o,y)\leq D_{\kappa}$, and thus $o\in\lbrack x,y\rbrack$. 
So assume that $y\not\in\overline{R_{ox}}$. 
Then $y\in R_{ox}\setminus\{ o\}$, and thus $\angle yox=0$. 
Combining this with the hypothesis that $\pi\leq\angle yox+\angle xoz$, we obtain 
$\angle xoz=\pi$, and therefore $z\in\overline{R_{ox}}$. 
Because $\lbrack x,o\rbrack\cup\lbrack o,z\rbrack\subseteq S$ by the convexity of $S$, 
it follows that $d_{\kappa}(x,o)+d_{\kappa}(o,z)<D_{\kappa}$, and thus $o\in\lbrack z,x\rbrack$, which completes the proof. 
\end{proof}

The following proposition follows from Proposition \ref{interior-of-triangle-prop} and Proposition \ref{triangle-boundary-prop}. 

\begin{proposition}\label{interior-of-triangle-angle-prop}
Let $\kappa\in\mathbb{R}$. 
Suppose $o,x,y,z \in M_{\kappa}^2$ 
are four distinct points such that $\{ o,x,y,z\}$ is contained in a side of some line. 
Assume that $\pi\leq\angle yox+\angle xoz$, and that 
$y$ and $z$ do not lie on the same side of $\ell (o,x)$. 
Then 
\begin{equation*}
\angle zxo+\angle oxy=\angle zxy,\quad\angle xyo+\angle oyz=\angle xyz ,\quad\angle yzo+\angle ozx=\angle yzx.
\end{equation*}
\end{proposition}

\begin{proof}
By Corollary \ref{plane-angle-coro}, it suffices to prove that $\angle zxo+\angle oxy=\angle zxy$. 
It follows from the hypothesis and Proposition \ref{interior-of-triangle-prop} that 
$o$ is neither on the opposite side of $\ell (x,z)$ from $y$ nor 
on the opposite side of $\ell (x,y)$ from $z$. 
Therefore, Proposition \ref{angle-equality-prop} implies that 
$\angle zxo+\angle oxy=\angle zxy$ whenever $0<\angle zxy$. 
So assume that $\angle zxy=0$. 
Then $z\in R_{xy}\setminus\{ x\}$. 
It follows that $o\in\ell (x,y)$ because otherwise $R_{xy}\setminus\{ x\}$ would coincide with the 
union of $\ell (x,y)$ and the side of $\ell (o,x)$ containing $y$, 
and therefore $y$ and $z$ would lie on the same side of $\ell (o,x)$, contradicting the hypothesis. 
Therefore, we have $o\in\lbrack x,y\rbrack\cup\lbrack y,z\rbrack\cup\lbrack z,x\rbrack$ by Proposition \ref{triangle-boundary-prop}. 
Since $x$, $y$ and $z$ all lie in $R_{xy}$, this implies that $o\in R_{xy}$. 
It follows that $\angle zxo=\angle oxy=0$, and therefore 
$\angle zxo+\angle oxy=\angle zxy$, which completes the proof. 
\end{proof}

\section{A generalization of Alexandrov's lemma}\label{alexandrov-sec}

In this section, we prove two lemmas, which generalize Alexandrov's lemma \cite[p.25]{BH}. 
The following lemma was proved for $\kappa =0$ in \cite[Lemma 8.4]{toyoda-five}. 

\begin{lemma}\label{larger-larger-lemma}
Let $\kappa\in\mathbb{R}$. 
Suppose $x,y,z,w,x',y',z',w'\in M_{\kappa}^2$ are points such that 
\begin{align*}
&d_{\kappa}(x,y)+d_{\kappa}(y,z)+d_{\kappa}(z,w)+d_{\kappa}(w,x)<2D_{\kappa},\\
&d_{\kappa} (x,y)=d_{\kappa}(x',y'),\quad d_{\kappa} (y,z)=d_{\kappa} (y',z'),\quad d_{\kappa}(z,w)= d_{\kappa}(z',w'),\\
&d_{\kappa}(w,x)=d_{\kappa}(w',x'),\quad d_{\kappa}(y,w)\leq d_{\kappa}(y',w'). 
\end{align*}
Whenever $x\neq y$, $x\neq z$, $x\neq w$, $y\neq z$ and $z\neq w$, 
assume that $\pi\leq\angle yzx+\angle xzw$, and that 
$y$ and $w$ do not lie on the same side of $\ell (x,z)$. 
Then 
\begin{equation*}
d_{\kappa} (x,z)\leq d_{\kappa}(x',z').
\end{equation*} 
\end{lemma}

\begin{figure}[htbp]
\centering\begin{tikzpicture}[scale=0.5]
\draw (0,0) -- (2,2);
\draw (0,0) -- (2,5);
\draw (2,5) -- (6,0);
\draw (2,2) -- (6,0);
\draw (2,2) -- (6.47,1.8);
\node [below left] at (0,0) {$y$};
\node [below right] at (6,0) {$w$};
\node [above] at (2,2) {$z$};
\node [above] at (2,5) {$x$};
\node [above right] at (6.47,1.8) {$\tilde{w}$};
\draw (12,0) -- (14.46,4.79);
\draw (14.46,4.79) -- (18.71,0);
\draw (12,0) -- (14.46,1.39);
\draw (14.46,1.39) -- (18.71,0);
\node [below left] at (12,0) {$y'$};
\node [below right] at (18.71,0) {$w'$};
\node [above] at (14.46,1.39) {$z'$};
\node [above] at (14.46,4.79) {$x'$};
\end{tikzpicture}
\caption{Proof of Lemma \ref{larger-larger-lemma}.}\label{larger-fig}
\end{figure}

\begin{proof}
We first consider the case in which $x$, $y$, $z$ and $w$ are not distinct. 
If $x=y$, then $x'=y'$ since $d_{\kappa}(x,y)=d_{\kappa}(x',y')$, and therefore 
\begin{equation*}
d_{\kappa}(x,z)=d_{\kappa}(y,z)=d_{\kappa}(y',z')=d_{\kappa}(x',z').
\end{equation*}
Similarly, if one of the equalities $x=w$, $y=z$ or $z=w$ holds, then we have 
\begin{equation*}
d_{\kappa}(x,z)=d_{\kappa}(x',z').
\end{equation*} 
If $x=z$, then the inequality $d_{\kappa}(x,z)\leq d_{\kappa}(x',z')$ holds clearly. 
So we may assume that 
\begin{equation*}
x\neq y,\quad x\neq z,\quad x\neq w,\quad y\neq z,\quad z\neq w.
\end{equation*}
Suppose that $y=w$. 
Then $y\in\ell (x,z)$ because $y$ and $w$ do not lie on the same side of $\ell (x,z)$ by hypothesis, and therefore we have 
$\angle xzy=\angle xzw=0$ or $\angle xzy=\angle xzw=\pi$. 
Because $\pi\leq\angle yzx+\angle xzw$ by hypothesis, we have $\angle xzy=\angle xzw=\pi$, and thus 
\begin{equation}\label{larger-y-in-ray-zx}
y\in\overline{R_{zx}}\backslash\{ z\} .
\end{equation} 
It follows that 
\begin{equation}\label{larger-xz-zy-diameter}
d_{\kappa}(x,z)+d_{\kappa}(z,y)<D_{\kappa}
\end{equation}
because otherwise $\kappa$ would be greater than $0$, and 
$d_{\kappa}(x,z)+d_{\kappa}(z,y)+d_{\kappa}(y,x)$ would be equal to $2D_{\kappa}$, 
which would imply that 
\begin{equation*}
d_{\kappa}(x,y)+d_{\kappa}(y,z)+d_{\kappa}(z,w)+d_{\kappa}(w,x)
\geq
d_{\kappa}(x,y)+d_{\kappa}(y,z)+d_{\kappa}(z,x)
=2D_{\kappa},
\end{equation*}
contradicting the hypothesis. 
It follows from \eqref{larger-y-in-ray-zx} and \eqref{larger-xz-zy-diameter}
that $z\in\lbrack x,y\rbrack$, and thus 
\begin{equation*}
d_{\kappa}(x,z)=d_{\kappa}(x,y)-d_{\kappa}(y,z)
=d_{\kappa}(x',y')-d_{\kappa}(y',z')\leq d_{\kappa}(x',z').
\end{equation*}
So henceforth we assume that $x$, $y$, $z$ and $w$ are distinct. 
We consider two cases. 

\textsc{Case 1}: 
{\em $z\not\in\ell (x,y)\cup\ell (y,w)\cup\ell (w,x)$.} 
Let $\tilde{w}\in M_{\kappa}^2$ be the point such that 
\begin{equation*}
d_{\kappa}(z,\tilde{w})=d_{\kappa}(z,w),\quad\angle yz\tilde{w} =\angle y'z'w' ,
\end{equation*}
and $\tilde{w}$ is not on the opposite side of $\ell (y,z)$ from $w$, 
as shown in \textsc{Figure} \ref{larger-fig}. 
Then 
\begin{equation}\label{larger-yzw-cong}
d_{\kappa}(y,\tilde{w})=d_{\kappa}(y',w'),\quad 
\angle \tilde{w}yz=\angle w'y'z'
\end{equation}
because the triangle with vertices $y$, $z$ and $\tilde{w}$ is congruent to 
that with vertices $y'$, $z'$ and $w'$. 
Since $d_{\kappa}(y,w)\leq d_{\kappa}(y',w')$ by hypothesis, 
Lemma \ref{law-of-cos-lemma} implies that 
\begin{equation*}
\angle yzw\leq\angle y'z'w' =\angle yz\tilde{w}, 
\end{equation*}
and therefore Proposition \ref{angle-equality-prop} implies that 
\begin{equation}\label{larger-yzw-wztildew-eq}
\angle yzw +\angle wz\tilde{w}
=
\angle yz\tilde{w}
\leq\pi .
\end{equation}
Because $\pi\leq\angle yzx+\angle xzw$, 
and $y$ and $w$ are not on the same side of $\ell (x,z)$, Corollary \ref{plane-angle-coro} implies that 
$\pi\leq\angle yzw+\angle wzx$. 
Combining this with \eqref{larger-yzw-wztildew-eq} 
yields $\angle wz\tilde{w}\leq\angle wzx$. 
Furthermore, $\tilde{w}$ and $x$ are not on opposite sides of $\ell (z,w)$ 
because 
\eqref{larger-yzw-wztildew-eq} implies that $\tilde{w}$ is not on the same side of $\ell (z,w)$ as $y$ by Proposition \ref{angle-equality-necessary-conditions-prop}, 
the hypothesis implies that $x$ is not on the same side of $\ell (z,w)$ as $y$ by Corollary \ref{plane-angle-coro}, and 
$y\not\in\ell (z,w)$ by the assumption of \textsc{Case 1}. 
Therefore, Proposition \ref{angle-equality-prop} implies that 
\begin{equation}\label{larger-wztildew-tildewzx-eq}
\angle wzx
=
\angle wz\tilde{w}+\angle\tilde{w}zx.
\end{equation}
We have $\angle\tilde{w}zx\leq\angle wzx$ by 
\eqref{larger-wztildew-tildewzx-eq}, and therefore 
Lemma \ref{law-of-cos-lemma} implies that 
\begin{equation*}
d_{\kappa}(\tilde{w},x)\leq d_{\kappa}(w,x)=d_{\kappa}(w' ,x' ).
\end{equation*}
Using Lemma \ref{law-of-cos-lemma} again, this implies that 
\begin{equation}\label{larger-yxz}
\angle \tilde{w}yx\leq\angle w' y' x' .
\end{equation}
Because the hypothesis implies $\angle yzw +\angle wzx +\angle xzy=2\pi$ by Proposition \ref{plane-angle-prop}, 
it follows from \eqref{larger-yzw-wztildew-eq} and \eqref{larger-wztildew-tildewzx-eq} that 
\begin{equation*}
\angle \tilde{w}zx +\angle xzy=2\pi-\angle yz\tilde{w}\geq\pi .
\end{equation*}
Furthermore, $y$ and $\tilde{w}$ are not on the same side of $\ell (x,z)$ 
because 
$y$ is not on the same side of $\ell (x,z)$ as $w$ by hypothesis, 
\eqref{larger-wztildew-tildewzx-eq} implies that $\tilde{w}$ is not on the opposite side of $\ell (x,z)$ from $w$ 
by Proposition \ref{angle-equality-necessary-conditions-prop}, 
and $w\not\in\ell (x,z)$ by the assumption of \textsc{Case 1}. 
Therefore, Proposition \ref{interior-of-triangle-angle-prop} implies that 
\begin{equation}\label{larger-tildewyz-zyx}
\angle \tilde{w}yx=\angle\tilde{w}yz+\angle zyx.
\end{equation}
We also have 
\begin{equation}\label{larger-wyz-zyx}
\angle w'y'x'
\leq
\angle w' y' z' +\angle z'y'x'
\end{equation}
by Proposition \ref{angle-triangle-inequality-prop}. 
Combining \eqref{larger-yzw-cong}, \eqref{larger-yxz}, \eqref{larger-tildewyz-zyx} and \eqref{larger-wyz-zyx}, we obtain
\begin{equation*}
\angle zyx
=
\angle \tilde{w}yx-\angle\tilde{w}yz
\leq
\angle w' y' x'-\angle\tilde{w}yz
=
\angle w' y' x' -\angle w' y' z'
\leq
\angle z'y'x' ,
\end{equation*}
and therefore Lemma \ref{law-of-cos-lemma} implies that $d_{\kappa}(x,z)\leq d_{\kappa}(x',z' )$.

\textsc{Case 2}: 
{\em $z\in\ell (x,y)\cup\ell (y,w)\cup\ell (w,x)$.} 
In this case, $z\in\lbrack x,y\rbrack\cup\lbrack y,w\rbrack\cup\lbrack w,x\rbrack$ by Proposition \ref{triangle-boundary-prop}. 
If $z\in\lbrack x,y\rbrack$, then 
\begin{equation*}
d_{\kappa}(x,z)
=
d_{\kappa}(x,y)-d_{\kappa}(y,z)
=
d_{\kappa}(x',y')-d_{\kappa}(y',z')
\leq
d_{\kappa}(x',z'). 
\end{equation*}
If $z\in\lbrack w,x\rbrack$, then we obtain $d_{\kappa}(x,z)\leq d_{\kappa}(x',z')$ similarly. 
So assume that $z\in\lbrack y,w\rbrack$. 
Then 
\begin{equation*}
d_{\kappa}(y,w)
\leq
d_{\kappa}(y',w')
\leq
d_{\kappa}(y',z')+d_{\kappa}(z',w')
=
d_{\kappa}(y,z)+d_{\kappa}(z,w)
=
d_{\kappa}(y,w), 
\end{equation*}
and thus $d_{\kappa}(y,w)=d_{\kappa}(y',w')$. 
It follows that the triangle with vertices $x$, $y$ and $w$ is congruent to 
the triangle with vertices $x'$, $y'$ and $w'$, and that 
$z\in\lbrack y,w\rbrack$ corresponds to $z'\in\lbrack y',w'\rbrack$ under this congruence, 
which implies $d_{\kappa}(x,z)=d_{\kappa}(x',z')$. 
\end{proof}

\begin{remark}
Suppose $x,y,z,w,x',y',z',w' \in M_{\kappa}^2$ are points that satisfy the hypothesis of Lemma \ref{larger-larger-lemma}. 
Alexandrov's lemma \cite[p.25]{BH} states that if the identity 
\begin{equation*}
d_{\kappa}(y',w')=d_{\kappa}(y',z')+d_{\kappa}(z',w')
\end{equation*}
holds in addition, 
then we have $d_{\kappa}(x,z)\leq d_{\kappa}(x',z')$. 
\end{remark}

Next we prove the following lemma, which looks similar to the previous one.

\begin{lemma}\label{larger-smaller-lemma}
Let $\kappa\in\mathbb{R}$. 
Suppose $x,y,z,w,x',y',z',w'\in M_{\kappa}^2$ are points such that 
\begin{align*}
&d_{\kappa}(x',y')+d_{\kappa}(y',z')+d_{\kappa}(z',w')+d_{\kappa}(w',x')<2D_{\kappa},\\
&d_{\kappa} (x,y)=d_{\kappa}(x',y'),\quad d_{\kappa} (y,z)=d_{\kappa} (y',z'),\quad d_{\kappa}(z,w)= d_{\kappa}(z',w'),\\
&d_{\kappa}(w,x)=d_{\kappa}(w',x'),\quad d_{\kappa}(x',z')\leq d_{\kappa}(x,z). 
\end{align*}
If $\lbrack x',z'\rbrack\cap\lbrack y',w' \rbrack\neq\emptyset$, then 
\begin{equation*}
d_{\kappa} (y,w)\leq d_{\kappa}(y',w').
\end{equation*}
\end{lemma}

\begin{figure}[htbp]
\centering\begin{tikzpicture}[scale=0.5]
\draw (0,0) -- (-1.6,-2.33);
\draw (-1.6,-2.33) -- (0,-6.5);
\draw (0,-6.5) -- (4.86,-2.33);
\draw (4.86,-2.33) -- (0,0);
\draw (12,0) -- (10,-2);
\draw (10,-2) -- (12,-6);
\draw (12,-6) -- (17,-2);
\draw (17,-2) -- (12,0);
\draw (10,-2) -- (10.83,-6.39);
\draw (0,0) -- (0,-6.5);
\draw (12,0) -- (12,-6);
\node [above] at (0,0) {$x$};
\node [left] at (-1.6,-2.33) {$y$};
\node [below] at (0,-6.5) {$z$};
\node [right] at (4.86,-2.33) {$w$};
\node [above] at (12,0) {$x'$};
\node [left] at (10,-2) {$y'$};
\node [below] at (12,-6) {$z'$};
\node [right] at (17,-2) {$w'$};
\node [below left] at (10.83,-6.39) {$\tilde{z}$};
\end{tikzpicture}
\caption{Proof of Lemma \ref{larger-smaller-lemma}.}\label{larger-smaller-fig}
\end{figure}

\begin{proof}
We first consider the case in which $x'$, $y'$, $z'$ and $w'$ are not distinct. 
If $x'=y'$, then $x=y$ since $d_{\kappa}(x,y)=d_{\kappa}(x',y')$, and therefore 
\begin{equation*}
d_{\kappa}(y,w)=d_{\kappa}(x,w)=d_{\kappa}(x',w')=d_{\kappa}(y',w').
\end{equation*}
Similarly, if one of the equalities $y'=z'$, $x'=w'$ or $z'=w'$ holds, then 
we have 
\begin{equation*}
d_{\kappa}(y,w)=d_{\kappa}(y',w').
\end{equation*}
If $x'=z'$, then $x'\in\lbrack y',w'\rbrack$ because 
$\lbrack x',z'\rbrack\cap\lbrack y',w' \rbrack\neq\emptyset$ by hypothesis, and therefore 
\begin{equation*}
d_{\kappa}(y,w)
\leq
d_{\kappa}(y,x)+d_{\kappa}(x,w)
=
d_{\kappa}(y',x')+d_{\kappa}(x',w')
=
d_{\kappa}(y',w'). 
\end{equation*}
Suppose that $y'=w'$. 
Then $y'\in\lbrack x',z'\rbrack$ and $w'\in\lbrack x',z'\rbrack$ because 
$\lbrack x',z'\rbrack\cap\lbrack y',w' \rbrack\neq\emptyset$. 
It follows that 
\begin{align*}
d_{\kappa}(x,z)
&\leq
d_{\kappa}(x,y)+d_{\kappa}(y,z)
=
d_{\kappa}(x',y')+d_{\kappa}(y',z')
=
d_{\kappa}(x',z')
\leq
d_{\kappa}(x,z),\\
d_{\kappa}(x,z)
&\leq
d_{\kappa}(x,w)+d_{\kappa}(w,z)
=
d_{\kappa}(x',w')+d_{\kappa}(w',z')
=
d_{\kappa}(x',z')
\leq
d_{\kappa}(x,z),
\end{align*}
and thus 
\begin{equation*}
d_{\kappa}(x,z)=d_{\kappa}(x,y)+d_{\kappa}(y,z),\quad
d_{\kappa}(x,z)=d_{\kappa}(x,w)+d_{\kappa}(w,z).
\end{equation*}
It follows that $y\in\lbrack x,z\rbrack$ and $w\in\lbrack x,z\rbrack$, and therefore 
\begin{equation*}
d_{\kappa}(y,w)
=
|d_{\kappa}(x,y)-d_{\kappa}(x,w)|
=
|d_{\kappa}(x',y')-d_{\kappa}(x',w')|
=
d_{\kappa}(y',w').
\end{equation*}
So henceforth we assume that $x'$, $y'$, $z'$ and $w'$ are distinct. 

Because $d_{\kappa}(x',y')+d_{\kappa}(y',z')+d_{\kappa}(z',w')+d_{\kappa}(w',x')<2D_{\kappa}$ by hypothesis, 
one of the inequalities 
$d_{\kappa}(x',y')+d_{\kappa}(y',z')<D_{\kappa}$ or 
$d_{\kappa}(z',w')+d_{\kappa}(w',x')<D_{\kappa}$ holds. 
We may assume without loss of generality that 
\begin{equation*}
d_{\kappa}(x',y')+d_{\kappa}(y',z')<D_{\kappa},\quad
d_{\kappa}(x',y')\leq d_{\kappa}(y',z'). 
\end{equation*}
Then we have 
\begin{equation}\label{quadrangles-zxy-ineq}
\angle yxz
\leq\angle y'x'z'
\end{equation}
by Lemma \ref{from-law-of-cos-lemma}. 
Let $\tilde{z}\in M^2_{\kappa}$ be the point such that 
\begin{equation*}
d_{\kappa}(y',\tilde{z})=d_{\kappa}(y',z'), \quad\angle x'y'\tilde{z} =\angle xyz,
\end{equation*}
and $\tilde{z}$ is not on the opposite side of $\ell (x',y')$ from $z'$, 
as shown in \textsc{Figure} \ref{larger-smaller-fig}. 
Then 
\begin{equation}\label{quadrangles-xyz-cong-eqs}
d_{\kappa}(x',\tilde{z})=d_{\kappa}(x,z),\quad 
\angle y'x'\tilde{z}=\angle yxz
\end{equation}
because the triangle with vertices $x'$, $y'$ and $\tilde{z}$ is congruent to 
that with vertices $x$, $y$ and $z$. 
Since $d_{\kappa}(x',z')\leq d_{\kappa}(x,z)$, 
Lemma \ref{law-of-cos-lemma} implies that 
\begin{equation}\label{quadrangles-xyz-eq}
\angle x'y'z'\leq\angle xyz =\angle x'y'\tilde{z}.
\end{equation}
Because $\tilde{z}$ is not on the opposite side of $\ell (x',y')$ from $z'$, 
\eqref{quadrangles-xyz-eq} implies 
\begin{equation}\label{quadrangles-xyz-zytildez-eq}
\angle x'y'\tilde{z}
=
\angle x'y'z' +\angle z'y'\tilde{z}
\end{equation}
by Proposition \ref{angle-equality-prop}. 
The hypothesis that $\lbrack x',z'\rbrack\cap\lbrack y',w'\rbrack\neq\emptyset$ implies 
\begin{equation}\label{quadrangles-xyw-wyz-eq}
\angle x'y'z'
=
\angle x'y'w'+\angle w'y'z'
\end{equation}
by Proposition \ref{diagonal-angle-prop}. 
By Proposition \ref{angle-betweenness-prop}, 
\eqref{quadrangles-xyz-zytildez-eq} and \eqref{quadrangles-xyw-wyz-eq} imply 
\begin{equation}\label{quadrangles-xyw-wytildez}
\angle x'y'\tilde{z}
=
\angle x'y'w'+\angle w'y'\tilde{z}. 
\end{equation}
Combining \eqref{quadrangles-xyz-eq}, \eqref{quadrangles-xyw-wyz-eq} and \eqref{quadrangles-xyw-wytildez}, we obtain 
$\angle w'y'z'\leq\angle w'y'\tilde{z}$, 
and therefore Lemma \ref{law-of-cos-lemma} implies that 
\begin{equation*}
d_{\kappa}(w,z)=d_{\kappa}(w',z')\leq d_{\kappa}(w',\tilde{z}).
\end{equation*}
Using Lemma \ref{law-of-cos-lemma} again, this implies 
\begin{equation}\label{quadrangles-wxz-ineq}
\angle zxw\leq\angle \tilde{z}x'w' .
\end{equation}
The hypothesis that $\lbrack x',z'\rbrack\cap\lbrack y',w'\rbrack\neq\emptyset$ implies 
\begin{equation}\label{quadrangles-wxz-zxy-eq}
\angle y'x'w'
=
\angle y'x'z' +\angle z'x'w'
\end{equation}
by Proposition \ref{diagonal-angle-prop}. 
It follows from \eqref{quadrangles-zxy-ineq} and \eqref{quadrangles-xyz-cong-eqs} that 
$\angle y'x'\tilde{z}\leq\angle y'x'z'$, which implies 
\begin{equation}\label{quadrangles-zxtildez-tildezxy-eq}
\angle y'x'z'
=
\angle y'x'\tilde{z} +\angle \tilde{z}x'z'
\end{equation}
by Proposition \ref{angle-equality-prop} because $\tilde{z}$ is not on the opposite side of $\ell (x',y')$ from $z'$. 
By Proposition \ref{angle-betweenness-prop}, 
\eqref{quadrangles-wxz-zxy-eq} and \eqref{quadrangles-zxtildez-tildezxy-eq} imply 
\begin{equation}\label{quadrangles-wxtildez-tildezxy-eq}
\angle y'x'w'
=
\angle y'x'\tilde{z}+\angle \tilde{z}x'w'.
\end{equation}
We also have 
\begin{equation}\label{quadrangles-wdashxdashyadash-ineq}
\angle yxw
\leq
\angle yxz+\angle zxw
\end{equation}
by Proposition \ref{angle-triangle-inequality-prop}. 
Combining \eqref{quadrangles-xyz-cong-eqs}, \eqref{quadrangles-wxz-ineq}, \eqref{quadrangles-wxtildez-tildezxy-eq} and 
\eqref{quadrangles-wdashxdashyadash-ineq}, we obtain 
\begin{equation*}
\angle yxw
\leq
\angle yxz+\angle zxw
\leq
\angle y'x'\tilde{z}+\angle \tilde{z}x'w'
=
\angle y'x'w', 
\end{equation*}
and therefore Lemma \ref{law-of-cos-lemma} implies that 
$d_{\kappa}(y,w)\leq d_{\kappa}(y',w')$. 
\end{proof}

Lemma \ref{larger-larger-lemma} and Lemma \ref{larger-smaller-lemma} imply the following corollary.

\begin{corollary}\label{larger-smaller-coro}
Suppose $x,y,z,w,x',y',z',w' \in M_{\kappa}^2$ are points that satisfy the hypothesis of Lemma \ref{larger-smaller-lemma}. 
Assume in addition that $x\neq z$, and that $y$ and $w$ do not lie on the same side of $\ell (x,z)$. 
Then $\lbrack x,z\rbrack\cap\lbrack y,w\rbrack\neq\emptyset$. 
\end{corollary}

\begin{proof}
If one of the equalities $x=y$, $x=w$, $y=z$ or $z=w$ holds, then we have 
$\lbrack x,z\rbrack\cap\lbrack y,w\rbrack\neq\emptyset$ clearly. 
So we assume that 
\begin{equation*}
x\neq y,\quad
x\neq z,\quad
x\neq w,\quad
y\neq z,\quad
z\neq w.
\end{equation*}
Suppose for the sake of contradiction that $\lbrack x,z\rbrack\cap\lbrack y,w\rbrack =\emptyset$. 
Then Proposition \ref{convex-quadrilateral-prop} implies that we have 
$\pi <\angle yxz+\angle zxw$ or $\pi <\angle yzx+\angle xzw$. 
Therefore, Lemma \ref{larger-larger-lemma} implies that $d_{\kappa}(x,z)\leq d_{\kappa}(x',z')$ because 
we have $d_{\kappa}(y,w)\leq d_{\kappa}(y',w')$ by Lemma \ref{larger-smaller-lemma}. 
Combining this with the hypothesis that $d_{\kappa}(x',z')\leq d_{\kappa}(x,z)$ yields $d_{\kappa}(x,z)=d_{\kappa}(x',z')$. 
Therefore, the triangle with vertices $x$, $y$ and $z$ is congruent to that with vertices $x'$, $y'$ and $z'$,  
and the triangle with vertices $x$, $w$ and $z$ is congruent to that with vertices $x'$, $w'$ and $z'$. 
It follows that 
\begin{equation*}
\pi <\angle yxz+\angle zxw=\angle y'x'z'+\angle z'x'w'
\end{equation*}
or 
\begin{equation*}
\pi <\angle yzx+\angle xzw=\angle y'z'x'+\angle x'z'w'.
\end{equation*}
Therefore, it follows from Proposition \ref{convex-quadrilateral-prop} that $\lbrack x',z'\rbrack\cap\lbrack y',w'\rbrack =\emptyset$, 
contradicting the hypothesis that $\lbrack x',z'\rbrack\cap\lbrack y',w'\rbrack\neq\emptyset$. 
Thus we have $\lbrack x,z\rbrack\cap\lbrack y,w\rbrack\neq\emptyset$. 
\end{proof}

\section{Proof of Lemma \ref{quadruple-p-lemma}}\label{quadruple-sec}

In this section, we prove Lemma \ref{quadruple-p-lemma}. 

\begin{proof}[Proof of Lemma \ref{quadruple-p-lemma}]
Let $X$, $x$, $y$, $z$, $w$, $x'$, $y'$, $z'$, and $w'$ be as in the hypothesis, and 
fix $p\in\lbrack x',z' \rbrack$. 
If $x' =z'$, then $p=x'$, and therefore 
\begin{equation*}
d_{X}(y,w)\leq d_{X}(y,x)+d_{X}(x,w)\leq d_{\kappa}(y',x')+d_{\kappa}(x',w')=d_{\kappa}(y',p)+d_{\kappa}(p,w').
\end{equation*}
So henceforth we assume that $x'\neq z'$. 
Then $x\neq z$ since $d_{\kappa}(x',z')\leq d_X (x,z)$. 
Because $X$ is $\mathrm{Cycl}_4 (\kappa )$, 
there exist $x_0 ,y_0 ,z_0 ,w_0\in M_{\kappa}^2$ such that 
\begin{align*}
&d_{\kappa}(x_0, y_0 )\leq d_{X}(x,y),\quad d_{\kappa}(y_0 ,z_0 )\leq d_{X}(y,z),\quad d_{\kappa}(z_0 ,w_0 )\leq d_{X}(z,w),\\
&d_{\kappa}(w_0 ,x_0 )\leq d_{X}(w,x),\quad d_{X}(x,z)\leq d_{\kappa}(x_0 ,z_0 ),\quad d_{X} (y,w)\leq d_{\kappa}(y_0 ,w_0 ).
\end{align*}
Then 
\begin{align*}
|d_{\kappa}(x',y')&-d_{\kappa}(y',z')|
\leq
d_{\kappa}(x',z')
\leq
d_X (x,z)
\leq
d_{\kappa}(x_0 ,z_0 )\\
&\leq
d_{\kappa}(x_0 ,y_0 )+d_{\kappa}(y_0 ,z_0 )
\leq
d_X (x,y)+d_X (y,z)
\leq
d_{\kappa}(x',y')+d_{\kappa}(y',z'),
\end{align*}
and thus 
\begin{equation*}
|d_{\kappa}(x',y')-d_{\kappa}(y',z')|
\leq
d_{\kappa}(x_0 ,z_0 )
\leq
d_{\kappa}(x',y')+d_{\kappa}(y',z').
\end{equation*}
This guarantees that there exists a point $\tilde{y}\in M_{\kappa}^2$ such that 
\begin{equation*}
d_{\kappa}(x_0 ,\tilde{y})=d_{\kappa}(x',y'),\quad d_{\kappa}(\tilde{y},z_0 )=d_{\kappa}(y',z'). 
\end{equation*}
Similarly, there also exists a point $\tilde{w}\in M_{\kappa}^2$ such that 
\begin{equation*}
d_{\kappa}(x_0 ,\tilde{w})=d_{\kappa} (x',w'),\quad d_{\kappa}(\tilde{w},z_0 )=d_{\kappa}(w',z'). 
\end{equation*}
Clearly we may assume that $\tilde{w}$ does not lie on the same side of $\ell (x_0 ,z_0 )$ as $\tilde{y}$. 
Let $\overline{w}\in M_{\kappa}^2$ be the point such that 
\begin{equation*}
d_{\kappa}(x',\overline{w})=d_{\kappa}(x',w'),\quad
d_{\kappa}(\overline{w},z')=d_{\kappa}(w',z'),
\end{equation*}
and $\overline{w}$ is not on the same side of $\ell (x',z' )$ as $y'$. 
(If $w'$ is not on the same side of $\ell (x',z' )$ as $y'$, then $\overline{w}$ is $w'$ itself.) 
Then it is easily seen that $d_{\kappa}(p,\overline{w})=d_{\kappa}(p,w')$, and therefore 
\begin{equation}\label{quadruple-p-lemma-ydashoverw}
d_{\kappa}(y',\overline{w})
\leq
d_{\kappa}(y',p)+d_{\kappa}(p,\overline{w})
=
d_{\kappa}(y',p)+d_{\kappa}(p,w'). 
\end{equation}
We consider two cases. 

\textsc{Case 1}: 
{\em $\lbrack x' ,z' \rbrack\cap\lbrack y' ,\overline{w}\rbrack\neq\emptyset$.} 
In this case, it follows from Lemma \ref{larger-smaller-lemma} that 
\begin{equation}\label{quadruple-p-tildeytildew-ydashoverw1-ineq}
d_{\kappa}(\tilde{y},\tilde{w})\leq d_{\kappa}(y' ,\overline{w})
\end{equation}
because 
\begin{align*}
&d_{\kappa}(x_0 ,\tilde{y})=d_{\kappa}(x' ,y' ),\quad d_{\kappa} (\tilde{y} ,z_0 )=d_{\kappa}(y' ,z' ),\quad 
d_{\kappa}(z_0 ,\tilde{w})=d_{\kappa}(z' ,\overline{w}),\\
&d_{\kappa}(\tilde{w},x_0 )=d_{\kappa}(\overline{w},x' ),
\quad d_{\kappa}(x' ,z' )\leq d_{X}(x,z)\leq d_{\kappa}(x_0 ,z_0 ).
\end{align*}
We also have $\lbrack x_0 ,z_0 \rbrack\cap\lbrack\tilde{y},\tilde{w}\rbrack\neq\emptyset$ by Corollary \ref{larger-smaller-coro}. 
Choose $p_0 \in\lbrack x_0 ,z_0 \rbrack\cap\lbrack\tilde{y},\tilde{w}\rbrack$. 
Then 
\begin{equation*}
d_{\kappa}(y_0, p_0 )\leq d_{\kappa}(\tilde{y},p_0 ),\quad
d_{\kappa}(p_0 ,w_0)\leq d_{\kappa}(p_0 ,\tilde{w})
\end{equation*}
by Corollary \ref{naibunten-coro}, and therefore 
\begin{equation}\label{quadruple-p-lemma-y0p0-p0w0-tildeytildew}
d_{\kappa}(y_0, w_0 )
\leq
d_{\kappa}(y_0, p_0 )+d_{\kappa}(p_0 ,w_0)
\leq
d_{\kappa}(\tilde{y},p_0 )+d_{\kappa}(p_0 ,\tilde{w})
=
d_{\kappa}(\tilde{y},\tilde{w}).
\end{equation}
It follows from \eqref{quadruple-p-lemma-ydashoverw}, \eqref{quadruple-p-tildeytildew-ydashoverw1-ineq} 
and \eqref{quadruple-p-lemma-y0p0-p0w0-tildeytildew} that 
\begin{equation*}
d_{X}(y,w)\leq d_{\kappa}(y_0 ,w_0 )
\leq
d_{\kappa}(\tilde{y},\tilde{w})
\leq
d_{\kappa}(y' ,\overline{w})
\leq
d_{\kappa}(y',p)+d_{\kappa}(p,w' ).
\end{equation*}

\textsc{Case 2}: 
{\em $\lbrack x' ,z' \rbrack\cap\lbrack y' ,\overline{w}\rbrack\neq\emptyset$.} 
In this case, we have 
\begin{equation*}
x'\neq y' ,\quad
x'\neq w' ,\quad
y' \neq z' ,\quad
z' \neq\overline{w}
\end{equation*}
clearly, and one of the inequalities 
$\pi <\angle y' x' z' +\angle z' x' \overline{w}$ or $\pi <\angle y' z' x' +\angle x' z' \overline{w}$ holds by Proposition \ref{convex-quadrilateral-prop}. 
We may assume without loss of generality that 
\begin{equation*}
\pi <\angle y' x' z' +\angle z' x' \overline{w}.
\end{equation*} 
Then we have 
\begin{equation*}
d_{\kappa}(y' ,x')+d_{\kappa}(x',\overline{w})
\leq
d_{\kappa}(y' ,p)+d_{\kappa}(p,\overline{w})
\end{equation*}
by Proposition \ref{triangle-convexhull-coro-prop}, 
and therefore 
\begin{align*}
d_X (y,w)
&\leq
d_X (y,x)+d_X (x,w)
\leq
d_{\kappa}(y',x')+d_{\kappa}(x',w')\\
&=
d_{\kappa}(y',x')+d_{\kappa}(x',\overline{w})
\leq
d_{\kappa}(y',p)+d_{\kappa}(p,\overline{w})
=
d_{\kappa}(y',p)+d_{\kappa}(p,w'),
\end{align*}
which completes the proof. 
\end{proof}

The following two corollaries follow immediately from Lemma \ref{quadruple-p-lemma}. 

\begin{corollary}\label{original-def-coro}
Let $\kappa\in\mathbb{R}$. 
Let $(X,d_{X})$ be a $\mathrm{Cycl}_4 (\kappa )$ space. 
Suppose $x,y,z,w\in X$ and $x' ,y',z' ,w'\in M_{\kappa}^2$ are points such that 
\begin{align*}
&d_{\kappa}(x',y')+d_{\kappa}(y',z')+d_{\kappa}(z',w')+d_{\kappa}(w',x')<2D_{\kappa},\quad \lbrack x' ,z' \rbrack\cap\lbrack y' ,w' \rbrack\neq\emptyset ,\\
&d_X (x,y)\leq d_{\kappa}(x' ,y' ),\quad d_X (y,z)\leq d_{\kappa}(y' ,z' ),\quad d_X (z,w)\leq d_{\kappa}(z' ,w' ),\\
&d_X (w,x)\leq d_{\kappa}(w' ,x' ),\quad d_{\kappa}(x',z')\leq d_{X}(x,z).
\end{align*}
Then $d_X (y,w)\leq d_{\kappa}(y' ,w' )$. 
\end{corollary}

\begin{corollary}\label{simplicial-coro}
Let $\kappa\in\mathbb{R}$. 
Let $(X,d_{X})$ be a $\mathrm{Cycl}_4 (\kappa )$ space, 
and let $(Y,d_Y )$ be a metric space. 
Suppose $x,y,z,w\in X$ and $x',y',z',w'\in Y$ are points such that 
\begin{align*}
&d_{Y}(x',y')+d_{Y}(y',z')+d_{Y}(z',w')+d_{Y}(w',x')<2D_{\kappa},\\
&d_X (x,y)\leq d_Y(x',y'),\quad d_X (y,z)\leq d_Y(y',z'),\quad d_X (z,w)\leq d_Y(z',w'),\\
&d_X (w,x)\leq d_Y(w',x'),\quad d_Y(x',z')\leq d_X (x,z). 
\end{align*}
Assume that there exist subsets $S$ and $T$ of $Y$ satisfying the following conditions: 
\begin{enumerate}
\item[$(1)$]
$S$ and $T$ are isometric to convex subsets of $M_{\kappa}^2$. 
\item[$(2)$]
$\{ x',y',z'\}\subseteq S$ and $\{ x',z',w'\}\subseteq T$.
\item[$(3)$]
There is a geodesic segment $\Gamma_1$ in $Y$ with endpoints $x'$ and $z'$ such that 
$\Gamma_1 \subseteq S\cap T$.
\item[$(4)$]
There is a point $p\in\Gamma_1$ such that $d_Y (y' ,w' )=d_Y (y' ,p)+d_Y (p,w' )$. 
\end{enumerate}
Then $d_X (y,w)\leq d_Y (y',w')$. 
\end{corollary}

\begin{remark}
Clearly, we may replace the condition $(4)$ in Corollary \ref{simplicial-coro} with the following condition:
\begin{enumerate}
\item[$(4')$]
There exists a geodesic segment $\Gamma_2$ in $Y$ with endpoints $y'$ and $w'$ such that 
$\Gamma_1\cap\Gamma_2\neq\emptyset$. 
\end{enumerate}
\end{remark}

\section{A property of convex polygons}\label{convex-polygon-sec}

Before proving Theorem \ref{nongeodesic-majorization-th} and Theorem \ref{Cycl-th}, 
we discuss a certain property of convex polygons. 
Fix $\kappa\in\mathbb{R}$ and an integer $n\geq 3$. 
Suppose $g:\mathbb{Z}/n\mathbb{Z}\to M_{\kappa}^2$ is a map such that 
\begin{equation*}
\sum_{i\in\mathbb{Z}/n\mathbb{Z}}d_{\kappa}(g(i),g(i+\lbrack 1\rbrack_n ))<2D_{\kappa},\quad
g(j)\neq g(j+\lbrack 1\rbrack_n )
\end{equation*}
for every $j\in\mathbb{Z}/n\mathbb{Z}$. 
It is known that if $\bigcup_{i\in\mathbb{Z}/n\mathbb{Z}}\lbrack g(i),g(i+\lbrack 1\rbrack_n )\rbrack$ 
forms a convex polygon in $M_{\kappa}^2$, then we have 
\begin{equation*}
\lbrack g(i),g(j)\rbrack\cap\lbrack g(i-\lbrack 1\rbrack_n ),g(i+\lbrack 1\rbrack_n )\rbrack\neq\emptyset
\end{equation*}
for any $i,j\in\mathbb{Z}/n\mathbb{Z}$ with $i\neq j$. 
We recall some facts concerning this property of convex polygons. 
For completeness, we prove all those facts. 

\begin{lemma}\label{edge-lemma}
Fix $\kappa\in\mathbb{R}$ and an integer $n\geq 3$. 
Suppose $g:\mathbb{Z}/n\mathbb{Z}\to M_{\kappa}^2$ is a map such that 
$\sum_{i\in\mathbb{Z}/n\mathbb{Z}}d_{\kappa}(g(i),g(i+\lbrack 1\rbrack_n ))<2D_{\kappa}$, and 
$g(j)\neq g(j+\lbrack 1\rbrack_n )$ for every $j\in\mathbb{Z}/n\mathbb{Z}$. 
Assume that $\lbrack g(i),g(j)\rbrack\cap\lbrack g(i-\lbrack 1\rbrack_n ),g(i+\lbrack 1\rbrack_n )\rbrack\neq\emptyset$ 
for any $i,j\in\mathbb{Z}/n\mathbb{Z}$ with $i\neq j$. 
Then $g(a)$ and $g(b)$ do not lie on opposite sides of $\ell (g(c),g(c+\lbrack 1\rbrack_n ))$ for any $a,b,c\in\mathbb{Z}/n\mathbb{Z}$. 
\end{lemma}

\begin{figure}[htbp]
\centering\begin{tikzpicture}[scale=0.5]
\draw (0,5) -- (1.5,3);
\draw (0,3) -- (4,0);
\draw[dashed] (-2,3) -- (9,3);
\fill (0,3) circle [radius=0.15];
\fill (0,5) circle [radius=0.15];
\fill (1.5,3) circle [radius=0.15];
\fill (4,0) circle [radius=0.15];
\fill (5,6) circle [radius=0.15];
\node [below left] at (0,3) {$g(c_1 )$};
\node [left] at (0,5) {$g(c_1 -\lbrack 1\rbrack_n )$};
\node [above right] at (1.5,3) {$g(c_1 +\lbrack 1\rbrack_n )$};
\node [right] at (4,0) {$g(b_0 )$};
\node [right] at (5,6) {$g(a_0 )$};
\node [right] at (9,3) {$L$};
\end{tikzpicture}
\caption{Proof of Lemma \ref{edge-lemma}.}\label{edge-lemma-fig}
\end{figure}

\begin{proof}
It follows from the hypothesis and Proposition \ref{hemisphere-prop} that 
we have $d_{\kappa}(g(i),g(j))<D_{\kappa}$ for any $i,j\in\mathbb{Z}/n\mathbb{Z}$. 
Assume for the sake of contradiction that 
there exist $a_0 ,b_0 ,c_0\in\mathbb{Z}/n\mathbb{Z}$ such that 
$g(a_0 )$ and $g(b_0 )$ are on opposite sides of the line $L=\ell (g(c_0 ),g(c_0 +\lbrack 1\rbrack_n ))$. 
Let $m$ be the smallest nonnegative integer such that 
$g(c_0 -\lbrack m+1\rbrack_n )\not\in L$, 
and let $c_1 =c_0 -\lbrack m\rbrack_n$. 
Then $g(c_1 ),g(c_1 +\lbrack 1\rbrack_n )\in L$, and 
$g(c_1 -\lbrack 1\rbrack_n )\not\in L$. 
We may assume without loss of generality that 
$g(c_1 -\lbrack 1\rbrack_n )$ is on the same side of $L$ as $g(a_0 )$. 
Let $A$ be the union of $L$ and the side of it containing $g(a_0 )$, 
and let $B$ be the union of $L$ and the side of it containing $g(b_0 )$. 
Then we have 
\begin{equation*}
\lbrack g(c_1 -\lbrack1\rbrack_n ),g(c_1 +\lbrack1\rbrack_n )\rbrack\subseteq A,\quad
\lbrack g(c_1 ),g(b_0 )\rbrack\subseteq B.
\end{equation*}
because both $A$ and $B$ are $D_{\kappa}$-convex subsets of $M_{\kappa}^2$. 
Since $L=A\cap B$, it follows that 
\begin{equation*}
\lbrack g(c_1 -\lbrack1\rbrack_n ),g( c_1 +\lbrack1\rbrack_n )\rbrack\cap\lbrack g(c_1 ),g(b_0 )\rbrack\subseteq L. 
\end{equation*}
On the other hand, we have 
\begin{align*}
&\lbrack g(c_1 -\lbrack1\rbrack_n ),g( c_1 +\lbrack1\rbrack_n )\rbrack\cap L=\{ g( c_1 +\lbrack1\rbrack_n )\} ,\\
&\lbrack g(c_1 ),g(b_0 )\rbrack\cap L=\{ g(c_1 )\} ,\quad
g(c_1 )\neq g(c_1 +\lbrack 1\rbrack_n ),
\end{align*}
and therefore 
$\lbrack g(c_1 -\lbrack1\rbrack_n ),g( c_1 +\lbrack1\rbrack_n )\rbrack\cap\lbrack g(c_1 ),g(b_0 )\rbrack =\emptyset$, 
contradicting the hypothesis. 
\end{proof}

\begin{lemma}\label{two-angles-lemma}
Fix $\kappa\in\mathbb{R}$ and an integer $n\geq 3$. 
Suppose $g:\mathbb{Z}/n\mathbb{Z}\to M_{\kappa}^2$ is a map such that 
$\sum_{i\in\mathbb{Z}/n\mathbb{Z}}d_{\kappa}(g(i),g(i+\lbrack 1\rbrack_n ))<2D_{\kappa}$, 
and $g(j)\neq g(j+\lbrack 1\rbrack_n )$ for every $j\in\mathbb{Z}/n\mathbb{Z}$. 
Assume that 
$\lbrack g(i),g(j)\rbrack\cap\lbrack g(i-\lbrack 1\rbrack_n ),g(i+\lbrack 1\rbrack_n )\rbrack\neq\emptyset$ 
for any $i,j\in\mathbb{Z}/n\mathbb{Z}$ with $i\neq j$. 
Then 
\begin{align*}
\angle g(k-\lbrack 1\rbrack_n )g(k)g(l)&\leq\angle g(k-\lbrack 1\rbrack_n )g(k)g(k+\lbrack 1\rbrack_{n}),\\
\angle g(l)g(k)g(k+\lbrack 1\rbrack_n )&\leq\angle g(k-\lbrack 1\rbrack_n )g(k)g(k+\lbrack 1\rbrack_{n})
\end{align*}
for any $k,l\in\mathbb{Z}/n\mathbb{Z}$ with $g(k)\neq g(l)$.
\end{lemma}

\begin{figure}[htbp]
\centering\begin{tikzpicture}[scale=0.5]
\draw (2.5,5) -- (5,0);
\draw (0,0) -- (6,2);
\fill (0,0) circle [radius=0.15];
\fill (5,0) circle [radius=0.15];
\fill (6,2) circle [radius=0.15];
\fill (2.5,5) circle [radius=0.15];
\node [below left] at (0,0) {$g(k-\lbrack 1\rbrack_{n})$};
\node [below right] at (5,0) {$g(k)$};
\node [right] at (6.1,2) {$g(k+\lbrack 1\rbrack_{n})$};
\node [above] at (2.5,5.1) {$g(l)$};
\end{tikzpicture}
\caption{Proof of Lemma \ref{two-angles-lemma}.}\label{two-angles-fig}
\end{figure}

\begin{proof}
It follows from the hypothesis and Proposition \ref{hemisphere-prop} that 
we have $d_{\kappa}(g(i),g(j))<D_{\kappa}$ for any $i,j\in\mathbb{Z}/n\mathbb{Z}$. 
Fix $k\in\mathbb{Z}/n\mathbb{Z}$ with $g(k)\neq g(\lbrack 0\rbrack_{n})$. 
Because 
\begin{equation*}
\lbrack g(k),g(l)\rbrack\cap\lbrack g(g(k-\lbrack 1\rbrack_{n}),g(k+\lbrack 1\rbrack_{n})\rbrack\neq\emptyset
\end{equation*}
by hypothesis, it follows from 
Proposition \ref{diagonal-angle-prop} that 
\begin{equation*}
\angle g(k-\lbrack 1\rbrack_{n})g(k)g(l)
+
\angle g(l)g(k)g(k+\lbrack 1\rbrack_{n})
=
\angle g(k-\lbrack 1\rbrack_{n})g(k)g(k+\lbrack 1\rbrack_{n}),
\end{equation*}
which implies the desired inequalities. 
\end{proof}

\begin{lemma}\label{plus-one-lemma}
Fix $\kappa\in\mathbb{R}$ and an integer $n\geq 3$. 
Suppose $g :\mathbb{Z}/(n+1)\mathbb{Z}\to M_{\kappa}^2$ is a map satisfying the following conditions: 
\begin{enumerate}
\item[$(1)$]
$\sum_{i\in\mathbb{Z}/(n+1)\mathbb{Z}}d_{\kappa}(g (i),g (i+\lbrack 1\rbrack_{n+1}))<2D_{\kappa}$;
\item[$(2)$]
$g(j)\neq g(j+\lbrack 1\rbrack_{n+1})$ for every $j\in\mathbb{Z}/(n+1)\mathbb{Z}$;
\item[$(3)$]
The map $g_0 :\mathbb{Z}/n\mathbb{Z}\to M_{\kappa}^2$ defined by $g_0 (\lbrack m\rbrack_n )=g(\lbrack m\rbrack_{n+1})$, 
$m\in\mathbb{Z}\cap\lbrack 0,n-1\rbrack$ satisfies 
$\lbrack g_0 (i),g_0 (j)\rbrack\cap\lbrack g_0 (i-\lbrack 1\rbrack_n ),g_0 (i+\lbrack 1\rbrack_n )\rbrack\neq\emptyset$ 
for any $i,j\in\mathbb{Z}/n\mathbb{Z}$ with $i\neq j$;
\item[$(4)$]
$g(\lbrack n-1\rbrack_{n+1})\neq g(\lbrack 0\rbrack_{n+1})$;
\item[$(5)$]
$g(\lbrack n\rbrack_{n+1})$ is not on the same side of $\ell (g(\lbrack n-1\rbrack_{n+1}),g(\lbrack 0\rbrack_{n+1}))$ 
as $g(k)$ for every $k\in (\mathbb{Z}/(n+1)\mathbb{Z})\setminus\{\lbrack n\rbrack_{n+1}\}$;
\item[$(6)$]
$\lbrack g(\lbrack n\rbrack_{n+1}),g(\lbrack 1\rbrack_{n+1})\rbrack\cap\lbrack g(\lbrack n-1\rbrack_{n+1}),g(\lbrack 0\rbrack_{n+1})\rbrack \neq\emptyset$;
\item[$(7)$]
$\lbrack g(\lbrack n\rbrack_{n+1}),g(\lbrack n-2\rbrack_{n+1})\rbrack\cap\lbrack g(\lbrack n-1\rbrack_{n+1}),g(\lbrack 0\rbrack_{n+1})\rbrack \neq\emptyset$.
\end{enumerate}
Then we have 
\begin{equation*}
\lbrack g(i),g(j)\rbrack\cap\lbrack g(i-\lbrack 1\rbrack_{n+1}),g(i+\lbrack 1\rbrack_{n+1})\rbrack\neq\emptyset
\end{equation*}
for any $i,j\in\mathbb{Z}/(n+1)\mathbb{Z}$ with $i\neq j$. 
\end{lemma}

\begin{proof}
Because $g$ satisfies the condition $(1)$, it follows from Proposition \ref{hemisphere-prop} 
that we have $d_{\kappa}(x,y)<D_{\kappa}$ 
for any $x,y\in\mathrm{conv}(g(\mathbb{Z}/(n+1)\mathbb{Z}))$. 
For each integer $m$, we denote the point $g(\lbrack m\rbrack_{n+1})\in M_{\kappa}^2$ by $x_m$. 
Since $g$ satisfies the condition $(3)$ and we clearly have 
$\lbrack x_s ,x_{t}\rbrack\cap\lbrack x_{s-1},x_{s+1}\rbrack\neq\emptyset$ 
for any $s,t\in\mathbb{Z}$ with $|s-t|=1$, 
it suffices to prove that 
\begin{align}
\label{plus-one-n-m}
\lbrack x_n ,x_m \rbrack\cap
\lbrack x_{n-1},x_0 \rbrack &\neq\emptyset ,\\
\label{plus-one-m-n}
\lbrack x_m ,x_n \rbrack\cap
\lbrack x_{m-1},x_{m+1}\rbrack &\neq\emptyset ,\\
\label{plus-one-0-k}
\lbrack x_0 ,x_k \rbrack\cap
\lbrack x_{n},x_{1}\rbrack &\neq\emptyset ,\\
\label{plus-one-n-1-l}
\lbrack x_{n-1} ,x_{l} \rbrack\cap
\lbrack x_{n-2},x_{n}\rbrack &\neq\emptyset
\end{align}
for any $m\in\mathbb{Z}\cap\lbrack 1,n-2\rbrack$, $k\in\mathbb{Z}\cap\lbrack 2,n-1\rbrack$ and 
$l\in\mathbb{Z}\cap\lbrack 0,n-3\rbrack$.

\begin{figure}[htbp]
\centering\begin{tikzpicture}[scale=0.5]
\draw (0,0) -- (6,2);
\draw (6,2) -- (2,-1);
\draw (2,-1) -- (2.5,5);
\draw (2.5,5) -- (5,0);
\draw (5,0) -- (0,0);
\fill (0,0) circle [radius=0.15];
\fill (2,-1) circle [radius=0.15];
\fill (5,0) circle [radius=0.15];
\fill (6,2) circle [radius=0.15];
\fill (2.5,5) circle [radius=0.15];
\fill (-0.5,1.5) circle [radius=0.15];
\node [below left] at (0,0) {$x_{n-1}$};
\node [below] at (2,-1.2) {$x_n$};
\node [below right] at (5,0) {$x_0$};
\node [right] at (6.1,2) {$x_1$};
\node [above] at (2.5,5.1) {$x_m$};
\node [left] at (-0.6,1.5) {$x_{n-2}$};
\end{tikzpicture}
\caption{Proof of \eqref{plus-one-n-m}.}\label{plus-one-n-m-fig}
\end{figure}

First we prove \eqref{plus-one-n-m}. 
Fix an integer $m\in\mathbb{Z}\cap\lbrack 1,n-2\rbrack$. 
If $x_m =x_0$ or $x_m =x_{n-1}$, 
then \eqref{plus-one-n-m} holds clearly. 
So we assume that $x_m \neq x_0$ and $x_m \neq x_{n-1}$. 
Since $g$ satisfies the condition $(6)$, 
Corollary \ref{diagonal-angle-coro} implies 
\begin{equation}\label{plus-one-10n-angle}
\angle x_n x_0 x_{n-1}
+
\angle x_{n-1}x_0 x_1
\leq\pi .
\end{equation}
Since $g$ satisfies the conditions $(1)$, $(2)$ and $(4)$, 
$g_0$ satisfies 
\begin{equation*}
\sum_{i\in\mathbb{Z}/n\mathbb{Z}}d_{\kappa}(g_0 (i),g_0 (i+\lbrack 1\rbrack_n ))<2D_{\kappa}
\end{equation*}
and $g_0 (j)\neq g_0 (j+\lbrack 1\rbrack_n )$ for every $j\in\mathbb{Z}/n\mathbb{Z}$. 
Therefore Lemma \ref{two-angles-lemma} implies 
\begin{equation}\label{plus-one-n-10m-n-101-ineq}
\angle x_{n-1}x_0 x_m
\leq
\angle x_{n-1}x_0 x_1
\end{equation}
because $g_0$ satisfies the condition $(3)$. 
Combining \eqref{plus-one-10n-angle} and \eqref{plus-one-n-10m-n-101-ineq}, we obtain 
\begin{equation*}
\angle x_n x_0 x_{n-1}
+
\angle x_{n-1}x_0 x_m
\leq
\angle x_n x_0 x_{n-1}
+
\angle x_{n-1}x_0 x_1
\leq\pi .
\end{equation*}
The same argument shows that 
\begin{equation*}
\angle x_n x_{n-1}x_0
+
\angle x_0 x_{n-1}x_m
\leq\pi .
\end{equation*}
Furthermore, 
$x_n$ and $x_m$ are not on the same side of $\ell (x_ {n-1},x_{0})$ since $g$ satisfies the condition $(5)$. 
Therefore, 
Proposition \ref{convex-quadrilateral-prop} implies \eqref{plus-one-n-m}. 

\begin{figure}[htbp]
\centering\begin{tikzpicture}[scale=0.5]
\draw (0,3.5) -- (6,3.5);
\draw (2,-1) -- (2.5,5);
\draw (0,0) -- (5,0);
\draw (0,0) -- (2.5,5);
\draw (2.5,5) -- (5,0);
\draw (1.75,3.5) -- (2.08,0);
\draw (3.25,3.5) -- (2.08,0);
\fill (0,0) circle [radius=0.15];
\fill (2,-1) circle [radius=0.15];
\fill (5,0) circle [radius=0.15];
\fill (6,3.5) circle [radius=0.15];
\fill (2.5,5) circle [radius=0.15];
\fill (0,3.5) circle [radius=0.15];
\fill (2.08,0) circle [radius=0.15];
\fill (1.75,3.5) circle [radius=0.15];
\fill (3.25,3.5) circle [radius=0.15];
\node [below left] at (0,0) {$x_{n-1}$};
\node [above left] at (0,3.5) {$x_{m+1}$};
\node [below] at (2,-1.1) {$x_n$};
\node [below right] at (5,0) {$x_0$};
\node [above right] at (6,3.5) {$x_{m-1}$};
\node [above] at (2.5,5.1) {$x_{m}$};
\node [below right] at (2.08,0) {$p$};
\node [above left] at (1.75,3.5) {$a$};
\node [above right] at (3.25,3.5) {$b$};
\end{tikzpicture}
\caption{Proof of \eqref{plus-one-m-n}.}\label{plus-one-m-n-fig}
\end{figure}

Next we prove \eqref{plus-one-m-n}. 
Again, fix an integer $m\in\mathbb{Z}\cap\lbrack 1,n-2\rbrack$. 
By \eqref{plus-one-n-m}, there exists a point $p\in\lbrack x_{n},x_{m}\rbrack\cap\lbrack x_{n-1},x_{0}\rbrack$. 
Because $g$ satisfies the condition $(3)$, there also exist points 
$a\in\lbrack x_m ,x_{n-1}\rbrack\cap\lbrack x_{m-1},x_{m+1}\rbrack$ and 
$b\in\lbrack x_m ,x_{0}\rbrack\cap\lbrack x_{m-1},x_{m+1}\rbrack$. 
If one of the equalities $x_m =a$, $x_m =b$, $p=a$, $p=b$, $p=x_{n-1}$ or $p=x_{0}$ holds, 
then \eqref{plus-one-m-n} holds clearly. 
If $x_m =p$, then we have $a\in\lbrack x_{n-1},p\rbrack$ and $b\in\lbrack p,x_0\rbrack$, 
which implies $p\in\lbrack a,b\rbrack$ since $p\in\lbrack x_{n-1},x_{0}\rbrack$, 
and therefore \eqref{plus-one-m-n} holds. 
So henceforth we assume that $x_{m}\not\in\{ p,a,b\}$ and $p\not\in\{ a,b,x_{n-1},x_{0},x_m\}$. 
Then it follows from Proposition \ref{diagonal-angle-prop} that 
\begin{align}
&\angle ax_m p+\angle px_m b
=
\angle x_{n-1}x_m p+\angle px_m x_0
=
\angle x_{n-1}x_m x_0
\leq
\pi ,\label{plus-one-amp-pmb}\\
&\angle x_{n-1}p x_m +\angle x_m px_0
=
\angle x_{n-1}p x_0
=
\pi ,\label{plus-one-n-1pm-mp0}\\
&\angle x_{n-1}pa+\angle apx_m =\angle x_{n-1}px_m ,\quad 
\angle x_{m}pb+\angle bpx_0 =\angle x_m px_0 .\label{plus-one-n-1pa-apm-mpb-bp0}
\end{align}
Combining \eqref{plus-one-n-1pm-mp0} and \eqref{plus-one-n-1pa-apm-mpb-bp0}, we obtain 
\begin{equation}\label{plus-one-apm-mpb}
\angle apx_{m}+\angle x_m pb
\leq
\angle x_{n-1}px_m +\angle x_m px_0
=\pi .
\end{equation}
Because $a\in (x_m ,x_{n-1}\rbrack$ and $b\in (x_m ,x_{0}\rbrack$, 
$x_{n-1}$ is on the same side of $\ell (x_m ,p)$ as $a$ and 
$x_{0}$ is on the same side of $\ell (x_m ,p)$ as $b$ 
whenever $a\not\in\ell (x_m ,p)$ and $b\not\in\ell (x_m ,p)$. 
It follows that $a$ and $b$ are not on the same side of $\ell (x_m ,p)$ 
because otherwise $x_{n-1}$ and $x_0$ would lie on the same side of $\ell (x_m ,x_n )=\ell (x_m ,p)$, 
contradicting \eqref{plus-one-n-m}. 
Therefore, \eqref{plus-one-amp-pmb} and \eqref{plus-one-apm-mpb} imply 
$\lbrack x_{m},p\rbrack\cap\lbrack a,b\rbrack\neq\emptyset$ by Proposition \ref{convex-quadrilateral-prop}, 
which implies \eqref{plus-one-m-n}.

\begin{figure}[htbp]
\centering\begin{tikzpicture}[scale=0.5]
\draw (0,0) -- (6,2);
\draw (6,2) -- (2,-1);
\draw (2,-1) -- (1.5,5);
\draw (1.5,5) -- (5,0);
\draw (5,0) -- (0,0);
\fill (0,0) circle [radius=0.15];
\fill (2,-1) circle [radius=0.15];
\fill (5,0) circle [radius=0.15];
\fill (6,2) circle [radius=0.15];
\fill (1.5,5) circle [radius=0.15];
\fill (3.33,0) circle [radius=0.15];
\node [below left] at (0,0) {$x_{n-1}$};
\node [below] at (2,-1.2) {$x_n$};
\node [below right] at (5,0) {$x_0$};
\node [right] at (6.1,2) {$x_1$};
\node [above] at (1.5,5.1) {$x_k$};
\node [below right] at (3.33,0) {$q$};
\end{tikzpicture}
\caption{Proof of \eqref{plus-one-0-k}.}\label{plus-one-0-k-fig}
\end{figure}

Next we prove \eqref{plus-one-0-k}. 
Fix $k\in\mathbb{Z}\cap\lbrack 2,n-1\rbrack$. 
If $x_k =x_1$ or $x_k =x_n$, then \eqref{plus-one-0-k} holds clearly. 
Since $g$ satisfies the condition $(6)$, there exists a point $q\in\lbrack x_n ,x_1\rbrack\cap\lbrack x_{n-1},x_0 \rbrack$. 
In particular, \eqref{plus-one-0-k} holds whenever $x_k =x_{n-1}$. 
Since $g$ satisfies the condition $(3)$, we have 
\begin{equation}\label{plus-one-g0-0-k}
\lbrack x_0 ,x_{k}\rbrack\cap\lbrack x_{n-1},x_{1}\rbrack\neq\emptyset .
\end{equation}
If $x_k =x_0$, then $x_0 \in\lbrack x_{n-1},x_1 \rbrack$ by \eqref{plus-one-g0-0-k}, 
which implies $x_0 \in\lbrack q,x_1 \rbrack$ since $q\in\lbrack x_{n-1},x_0 \rbrack$, 
and therefore \eqref{plus-one-0-k} holds. 
So henceforth we assume that $x_k \not\in\{ x_{n-1} ,x_n ,x_0 ,x_1 \}$. 
By Proposition \ref{diagonal-angle-prop}, \eqref{plus-one-g0-0-k} implies that 
\begin{align}
\angle x_{n-1}x_{0}x_{k}+\angle x_{k}x_{0}x_{1}&=\angle x_{n-1}x_{0}x_{1},\label{plus-one-n-10k-k01}\\
\angle x_{n-1}x_{k}x_{0}+\angle x_{0}x_{k}x_{1}&=\angle x_{n-1}x_{k}x_{1}.\label{plus-one-n-1k0-0k1}
\end{align}
Since $\lbrack x_n ,x_k \rbrack\cap\lbrack x_{n-1},x_{0}\rbrack\neq\emptyset$ 
by \eqref{plus-one-n-m}, 
Proposition \ref{diagonal-angle-prop} also implies that 
\begin{align}
\angle x_{n}x_{0}x_{n-1}+\angle x_{n-1}x_{0}x_{k}&=\angle x_{n}x_{0}x_{k},\label{plus-one-n0n-1-n-10k}\\
\angle x_{n-1}x_{k}x_{n}+\angle x_{n}x_{k}x_{0}&=\angle x_{n-1}x_{k}x_{0}\label{plus-one-n-1kn-nk0}.
\end{align}
By \eqref{plus-one-10n-angle}, \eqref{plus-one-n-10k-k01} and \eqref{plus-one-n0n-1-n-10k}, we have 
\begin{align}\label{plus-one-n0k-k01}
\angle x_{n}x_{0}x_{k}+\angle x_{k}x_{0}x_{1}
&=\angle x_{n}x_{0}x_{n-1}+x_{n-1}x_{0}x_{k}+\angle x_{k}x_{0}x_{1}\\
&=\angle x_{n}x_{0}x_{n-1}+\angle x_{n-1}x_{0}x_{1}
\leq\pi .\nonumber
\end{align}
By \eqref{plus-one-n-1k0-0k1} and \eqref{plus-one-n-1kn-nk0}, we have 
\begin{equation}\label{plus-one-nk0-0k1}
\angle x_{n}x_{k}x_{0}+\angle x_{0}x_{k}x_{1}
\leq
\angle x_{n-1}x_{k}x_{0}+\angle x_{0}x_{k}x_{1}
=
\angle x_{n-1}x_{k}x_{1}
\leq\pi .
\end{equation}
Because $x_{n-1}$ and $x_{1}$ are not on the same side of $\ell (x_0 ,x_k )$ by \eqref{plus-one-g0-0-k}, 
if $x_{n}$ and $x_{1}$ were on the same side $S$ of $\ell (x_0 ,x_k )$, 
then we would have $\lbrack x_{n},x_{1}\rbrack\subseteq S$ 
and $\lbrack x_{n-1},x_0 \rbrack\subseteq M_{\kappa}^2 \setminus S$, 
contradicting the assumption that $g$ satisfies the condition $(6)$. 
Thus $x_{n}$ and $x_{1}$ are not on the same side of $\ell (x_0 ,x_k )$. 
Therefore, 
\eqref{plus-one-n0k-k01} and \eqref{plus-one-nk0-0k1} imply \eqref{plus-one-0-k} by Proposition \ref{convex-quadrilateral-prop}. 

Exactly the same argument as in the proof of \eqref{plus-one-0-k} shows that \eqref{plus-one-n-1-l} holds for every 
$l\in\mathbb{Z}\cap\lbrack 0,n-3\rbrack$, which completes the proof. 
\end{proof}

\section{Proofs of Theorem \ref{nongeodesic-majorization-th} and Theorem \ref{Cycl-th}}\label{quadruple-Cycln-sec}

In this section, we prove Theorem \ref{nongeodesic-majorization-th} and Theorem \ref{Cycl-th}. 
To prove Theorem \ref{nongeodesic-majorization-th}, we define the following conditions by slightly 
modifying the definition of the $\mathrm{Cycl}_n (\kappa )$ conditions. 

\begin{definition}\label{strongly-Cycl-def}
Fix $\kappa\in\mathbb{R}$ and a positive integer $n$. 
We say that a metric space $(X,d_X )$ is a {\em $\mathrm{Cycl}'_n(\kappa )$ space} if 
for any map $f:\mathbb{Z}/n\mathbb{Z}\to X$ that satisfies 
\begin{equation*}
\sum_{i\in\mathbb{Z}/n\mathbb{Z}}d_X \left( f(i),f(i+\lbrack 1\rbrack_n )\right)<2 D_{\kappa},\quad
f(j)\neq f(j+\lbrack 1\rbrack_n )
\end{equation*}
for every $j\in\mathbb{Z}/n\mathbb{Z}$, 
there exists a map $g:\mathbb{Z}/n\mathbb{Z}\to M_{\kappa}^2$ 
that satisfies the following two conditions: 
\begin{enumerate}
\item[$(1)$]
For any $i,j\in\mathbb{Z}/n\mathbb{Z}$, 
\begin{equation*}
d_{\kappa}(g(i),g(i+\lbrack 1\rbrack_n )) =d_X (f(i),f(i+\lbrack 1\rbrack_n )),\quad 
d_{\kappa}(g(i),g(j))\geq d_X (f(i),f(j)).
\end{equation*}
\item[$(2)$]
For any $i,j\in\mathbb{Z}/n\mathbb{Z}$ with $i\neq j$, 
$\lbrack g(i),g(j)\rbrack\cap\lbrack g(i-\lbrack 1\rbrack_n ),g(i+\lbrack 1\rbrack_n )\rbrack\neq\emptyset$. 
\end{enumerate}
We call such a map $g:\mathbb{Z}/n\mathbb{Z}\to M_{\kappa}^2$ 
that satisfies the above two conditions a {\em comparison map} of $f$. 
\end{definition}

It is easily seen that every metric space is 
$\mathrm{Cycl}'_1(\kappa )$, $\mathrm{Cycl}'_2 (\kappa )$ and $\mathrm{Cycl}'_3 (\kappa )$ for 
any $\kappa\in\mathbb{R}$. 
The $\mathrm{Cycl}'_{n}(\kappa )$ conditions and the $\mathrm{Cycl}_{n}(\kappa )$ conditions 
are related by the following lemma. 

\begin{lemma}\label{strongly-Cycl-relation-lemma}
Fix $\kappa\in\mathbb{R}$ and an integer $n\geq 4$. 
If a metric space $X$ is $\mathrm{Cycl}'_{m}(\kappa )$ for every $m\in\mathbb{Z}\cap\lbrack 1,n\rbrack$, 
then $X$ is $\mathrm{Cycl}_n (\kappa )$. 
\end{lemma}

\begin{proof}
Fix $\kappa\in\mathbb{R}$ and an integer $n\geq 4$. 
Assume that a metric space $(X,d_X )$ is $\mathrm{Cycl}'_{m}(\kappa )$ for every $m\in\mathbb{Z}\cap\lbrack 1,n\rbrack$. 
To prove that $X$ is $\mathrm{Cycl}_n (\kappa )$, 
let $f:\mathbb{Z}/n\mathbb{Z}\to X$ be an arbitrary map with 
$\sum_{i\in\mathbb{Z}/n\mathbb{Z}}d_{X}(f(i),f(i+\lbrack 1\rbrack_n ))<2D_{\kappa}$. 
If $f$ is constant, then we can take a constant map as 
a map $g:\mathbb{Z}/n\mathbb{Z}\to M_{\kappa}^2$ in Definition \ref{Cycl-def}. 
So assume that $f$ is nonconstant. 
Then there exists $m_0 \in\mathbb{Z}$ with 
$f(\lbrack m_0 -1\rbrack_n )\neq f(\lbrack m_0\rbrack_n )$. 
Let $N$ be the cardinality of the set $\{i\in\mathbb{Z}/n\mathbb{Z}\hspace{1mm}|\hspace{1mm}f(i-\lbrack 1\rbrack_n )\neq f(i)\}$. 
Then there exist $N$ distinct integers $m_1 ,m_2 ,\ldots ,m_{N}\in\mathbb{Z}\cap\lbrack m_0 +1,m_0 +n\rbrack$ such that 
\begin{equation*}
m_0 <m_1 <\cdots <m_{N},\quad
f(\lbrack m_l -1\rbrack_n )\neq f(\lbrack m_l\rbrack_n )
\end{equation*}
for every $l\in\mathbb{Z}\cap\lbrack 1,N\rbrack$. 
We clearly have $1\leq N\leq n$ and $m_{N}=m_0 +n$. 
Define a map $f_0 :\mathbb{Z}/N\mathbb{Z}\to X$ by 
\begin{equation*}
f_0 (\lbrack l\rbrack_{N})=f(\lbrack m_l \rbrack_n ),\quad
l\in\mathbb{Z}\cap\lbrack 0,N-1\rbrack .
\end{equation*}
Then $\sum_{i\in\mathbb{Z}/N\mathbb{Z}}d_{\kappa}(f_{0}(i),f_{0}(i+\lbrack 1\rbrack_{N}))<2D_{\kappa}$, and 
$f_{0}(j)\neq f_{0}(j+\lbrack 1\rbrack_{N})$ for every $j\in\mathbb{Z}/N\mathbb{Z}$. 
Therefore, there exists a map $g_0 :\mathbb{Z}/N\mathbb{Z}\to M_{\kappa}^2$ such that 
\begin{equation*}
d_{\kappa}(g_0 (i),g_0 (i+\lbrack 1\rbrack_{N}))= d_X (f_0 (i),f_0 (i+\lbrack 1\rbrack_{N})),\quad
d_{\kappa}(g_0 (i),g_0 (j))\geq d_X (f_0 (i),f_0 (j))
\end{equation*}
for any $i,j\in\mathbb{Z}/N\mathbb{Z}$ 
because $X$ is $\mathrm{Cycl}'_{N}(\kappa )$. 
Define a map $g_1 :\mathbb{Z}/n\mathbb{Z}\to M_{\kappa}^2$ by setting 
$g_1 (\lbrack m\rbrack_{n}) =g_0 (\lbrack l\rbrack_{N})$ 
when $m\in\mathbb{Z}\cap\lbrack m_{l},m_{l+1})$, $l\in\mathbb{Z}\cap\lbrack 0,N-1\rbrack$. 
Then it is easily seen that 
\begin{equation*}
d_{\kappa}(g_1 (i),g_1 (i+\lbrack 1\rbrack_{n}))=d_X (f(i),f(i+\lbrack 1\rbrack_{n})),\quad
d_{\kappa}(g_1 (i),g_1 (j))\geq d_X (f(i),f(j))
\end{equation*}
for any $i,j\in\mathbb{Z}/n\mathbb{Z}$. 
Thus $X$ is $\mathrm{Cycl}_n (\kappa )$. 
\end{proof}

To prove Theorem \ref{nongeodesic-majorization-th}, 
it clearly suffices to prove that every $\mathrm{Cycl}_4 (\kappa )$ space is 
$\mathrm{Cycl}'_n (\kappa )$ for every positive integer $n$. 
The convexity condition $(2)$ in Definition \ref{strongly-Cycl-def} allows us to prove it  
by induction on $n$. 
A similar idea was used by 
Ballmann in his lecture note \cite{B} 
for proving Reshetnyak's majorization theorem.

\begin{proof}[Proof of Theorem \ref{nongeodesic-majorization-th}]
Fix $\kappa\in\mathbb{R}$. 
We will prove that every $\mathrm{Cycl}_4 (\kappa)$ space is $\mathrm{Cycl}'_n (\kappa )$ 
for every positive integer $n$ by induction on $n$. 
As we mentioned above, 
every metric space is $\mathrm{Cycl}'_1 (\kappa )$, $\mathrm{Cycl}'_2 (\kappa )$ and $\mathrm{Cycl}'_3 (\kappa )$ trivially. 
Fix an integer $n\geq 3$, and 
assume that every $\mathrm{Cycl}_4 (\kappa)$ space is $\mathrm{Cycl}'_{l} (\kappa )$ 
for every $l\in\mathbb{Z}\cap\lbrack 1,n\rbrack$. 
Let $(X,d_X )$ be an arbitrary $\mathrm{Cycl}_4 (\kappa )$ space, and let 
$f:\mathbb{Z}/(n+1)\mathbb{Z}\to X$ be a map that satisfies 
\begin{equation}\label{strongly-f-perimeter-distinct-assumption}
\sum_{i\in\mathbb{Z}/(n+1)\mathbb{Z}}d_X (f(i),f(i+\lbrack 1\rbrack_{n+1}))<2 D_{\kappa},\quad
f(j)\neq f(j+\lbrack 1\rbrack_{n+1})
\end{equation}
for every $j\in\mathbb{Z}/(n+1)\mathbb{Z}$. 
We will prove the existence of a comparison map of $f$. 
Define a map $f_0 :\mathbb{Z}/n\mathbb{Z}\to X$ by 
\begin{equation*}
f_0 (\lbrack m\rbrack_n )=f(\lbrack m\rbrack_{n+1}),\quad
m\in\mathbb{Z}\cap\lbrack 0,n-1\rbrack .
\end{equation*}
We will consider the case in which $f(\lbrack n-1\rbrack_{n+1})=f(\lbrack 0\rbrack_{n+1})$ later, 
and we first assume that $f(\lbrack n-1\rbrack_{n+1})\neq f(\lbrack 0\rbrack_{n+1})$. 
Then it follows from \eqref{strongly-f-perimeter-distinct-assumption} that we have 
\begin{equation*}
\sum_{i\in\mathbb{Z}/n\mathbb{Z}}d_X (f_0 (i),f_0 (i+\lbrack 1\rbrack_{n}))<2 D_{\kappa},\quad
f_0 (j)\neq f_0 (j+\lbrack 1\rbrack_{n})
\end{equation*}
for every $j\in\mathbb{Z}/n\mathbb{Z}$, and so 
there exists a comparison map $g_0 :\mathbb{Z}/n\mathbb{Z}\to M_{\kappa}^2$ of $f_0$ 
by the inductive hypothesis. 
Since 
\begin{align*}
d_{\kappa}(g_0 (\lbrack n-1\rbrack_n ),g_0 (\lbrack 0\rbrack_n ))=d_X (f_0 (\lbrack n-1\rbrack_n ),f_0 (\lbrack 0\rbrack_n ))
=d_X (f(\lbrack n-1\rbrack_{n+1}),f(\lbrack 0\rbrack_{n+1})) ,
\end{align*}
there exists $p\in M_{\kappa}^2$ such that 
\begin{align*}
d_{\kappa}(g_0 (\lbrack n-1\rbrack_n ),p)&=d_X (f(\lbrack n-1\rbrack_{n+1}) ,f(\lbrack n\rbrack_{n+1})), \\
d_{\kappa}(p,g_0 (\lbrack 0\rbrack_n ))&=d_X (f(\lbrack n\rbrack_{n+1}) ,f(\lbrack 0\rbrack_{n+1})).
\end{align*}
Because $g_0$ is a comparison map of $f_0$, it follows from Lemma \ref{edge-lemma} that 
for any $i,j\in\mathbb{Z}/n\mathbb{Z}$, $g_0 (i)$ and $g_0 (j)$ 
do not lie on opposite sides of $\ell (g_0 (\lbrack n-1\rbrack_n ),g_0 (\lbrack 0\rbrack_n ))$. 
So we may assume that 
$p$ is not on the same side of $\ell (g_0 (\lbrack n-1\rbrack_n ),g_0 (\lbrack 0\rbrack_n ))$ as $g_0 (i)$ for every $i\in\mathbb{Z}/n\mathbb{Z}$. 
Define a map $g:\mathbb{Z}/(n+1)\mathbb{Z}\to M_{\kappa}^2$ by 
\begin{equation*}
g (\lbrack m\rbrack_{n+1})
=
\begin{cases}
g_0 (\lbrack m\rbrack_n),\quad &\textrm{if }m\in\mathbb{Z}\cap\lbrack 0,n-1\rbrack ,\\
p,\quad &\textrm{if }m=n.
\end{cases}
\end{equation*}
Then 
\begin{align}
\label{g-mn-ineq}
d_{\kappa}\left( g(\lbrack l\rbrack_{n+1}),g(\lbrack m\rbrack_{n+1})\right)
&=
d_{\kappa}\left( g_0 (\lbrack l\rbrack_n ),g_0 (\lbrack m\rbrack_n )\right) \\
&\geq
d_X ( f_0 (\lbrack l\rbrack_n ),f_0 (\lbrack m\rbrack_n ))
=
d_X \left( f (\lbrack l\rbrack_{n+1}),f (\lbrack m\rbrack_{n+1})\right)\nonumber
\end{align}
for any $l,m\in\mathbb{Z}\cap\lbrack 0,n-1\rbrack$, and 
\begin{align}\label{g-ll+1-eq}
d_{\kappa}( g (\lbrack l\rbrack_{n+1}),&g (\lbrack l+1\rbrack_{n+1}))
=
d_{\kappa}\left( g_0 (\lbrack l\rbrack_{n}),g_0 (\lbrack l+1\rbrack_{n})\right) \\
&=
d_X (f_0 (\lbrack l\rbrack_{n}),f_0 (\lbrack l+1\rbrack_{n}))
=
d_X \left(f (\lbrack l\rbrack_{n+1}),f(\lbrack l+1\rbrack_{n+1})\right)\nonumber
\end{align}
for any  $l\in\mathbb{Z}\cap\lbrack 0,n-2\rbrack$. 
Furthermore, 
\begin{align}
\label{g-k-1k-eq}
d_{\kappa}(g(\lbrack n-1\rbrack_{n+1}),g(\lbrack n\rbrack_{n+1}))
&=
d_{\kappa}(g_0 (\lbrack n-1\rbrack_n ),p)
=
d_X (f(\lbrack n-1\rbrack_{n+1}) ,f(\lbrack n\rbrack_{n+1})),\\
\label{g-k0-eq}
d_{\kappa}(g(\lbrack n\rbrack_{n+1}),g(\lbrack 0\rbrack_{n+1}))
&=
d_{\kappa}(p,g_0 (\lbrack 0\rbrack_n ))
=
d_X (f(\lbrack n\rbrack_{n+1}) ,f(\lbrack 0\rbrack_{n+1})),\\
\label{g-k-10-eq}
d_{\kappa}(g(\lbrack n-1\rbrack_{n+1}),g(\lbrack 0\rbrack_{n+1}))
&=
d_{\kappa}(g_0 (\lbrack n-1\rbrack_{n}),g_0 (\lbrack 0\rbrack_n )) \\
=
d_X &( f_0 (\lbrack n-1\rbrack_{n}),f_0 (\lbrack 0\rbrack_n ))
=
d_X (f(\lbrack n-1\rbrack_{n+1}) ,f(\lbrack 0\rbrack_{n+1})).\nonumber
\end{align}
We consider two cases. 

\textsc{Case 1}: 
{\em The map $g$ satisfies 
\begin{align*}
\lbrack g(\lbrack n\rbrack_{n+1}),g(\lbrack n-2\rbrack_{n+1})\rbrack\cap\lbrack g(\lbrack n-1\rbrack_{n+1}),g(\lbrack 0\rbrack_{n+1})\rbrack &\neq\emptyset ,\\
\lbrack g(\lbrack n\rbrack_{n+1}),g(\lbrack 1\rbrack_{n+1})\rbrack\cap\lbrack g(\lbrack n-1\rbrack_{n+1}),g(\lbrack 0\rbrack_{n+1})\rbrack &\neq\emptyset .
\end{align*}}
In this case, we have 
\begin{equation*}
\lbrack g(i),g(j)\rbrack\cap\lbrack g(i-\lbrack 1\rbrack_{n+1}),g(i+\lbrack 1\rbrack_{n+1})\rbrack\neq\emptyset
\end{equation*}
for any $i,j\in\mathbb{Z}/(n+1)\mathbb{Z}$ with $i\neq j$ by Lemma \ref{plus-one-lemma}. 
In particular, we have 
\begin{equation*}
\lbrack g(\lbrack n-1\rbrack_{n+1}),g(\lbrack 0\rbrack_{n+1})\rbrack\cap
\lbrack g(\lbrack m\rbrack_{n+1}),g(\lbrack n\rbrack_{n+1})\rbrack\neq\emptyset
\end{equation*}
for every $m\in\mathbb{Z}\cap\lbrack 0,n-1\rbrack$. 
It easily follows from \eqref{strongly-f-perimeter-distinct-assumption}, the triangle inequality for $d_{\kappa}$, and 
the definition of $g$ that 
\begin{multline*}
d_{\kappa} (g(\lbrack 0\rbrack_{n+1}),g(\lbrack m\rbrack_{n+1}))
+d_{\kappa} (g(\lbrack m\rbrack_{n+1}),g(\lbrack n-1\rbrack_{n+1}))\\
+d_{\kappa} (g(\lbrack n-1\rbrack_{n+1}),g(\lbrack n\rbrack_{n+1}))
+d_{\kappa} (g(\lbrack n\rbrack_{n+1}),g(\lbrack 0\rbrack_{n+1}))
<2D_{\kappa}
\end{multline*}
Therefore, Corollary \ref{original-def-coro} implies 
\begin{equation}\label{strongly-Cycl-th-g-i-k+1-ineq}
d_{\kappa}(g(\lbrack m\rbrack_{n+1}),g(\lbrack n\rbrack_{n+1}))
\geq
d_X (f (\lbrack m\rbrack_{n+1}) ,f(\lbrack n\rbrack_{n+1}))
\end{equation}
for every $m\in\mathbb{Z}\cap\lbrack 0,n-1\rbrack$ because $X$ is $\mathrm{Cycl}_4 (\kappa )$ and we have 
\begin{align*}
d_X (f(\lbrack 0\rbrack_{n+1}),f(\lbrack m\rbrack_{n+1}))&\leq d_{\kappa}(g(\lbrack 0\rbrack_{n+1}),g(\lbrack m\rbrack_{n+1})),\\
d_X (f(\lbrack m\rbrack_{n+1}),f(\lbrack n-1\rbrack_{n+1}))&\leq d_{\kappa}(g(\lbrack m\rbrack_{n+1}),g(\lbrack n-1\rbrack_{n+1})),\\
d_X (f(\lbrack n-1\rbrack_{n+1}),f(\lbrack n\rbrack_{n+1}))&=d_{\kappa}(g(\lbrack n-1\rbrack_{n+1}),g(\lbrack n\rbrack_{n+1})) ,\\
d_X (f(\lbrack n\rbrack_{n+1}),f(\lbrack 0\rbrack_{n+1}))&=d_{\kappa}(g(\lbrack n\rbrack_{n+1}),g(\lbrack 0\rbrack_{n+1})),\\
d_X (f(\lbrack 0\rbrack_{n+1}),f(\lbrack n-1\rbrack_{n+1}))&=d_{\kappa}(g(\lbrack 0\rbrack_{n+1}),g(\lbrack n-1\rbrack_{n+1}))
\end{align*}
by \eqref{g-mn-ineq}, \eqref{g-k-1k-eq}, \eqref{g-k0-eq} and \eqref{g-k-10-eq}. 
By \eqref{g-mn-ineq}, \eqref{g-ll+1-eq}, \eqref{g-k-1k-eq}, \eqref{g-k0-eq} and \eqref{strongly-Cycl-th-g-i-k+1-ineq}, we have 
\begin{equation*}
d_{\kappa}(g(i),g(i+\lbrack 1\rbrack_{n+1}))=d_X (f(i),f(i+\lbrack 1\rbrack_{n+1})),\quad d_{\kappa}(g(i),g(j))\geq d_X (f(i),f(j))
\end{equation*}
for any $i,j\in\mathbb{Z}/(n+1)\mathbb{Z}$. 
Thus $g$ is a comparison map of $f$.

\textsc{Case 2}: 
{\em The map $g$ satisfies 
\begin{equation}\label{nongeodesic-majorization-th-case2-1}
\lbrack g(\lbrack n\rbrack_{n+1}),g(\lbrack n-2\rbrack_{n+1})\rbrack\cap\lbrack g(\lbrack n-1\rbrack_{n+1}),g(\lbrack 0\rbrack_{n+1})\rbrack =\emptyset
\end{equation}
or 
\begin{equation}\label{nongeodesic-majorization-th-case2-2}
\lbrack g(\lbrack n\rbrack_{n+1}),g(\lbrack 1\rbrack_{n+1})\rbrack\cap\lbrack g(\lbrack n-1\rbrack_{n+1}),g(\lbrack 0\rbrack_{n+1})\rbrack =\emptyset .
\end{equation}}
\noindent
If \eqref{nongeodesic-majorization-th-case2-1} holds, then we have 
$g(\lbrack n-2\rbrack_{n+1})\neq g(\lbrack 0\rbrack_{n+1})$, and it follows from 
Proposition \ref{convex-quadrilateral-prop} that we have 
\begin{equation}\label{strongly-cycl-th-n-2n-10-0n-1n-ineq}
\pi <
\angle g(\lbrack n-2\rbrack_{n+1})g(\lbrack n-1\rbrack_{n+1})g(\lbrack 0\rbrack_{n+1})
+
\angle g(\lbrack 0\rbrack_{n+1})g(\lbrack n-1\rbrack_{n+1})g(\lbrack n\rbrack_{n+1})
\end{equation}
or
\begin{equation}\label{strongly-cycl-th-n-20n-1-n-10n-ineq}
\pi <
\angle g(\lbrack n-2\rbrack_{n+1})g(\lbrack 0\rbrack_{n+1})g(\lbrack n-1\rbrack_{n+1})
+
\angle g(\lbrack n-1\rbrack_{n+1})g(\lbrack 0\rbrack_{n+1})g(\lbrack n\rbrack_{n+1}).
\end{equation}
Because we have 
\begin{equation*}
\angle g(\lbrack n-2\rbrack_{n+1})g(\lbrack 0\rbrack_{n+1})g(\lbrack n-1\rbrack_{n+1})\leq
\angle g(\lbrack 1\rbrack_{n+1})g(\lbrack 0\rbrack_{n+1})g(\lbrack n-1\rbrack_{n+1})
\end{equation*}
by Lemma \ref{two-angles-lemma}, 
\eqref{strongly-cycl-th-n-20n-1-n-10n-ineq} implies 
\begin{equation}\label{strongly-cycl-th-10n-1-n-10n-ineq}
\pi <
\angle g(\lbrack 1\rbrack_{n+1})g(\lbrack 0\rbrack_{n+1})g(\lbrack n-1\rbrack_{n+1})
+
\angle g(\lbrack n-1\rbrack_{n+1})g(\lbrack 0\rbrack_{n+1})g(\lbrack n\rbrack_{n+1}).
\end{equation}
In the case in which \eqref{nongeodesic-majorization-th-case2-2} holds, 
the same argument shows that we have 
\eqref{strongly-cycl-th-n-2n-10-0n-1n-ineq} or \eqref{strongly-cycl-th-10n-1-n-10n-ineq}. 
Thus we always have \eqref{strongly-cycl-th-n-2n-10-0n-1n-ineq} or \eqref{strongly-cycl-th-10n-1-n-10n-ineq} 
in \textsc{Case 2}. 
We may assume without loss of generality that we have \eqref{strongly-cycl-th-n-2n-10-0n-1n-ineq}. 
Let 
\begin{equation*}
S=\mathrm{conv}(g_0 (\mathbb{Z}/n\mathbb{Z})),\quad 
T=\mathrm{conv}(\{ g_0 (\lbrack n-1\rbrack_n ),p,g_0 (\lbrack 0\rbrack_n )\}).
\end{equation*}
Equip the subsets $S$ and $T$ of $M_{\kappa}^2$ with the induced metrics, and 
regard them as disjoint metric spaces. 
Define $(R, d_{R} )$ to be the metric space obtained by gluing $S$ and $T$ by identifying 
$\lbrack g_0 (\lbrack n-1\rbrack_n ),g_0 (\lbrack 0\rbrack_n )\rbrack\subseteq S$ with 
$\lbrack g_0 (\lbrack n-1\rbrack_n ),g_0 (\lbrack 0\rbrack_n)\rbrack\subseteq T$. 
Then $R$ is a $\mathrm{CAT}(\kappa )$ space by Reshetnyak's gluing theorem. 
We denote by $r_m$ the point in $R$ represented by $g_0 (\lbrack m\rbrack_n )\in S$ 
for each $m\in\mathbb{Z}\cap\lbrack 0,n-1\rbrack$, 
and by $r_n$ 
the point in $R$ represented by $p\in T$ 
(see \textsc{Figure} \ref{case3-fig}). 
\begin{figure}[htbp]
\centering
\begin{tikzpicture}[scale=0.6]
\draw [opacity=0.5,fill=gray] (0,3) -- (1,5) -- (3,6) -- (5,5) -- (7,0) -- (5,1.8) -- (3,0) -- (0.5,1) -- (0,3);
\draw [fill] (3,6) circle [radius=0.06];
\draw [fill] (5,5) circle [radius=0.06];
\draw [fill] (7,0) circle [radius=0.06];
\draw [fill] (5,1.8) circle [radius=0.06];
\draw [fill] (3,0) circle [radius=0.06];
\node [above] at (3.2,2.6) {$R$};
\node [above] at (3,6) {$r_1$};
\node [above] at (5.2,5) {$r_0$};
\node [right] at (7,0) {$r_n$};
\node [below] at (5.1,1.3) {$r_{n-1}$};
\node [below] at (3,0) {$r_{n-2}$};
\end{tikzpicture}
\caption{The $\mathrm{CAT}(\kappa )$ space $R$. }
\label{case3-fig}
\end{figure}
Define a map $f_1 :\mathbb{Z}/n\mathbb{Z}\to R$ by 
\begin{equation*}
f_1 (\lbrack m\rbrack_n )
=
\begin{cases}
r_m ,\quad &\textrm{if }m\in\mathbb{Z}\cap\lbrack 0,n-2\rbrack ,\\
r_{n},\quad &\textrm{if }m=n-1.
\end{cases}
\end{equation*}
Then it follows from \eqref{strongly-f-perimeter-distinct-assumption} and the definition of $f_1$ that 
\begin{equation*}
\sum_{i\in\mathbb{Z}/n\mathbb{Z}}d_{R}(f_1 (i),f_1 (i+\lbrack 1\rbrack_{n}))<2 D_{\kappa},\quad
f_1 (j)\neq f_1 (j+\lbrack 1\rbrack_{n})
\end{equation*}
for every $j\in\mathbb{Z}/n\mathbb{Z}$. 
Because $R$ is $\mathrm{Cycl}_4 (\kappa )$ by Theorem \ref{Cycl-facts-th}, 
there exists a comparison map $g_1 :\mathbb{Z}/n\mathbb{Z}\to M_{\kappa}^2$ of $f_1$ by the inductive hypothesis. 
It follows from \eqref{strongly-cycl-th-n-2n-10-0n-1n-ineq} and Proposition \ref{triangle-convexhull-coro-prop} that 
\begin{multline*}
d_{\kappa}(g(\lbrack n-2\rbrack_{n+1}),g(\lbrack n-1\rbrack_{n+1}))+d_{\kappa}(g(\lbrack n-1\rbrack_{n+1}),g(\lbrack n\rbrack_{n+1}))\\
\leq
d_{\kappa}(g(\lbrack n-2\rbrack_{n+1}),a)+d_{\kappa}(a,g(\lbrack n\rbrack_{n+1}))
\end{multline*}
for any $a\in\lbrack g(\lbrack n-1\rbrack_{n+1}),g(\lbrack 0\rbrack_{n+1})\rbrack$. 
By definition of the gluing of metric spaces, this implies 
\begin{equation}\label{strongly-Cycl-th-q-predef-eqs}
d_{\kappa}(g_1 (\lbrack n-2\rbrack_n ),g_1 (\lbrack n-1\rbrack_n ))=d_R (r_{n-2},r_{n})=d_R (r_{n-2},r_{n-1} )+d_R (r_{n-1} ,r_{n}).
\end{equation}
Therefore, there exists a point $q\in\lbrack g_1 (\lbrack n-2\rbrack_n ),g_1 (\lbrack n-1\rbrack_n )\rbrack$ such that 
\begin{equation}\label{strongly-Cycl-th-q-def-eqs}
d_{\kappa}(g_1 (\lbrack n-2\rbrack_n ),q)=d_R (r_{n-2},r_{n-1} ), \quad
d_{\kappa}(q,g_1 (\lbrack n-1\rbrack_n ))=d_R (r_{n-1} ,r_{n}). 
\end{equation}
Define a map $g_2 :\mathbb{Z}/(n+1)\mathbb{Z}\to M_{\kappa}^2$ by 
\begin{equation}\label{strongly-Cycl-th-cycl-g2-def}
g_2 (\lbrack m\rbrack_{n+1})
=
\begin{cases}
g_1 (\lbrack m\rbrack_n ) ,\quad &\textrm{if }m\in\mathbb{Z}\cap\lbrack 0,n-2\rbrack ,\\
q ,\quad &\textrm{if }m=n-1 ,\\
g_1 (\lbrack n-1\rbrack_{n}),\quad &\textrm{if }m=n.
\end{cases}
\end{equation}
Then 
\begin{align}\label{strongly-Cycl-th-g2-ij-ineq}
d_{\kappa}( g_2 (\lbrack l\rbrack_{n+1}),&g_2 (\lbrack m\rbrack_{n+1}))
=
d_{\kappa}(g_1 (\lbrack l\rbrack_{n}),g_1 (\lbrack m\rbrack_{n}))\\
\geq&
d_R (f_1 (\lbrack l\rbrack_{n}),f_1 (\lbrack m\rbrack_{n}))
=
d_R (r_l ,r_m )
=
d_{\kappa}(g_0 (\lbrack l\rbrack_{n}),g_0 (\lbrack m\rbrack_{n}))\nonumber\\
\geq&
d_X (f_0 (\lbrack l\rbrack_{n}),f_0 (\lbrack m\rbrack_{n}))
=
d_X (f(\lbrack l\rbrack_{n+1}),f(\lbrack m\rbrack_{n+1}))\nonumber
\end{align}
for any $l,m\in\mathbb{Z}\cap\lbrack 0,n-2\rbrack$, and 
\begin{align}\label{strongly-Cycl-th-g2-ij-eq}
d_{\kappa}(g_2 (\lbrack l\rbrack_{n+1}),&g_2 (\lbrack l+1\rbrack_{n+1}))
=
d_{\kappa}(g_1 (\lbrack l\rbrack_{n}),g_1 (\lbrack l+1\rbrack_{n})) \\
=&
d_R (f_1 (\lbrack l\rbrack_{n}),f_1 (\lbrack l+1\rbrack_{n}))
=
d_R (r_l ,r_{l+1} )
=
d_{\kappa}(g_0 (\lbrack l\rbrack_{n}),g_0 (\lbrack l+1\rbrack_{n})) \nonumber\\
=&
d_X (f_0 (\lbrack l\rbrack_{n}),f_0 (\lbrack l+1\rbrack_{n}))
=
d_X (f(\lbrack l\rbrack_{n+1}),f(\lbrack l+1\rbrack_{n+1}))\nonumber
\end{align}
for any $l\in\mathbb{Z}\cap\lbrack 0,n-3 \rbrack$. 
Furthermore, 
\begin{align}
\label{strongly-Cycl-th-g2-k-1-k-eq}
d_{\kappa}(g_2 (\lbrack n-2 &\rbrack_{n+1}),g_2 (\lbrack n-1\rbrack_{n+1}))
=
d_{\kappa}(g_1 (\lbrack n-2\rbrack_n ),q) \\
=&
d_R (r_{n-2} ,r_{n-1} )
=
d_{\kappa}(g_0 (\lbrack n-2\rbrack_n ),g_0 (\lbrack n-1\rbrack_n )) \nonumber\\
=&
d_X (f_0 (\lbrack n-2\rbrack_n ),f_0 (\lbrack n-1\rbrack_n ))
=
d_X (f(\lbrack n-2\rbrack_{n+1}),f(\lbrack n-1\rbrack_{n+1})), \nonumber
\end{align}
\begin{align}
\label{strongly-Cycl-th-g2-k-k+1-eq}
d_{\kappa}(g_2 (\lbrack n-1\rbrack_{n+1}),&g_2 (\lbrack n\rbrack_{n+1}))
=
d_{\kappa}(q,g_1 (\lbrack n-1\rbrack_{n}))
=
d_R (r_{n-1} ,r_{n}) \\
&=
d_{\kappa}(g_0 (\lbrack n-1\rbrack_n ),p)
=
d_X (f(\lbrack n-1\rbrack_{n+1}),f(\lbrack n\rbrack_{n+1})), \nonumber
\end{align}
\begin{align}
\label{strongly-Cycl-th-g2-k+1-1-eq}
d_{\kappa}(g_2 (\lbrack n\rbrack_{n+1}),&g_2 (\lbrack 0\rbrack_{n+1}))
=
d_{\kappa}(g_1 (\lbrack n-1\rbrack_{n}),g_1 (\lbrack 0\rbrack_{n})) \\
&=
d_R (f_1 (\lbrack n-1\rbrack_{n}),f_1 (\lbrack 0\rbrack_{n}))
=
d_R (r_{n},r_0 )\nonumber\\
&=
d_{\kappa}(p,g_0 (\lbrack 0\rbrack_{n}))
=
d_X (f(\lbrack n\rbrack_{n+1}),f(\lbrack 0\rbrack_{n+1})) .\nonumber
\end{align}
It follows from Corollary \ref{original-def-coro} that  
\begin{equation*}
d_{\kappa}(g_1 (\lbrack m\rbrack_n ),q)\geq d_R (r_m ,r_{n-1})
\end{equation*}
for every $m\in\mathbb{Z}\cap\lbrack 0,n-2\rbrack$ 
because $R$ is $\mathrm{Cycl}_4 (\kappa )$ and we have 
\begin{align*}
d_{\kappa}(g_1 (\lbrack n-1\rbrack_n ),g_1 (\lbrack m\rbrack_n ))
&\geq
d_R (f_1 (\lbrack n-1\rbrack_n ),f_1 (\lbrack m\rbrack_n))
=
d_R (r_{n},r_m ), \\
d_{\kappa}(g_1 (\lbrack m\rbrack_n ),g_1 (\lbrack n-2\rbrack_n ))
&\geq
d_R (f_1 (\lbrack m\rbrack_n ),f_1 (\lbrack n-2\rbrack_n ))
=
d_R (r_m ,r_{n-2}),\\
d_{\kappa}(g_1 (\lbrack n-2\rbrack_n ),q)
&=
d_R (r_{n-2},r_{n-1} ),\quad
d_{\kappa}(q,g_1 (\lbrack n-1\rbrack_n ))=d_R (r_{n-1} ,r_{n}),\\
d_{\kappa}(g_1 (\lbrack n-2\rbrack_n ),g_1(\lbrack n-1\rbrack_{n}))
&=
d_R (f_1 (\lbrack n-2\rbrack_n ),f_1 (\lbrack n-1\rbrack_n ))
=
d_R (r_{n-2},r_{n}),
\end{align*}
and $\lbrack g_1 (\lbrack n-2\rbrack_{n}), g_1 (\lbrack n-1\rbrack_{n})\rbrack\cap\lbrack g_1 (\lbrack m\rbrack_{n}),q\rbrack\neq\emptyset$. 
Therefore, we have 
\begin{align}\label{strongly-Cycl-th-g2-i-k-ineq}
d_{\kappa}(g_2 (\lbrack m\rbrack_{n+1}),&g_2 (\lbrack n-1\rbrack_{n+1}))
=
d_{\kappa}(g_1 (\lbrack m\rbrack_{n}),q)\\
&\geq
d_R (r_m ,r_{n-1})
=
d_{\kappa}(g_0 (\lbrack m\rbrack_{n}),g_0 (\lbrack n-1\rbrack_n )) \nonumber\\
&\geq
d_X (f_0 (\lbrack m\rbrack_{n}),f_0 (\lbrack n-1\rbrack_n ))
=
d_X (f(\lbrack m\rbrack_{n+1}),f(\lbrack n-1\rbrack_{n+1}))\nonumber
\end{align}
for every $m\in\mathbb{Z}\cap\lbrack 0,n-2\rbrack$. 
Let $S'$ and $T'$ be the images of $S$ and $T$, respectively under the natural inclusions into $R$. 
Then clearly 
\begin{equation*}
S'\cap T' =\lbrack r_0 ,r_{n-1} \rbrack ,\quad
\lbrack r_0 ,r_{n-1} \rbrack\cap\lbrack r_m, r_{n}\rbrack\neq\emptyset
\end{equation*}
for every $m\in\mathbb{Z}\cap\lbrack 0,n-1\rbrack$. 
Therefore, Corollary \ref{simplicial-coro} implies  
\begin{equation*}
d_X (f(\lbrack m\rbrack_{n+1}),f(\lbrack n\rbrack_{n+1}))
\leq
d_R (r_m ,r_{n})
\end{equation*}
for every $m\in\mathbb{Z}\cap\lbrack 0,n-1\rbrack$ because $X$ is $\mathrm{Cycl}_4 (\kappa )$ and we have 
\begin{align*}
d_R (r_0 ,r_m )
&=
d_{\kappa}(g_0 (\lbrack 0\rbrack_n ),g_0 (\lbrack m\rbrack_n ))
\geq
d_X (f_0 (\lbrack 0\rbrack_n ) ,f_0 (\lbrack m\rbrack_n ))\\
&=
d_X (f(\lbrack 0\rbrack_{n+1}),f(\lbrack m\rbrack_{n+1})), \\
d_R (r_m ,r_{n-1})
&=
d_{\kappa}(g_0 (\lbrack m\rbrack_n ),g_0 (\lbrack n-1\rbrack_n ))
\geq
d_X (f_0 (\lbrack m\rbrack_n ) ,f_0(\lbrack n-1\rbrack_n ))\\
&=
d_X (f(\lbrack m\rbrack_{n+1}),f(\lbrack n-1\rbrack_{n+1} )), \\
d_R (r_{n-1} ,r_{n})
&=
d_{\kappa}(g_0 (\lbrack n-1\rbrack_{n}),p)
=
d_X (f(\lbrack n-1\rbrack_{n+1}),f(\lbrack n\rbrack_{n+1})),\\
d_R (r_{n},r_0 )
&=
d_{\kappa}(p,g_0 (\lbrack 0\rbrack_n ))
=
d_X (f(\lbrack n\rbrack_{n+1}),f(\lbrack 0\rbrack_{n+1})),\\
d_R (r_0,r_{n-1} )
&=
d_{\kappa}(g_0 (\lbrack 0\rbrack_n ),g_0 (\lbrack n-1\rbrack_{n}))
=
d_X (f_0 (\lbrack 0\rbrack_n ),f_0 (\lbrack n-1\rbrack_{n}))\\
&=
d_X (f(\lbrack 0\rbrack_{n+1}),f(\lbrack n-1\rbrack_{n+1})).
\end{align*}
Therefore, we have 
\begin{align}\label{strongly-Cycl-th-g2-i-k+1-ineq}
d_{\kappa}(g_2 (\lbrack m\rbrack_{n+1}),g_2 (\lbrack n\rbrack_{n+1}))
&=
d_{\kappa}(g_1 (\lbrack m\rbrack_{n}),g_1 (\lbrack n-1\rbrack_{n}))\\
&\geq
d_R (f_1 (\lbrack m\rbrack_{n}),f_1 (\lbrack n-1\rbrack_{n})) \nonumber\\
&=
d_R (r_m ,r_{n})
\geq
d_X (f(\lbrack m\rbrack_{n+1}),f(\lbrack n\rbrack_{n+1}))\nonumber
\end{align}
for every $m\in\mathbb{Z}\cap\lbrack 0,n-2\rbrack$. 
It follows form \eqref{strongly-Cycl-th-g2-ij-ineq}, \eqref{strongly-Cycl-th-g2-ij-eq}, 
\eqref{strongly-Cycl-th-g2-k-1-k-eq}, \eqref{strongly-Cycl-th-g2-k-k+1-eq}, \eqref{strongly-Cycl-th-g2-k+1-1-eq}, 
\eqref{strongly-Cycl-th-g2-i-k-ineq} and \eqref{strongly-Cycl-th-g2-i-k+1-ineq} that 
\begin{equation*}
d_{\kappa}(g_2 (i),g_2 (i+\lbrack 1\rbrack_{n+1}))=d_X (f(i),f(i+\lbrack 1\rbrack_{n+1})),\quad
d_{\kappa}(g_2 (i),g_2 (j))\geq d_X (f(i),f(j))
\end{equation*}
for any $i,j\in\mathbb{Z}/(n+1)\mathbb{Z}$. 
Since $g_2 (\lbrack n-1\rbrack_{n+1})\in\lbrack g_2 (\lbrack n-2\rbrack_{n+1}),g_2 (\lbrack n\rbrack_{n+1})\rbrack$ by definition of $g_2$, 
we clearly have 
\begin{align*}
\lbrack g_2 (\lbrack n-1\rbrack_{n+1}),g_2 (\lbrack 0\rbrack_{n+1})\rbrack\cap\lbrack g_2 (\lbrack n-2\rbrack_{n+1}),g_2 (\lbrack n\rbrack_{n+1})\rbrack
&\neq\emptyset ,\\
\lbrack g_2 (\lbrack n-1\rbrack_{n+1}),g_2 (\lbrack n-3\rbrack_{n+1})\rbrack\cap\lbrack g_2 (\lbrack n-2\rbrack_{n+1}),g_2 (\lbrack n\rbrack_{n+1})\rbrack
&\neq\emptyset .
\end{align*}
Therefore, it follows from the definition of $g_2$ and Lemma \ref{plus-one-lemma} that 
\begin{equation*}
\lbrack g_2 (i),g_2 (j)\rbrack\cap\lbrack g_2 (i-\lbrack 1\rbrack_{n+1} ),g_2 (i+\lbrack 1\rbrack_{n+1})\rbrack\neq\emptyset
\end{equation*}
for any $i,j\in\mathbb{Z}/(n+1)\mathbb{Z}$ with $i\neq j$. 
Therefore, $g_2$ is a comparison map of $f$. 

We have proved the existence of a comparison map of $f$ under the assumption that 
$f(\lbrack n-1\rbrack_{n+1})\neq f(\lbrack 0\rbrack_{n+1})$. 
So assume henceforth that $f(\lbrack n-1\rbrack_{n+1})=f(\lbrack 0\rbrack_{n+1})$, and set 
\begin{equation*}
d
=
d_{X}(f(\lbrack n-1\rbrack_{n+1}),f(\lbrack n\rbrack_{n+1}))
=
d_{X}(f(\lbrack 0\rbrack_{n+1}),f(\lbrack n\rbrack_{n+1})).
\end{equation*}
Define a map $\tilde{f}_0 :\mathbb{Z}/(n-1)\mathbb{Z}\to X$ by 
$\tilde{f}_0 (\lbrack m\rbrack_{n-1})=f(\lbrack m\rbrack_{n+1})$, $m\in\mathbb{Z}\cap\lbrack 0,n-2\rbrack$. 
Then we have 
\begin{equation*}
\sum_{i\in\mathbb{Z}/(n-1)\mathbb{Z}}d_X (\tilde{f}_0 (i),\tilde{f}_0 (i+\lbrack 1\rbrack_{n-1}))<2 D_{\kappa},\quad
\tilde{f}_0 (j)\neq\tilde{f}_0 (j+\lbrack 1\rbrack_{n-1})
\end{equation*}
for every $j\in\mathbb{Z}/(n-1)\mathbb{Z}$ by \eqref{strongly-f-perimeter-distinct-assumption}. 
Therefore, by the inductive hypothesis, 
there exists a comparison map $\tilde{g}_0 :\mathbb{Z}/(n-1)\mathbb{Z}\to M_{\kappa}^2$ of $\tilde{f}_0$. 
Let 
\begin{equation*}
\tilde{S}=\mathrm{conv}(\tilde{g}_0 (\mathbb{Z}/(n-1)\mathbb{Z})),\quad
\tilde{T}=\lbrack 0, d\rbrack .
\end{equation*}
Equip $\tilde{S}\subseteq M_{\kappa}^2$ and $\tilde{T}\subseteq\mathbb{R}$ the induced metrics, 
and regard them as metric spaces in their own right. 
Define $(\tilde{R},d_{\tilde{R}})$ to be the metric space obtained by gluing $\tilde{S}$ and $\tilde{T}$ by identifying 
$\{\tilde{g}_0 (\lbrack 0\rbrack_{n-1})\}\subseteq\tilde{S}$ with $\{ 0\}\subseteq\tilde{T}$. 
Then $\tilde{R}$ is a $\mathrm{CAT}(\kappa )$ space by Reshetnyak's gluing theorem. 
We denote by $\tilde{r}_m$ the point in $\tilde{R}$ represented by $\tilde{g}_0 (\lbrack m\rbrack_{n-1})\in\tilde{S}$ for 
each $m\in\mathbb{Z}\cap\lbrack 0,n-2\rbrack$, 
by $\tilde{r}_{n-1}$ the point in $\tilde{R}$ represented by $\tilde{g}_0 (\lbrack 0\rbrack_{n-1})\in\tilde{S}$, 
and by $\tilde{r}_n$ the point in $\tilde{R}$ represented by $d\in\tilde{T}$. 
In particular, we have $\tilde{r}_0 =\tilde{r}_{n-1}$. 
Define a map $\tilde{f}_1 :\mathbb{Z}/n\mathbb{Z}\to R$ by 
\begin{equation*}
\tilde{f}_1 (\lbrack m\rbrack_n )
=
\begin{cases}
\tilde{r}_m ,\quad &\textrm{if }m\in\mathbb{Z}\cap\lbrack 0,n-2\rbrack ,\\
\tilde{r}_{n},\quad &\textrm{if }m=n-1.
\end{cases}
\end{equation*}
Then it follows from \eqref{strongly-f-perimeter-distinct-assumption} and the definition of $\tilde{f}_1$ that 
\begin{equation*}
\sum_{i\in\mathbb{Z}/n\mathbb{Z}}d_{\tilde{R}}(\tilde{f}_1 (i),\tilde{f}_1 (i+\lbrack 1\rbrack_{n}))<2 D_{\kappa},\quad
\tilde{f}_1 (j)\neq\tilde{f}_1 (j+\lbrack 1\rbrack_{n})
\end{equation*}
for every $j\in\mathbb{Z}/n\mathbb{Z}$. 
Because $\tilde{R}$ is $\mathrm{Cycl}_4 (\kappa )$ by Theorem \ref{Cycl-facts-th}, 
there exists a comparison map $\tilde{g}_1 :\mathbb{Z}/n\mathbb{Z}\to M_{\kappa}^2$ of $\tilde{f}_1$ by the inductive hypothesis. 
By definition of the gluing of metric spaces, we have 
\begin{equation*}
d_{\kappa}(\tilde{g}_1 (\lbrack n-2\rbrack_n ),\tilde{g}_1 (\lbrack n-1\rbrack_n ))
=
d_{\tilde{R}}(\tilde{r}_{n-2},\tilde{r}_{n})=d_{\tilde{R}}(\tilde{r}_{n-2},\tilde{r}_{n-1} )+d_{\tilde{R}}(\tilde{r}_{n-1} ,\tilde{r}_{n}).
\end{equation*}
Therefore, there exists a point $\tilde{q}\in\lbrack \tilde{g}_1 (\lbrack n-2\rbrack_n ),\tilde{g}_1 (\lbrack n-1\rbrack_n )\rbrack$ such that 
\begin{equation*}
d_{\kappa}(\tilde{g}_1 (\lbrack n-2\rbrack_n ),\tilde{q})=d_{\tilde{R}}(\tilde{r}_{n-2},\tilde{r}_{n-1} ), \quad
d_{\kappa}(\tilde{q},\tilde{g}_1 (\lbrack n-1\rbrack_n ))=d_{\tilde{R}}(\tilde{r}_{n-1} ,\tilde{r}_{n}). 
\end{equation*}
Define a map $\tilde{g}_2 :\mathbb{Z}/(n+1)\mathbb{Z}\to M_{\kappa}^2$ by 
\begin{equation*}
\tilde{g}_2 (\lbrack m\rbrack_{n+1})
=
\begin{cases}
\tilde{g}_1 (\lbrack m\rbrack_n ) ,\quad &\textrm{if }m\in\mathbb{Z}\cap\lbrack 0,n-2\rbrack ,\\
\tilde{q},\quad &\textrm{if }m=n-1 ,\\
\tilde{g}_1 (\lbrack n-1\rbrack_{n}),\quad &\textrm{if }m=n.
\end{cases}
\end{equation*}
Then the same argument as in \textsc{Case 2} shows that 
$\tilde{g}_2$ is a comparison map of $f$, which completes the proof. 
\end{proof}

Theorem \ref{Cycl-th} follows from Theorem \ref{nongeodesic-majorization-th} and Lemma \ref{strongly-Cycl-relation-lemma}. 

\begin{proof}[Proof of Theorem \ref{Cycl-th}]
Let $X$ be a $\mathrm{Cycl}_4 (\kappa )$ space. 
Then $X$ is $\mathrm{Cycl}'_m (\kappa)$ for every positive integer $m$ as shown in the proof of Theorem \ref{nongeodesic-majorization-th}. 
Thus $X$ is $\mathrm{Cycl}_n (\kappa )$ for every integer $n\geq 4$ by Lemma \ref{strongly-Cycl-relation-lemma}. 
\end{proof}

\section{Proof of Theorem \ref{Cycl4-boxtimes-th}}\label{boxtimes-sec}

In this section, we present a proof of Theorem \ref{Cycl4-boxtimes-th} for completeness (see Remark \ref{proof-of-Cycl4-boxtimes-remark}). 
First, we recall the following fact, which was established by Sturm 
when he proved in \cite[Theorem 4.9]{St} that 
if a geodesic metric space satisfies the $\boxtimes$-inequalities, then it is $\mathrm{CAT}(0)$.

\begin{proposition}\label{sdi-CAT(0)-prop}
Let $(X, d_X )$ be a metric space that satisfies the $\boxtimes$-inequalities. 
Suppose $x,y,z\in X$ are points such that 
$x\neq z$, and 
\begin{equation}\label{sdi-CAT(0)-prop-triangle-equality-eq}
d_X (x,z)=d_X (x,y)+d_X (y,z). 
\end{equation}
Set $t=d_X (x,y)/d_X (x,z)$. 
Then we have 
\begin{equation*}
d_X (y,w) ^2 \le 
 (1 - t) d_X (x,w) ^2 + t d_X (z,w) ^2 - t(1 - t) d_X (x,z)^2 .
\end{equation*}
for any $w\in X$. 
\end{proposition}

\begin{proof}
By the hypothesis \eqref{sdi-CAT(0)-prop-triangle-equality-eq}, we compute 
\begin{align*}
(1-t)d_X (x,y)^2 +td_X (y,z)^2
&=
\frac{d_X (y,z)}{d_X (x,z)}d_X (x,y)^2+\frac{d_X (x,y)}{d_X (x,z)}d_X (y,z)^2 \\
&=
\frac{d_X (x,y)d_X (y,z)}{d_X (x,z)}\left(d_X (x,y)+d_X (y,z)\right) \\
&=
d_X (x,y)d_X (y,z) \\
&=
t(1-t)d_X (x,z)^2 .
\end{align*}
Combining this with the $\boxtimes$-inequality in $X$ yields 
\begin{align*}
0
&\leq
(1-t)(1-s) d_X (x,y)^2 + t(1-s) d_X (y,z)^2 + ts d_X (z,w)^2 \\
&\hspace{4mm}
+s(1-t) d_X (w,x)^2- t(1-t) d_X (x,z)^2  -s(1-s) d_X (y,w)^2 \\
&=
(1-s)\left( (1-t) d_X (x,y)^2 + td_X (y,z)^2 \right) + ts d_X (z,w)^2 \\
&\hspace{4mm}
+s(1-t) d_X (w,x)^2- t(1-t) d_X (x,z)^2  -s(1-s) d_X (y,w)^2 \\
&=
(1-s)t(1-t)d_X (x,z)^2 + ts d_X (z,w)^2 \\
&\hspace{4mm}
+s(1-t) d_X (w,x)^2- t(1-t) d_X (x,z)^2  -s(1-s) d_X (y,w)^2 \\
&=
ts d_X (z,w)^2 +s(1-t) d_X (w,x)^2
-st(1-t) d_X (x,z)^2  -s(1-s) d_X (y,w)^2 
\end{align*}
for every $s\in\lbrack 0,1\rbrack$. 
For any $s\in (0,1\rbrack$, dividing this by $s$, we obtain 
\begin{equation*}
(1-s) d_X (y,w)^2
\leq
t d_X (z,w)^2 +(1-t) d_X (w,x)^2 
-t(1-t) d_X (x,z)^2 .
\end{equation*}
Letting $s\to 0$ in this inequality 
yields the desired inequality. 
\end{proof}

We now prove Theorem \ref{Cycl4-boxtimes-th}. 

\begin{proof}[Proof of Theorem \ref{Cycl4-boxtimes-th}]
First, assume that a metric space $(X,d_X )$ is $\mathrm{Cycl}_4(0)$. 
Then for any $x,y,z,w \in X$, 
there exist $x',y',z',w' \in\mathbb{R}^2$ such that 
\begin{align*}
&\|x'-y' \|\leq d_X (x,y),\quad
\|y'-z' \|\leq d_X (y,z),\quad
\|z'-w'\|\leq d_X (z,w),\\
&\|w'-x'\|\leq d_X (w,x),\quad
\|x'-z' \|\geq d_X (x,z),\quad
\|y'-w'\|\geq d_X (y,w). 
\end{align*}
Therefore, for any $s,t\in\lbrack 0,1\rbrack$, we have 
\begin{align*}
    & (1-t)(1-s) d_X (x,y)^2 + t(1-s) d_X (y,z)^2 
            + ts d_X (z,w)^2 + (1-t)s d_X (w,x)^2  \\
	&	- t(1-t) d_X (x,z)^2 - s(1-s) d(y,w)^2  \\
 \geq & (1-t)(1-s) \|x'-y'\|^2 + t(1-s) \|y'-z'\|^2 
             + ts \|z'-w'\|^2 + (1-t)s \|w'-x'\|^2  \\
	& 	 - t(1-t) \|x'-z'\|^2 - s(1-s) \|y'-w'\|^2  \\
	= &  \|((1-t)x' + t z') - ((1-s)y' + s w') \|^2 \geq 0, 
\end{align*}
which means that $X$ satisfies the $\boxtimes$-inequalities.  

For the converse, assume that $(X,d_X )$ satisfies the $\boxtimes$-inequalities. 
Fix $x,y,z,w\in X$. 
If $x$, $y$, $z$ and $w$ are not distinct, then we can embed $\{ x,y,z,w\}$ isometrically into $\mathbb{R}^2$. 
So we assume that $x$, $y$, $z$ and $w$ are distinct. 
Then there exist $x' ,y',z' ,w'\in\mathbb{R}^2$ such that 
\begin{align*}
&\| x'-y'\| =d_X (x,y),\quad \| y'-z'\| =d_X (y,z),\quad \| z'-w'\| =d_X (z,x),\\
&\| x'-w'\| =d_X (x,w),\quad \| w'-z'\|=d_X (w,z), 
\end{align*}
and $y'$ and $w'$ do not lie on the same side of $\ell (x', z')$. 
We consider three cases. 

\textsc{Case 1}: 
{\em $\lbrack x' ,z' \rbrack\cap (y' ,w' )\neq\emptyset$.} 
In this case, 
there exist $s\in (0, 1)$ and $t\in\lbrack 0,1\rbrack$ such that 
\begin{equation*}
(1-t)x' + tz'
=
(1-s)y' + sw'.
\end{equation*}
It follows that 
\begin{align*}
  0= & \left\|\left( (1-t)x' +t z' \right) -
       \left( (1-s)y' +s w'\right)\right\|^2 \\
	=& (1-t)(1-s)\| x' -y' \|^2
	  + t(1-s)   \| y' -z' \|^2
	  + ts       \| z' -w'\|^2
	  + (1-t)s   \| w' -x' \|^2 \\
	& - t(1-t)   \| x' -z'\|^2
	  - s(1-s)   \| y' -w'\|^2 \\
	=& (1-t)(1-s) d_X (x,y)^2
	  + t(1-s)     d_X (y,z)^2
	  + ts         d_X (z,w)^2
	  + (1-t)s     d_X (w,x)^2 \\
	& - t(1-t)     d_X (x,z)^2
	  - s(1-s) \| y' -w'\|^2 .
\end{align*}
On the other hand, we have 
\begin{align*}
0\leq  &(1-t)(1-s) d_X (x,y)^2
   + t(1-s)   d_X (y,z)^2
   + ts       d_X (z,w)^2
   + (1-t)s   d_X (w,x)^2 \\
   - & t(1-t) d_X (x,z)^2
   - s(1-s)   d_X (y,w)^2
\end{align*}
because $X$ satisfies the $\boxtimes$-inequalities. 
Comparing these yields 
$d_X (y,w)\leq\| y'-w'\|$.

\textsc{Case 2}: 
{\em $\lbrack x' ,z' \rbrack\cap\{y' ,w' \}\neq\emptyset$.}  
In this case, 
we may assume without loss of generality that $y'\in\lbrack x' ,z'\rbrack$. 
Then we can write 
\begin{equation*}
y'=(1-c)x' + cz' ,
\end{equation*}
where 
\begin{equation}\label{quadrangle-lemma-k-eq}
c=\frac{\| x'-y'\|}{\| x'-z'\|}=\frac{d_X (x,y)}{d_X (x,z)}\in (0,1).
\end{equation} 
It follows that 
\begin{align*}
\label{quadrangle-lemma-compare-eq1}
\| y' -w'\|^2
&=\|(1-c)x' +cz' -w' \|^2 \\
&=(1-c)\|x' -w'\|^2 +c\|z' -w' \|^2 -c(1-c)\|x' -z' \|^2 \\
&= (1-c)d_X (x,w)^2 +cd_X (z,w)^2 -c(1-c)d_X (x,z)^2 .
\end{align*}
On the other hand, it follows from 
\eqref{quadrangle-lemma-k-eq} and Proposition \ref{sdi-CAT(0)-prop} that 
\begin{equation*}\label{quadrangle-lemma-compare-ineq1}
d_X (y,w)^2
\leq 
(1 - c) d_X (x,w)^2 + c d_X (z,w)^2 - c(1-c) d_X (x,z)^2
\end{equation*}
because we have 
\begin{equation*}
d_X (x,z)=\| x'-z' \| =\| x'-y' \| +\| y'-z' \|
=d_X (x,y)+d_X (y,z). 
\end{equation*}
Combining these yields 
$d_X (y,w)\leq\| y'-w'\|$.

\textsc{Case 3}: 
{\em $\lbrack x' ,z' \rbrack\cap\lbrack y' ,w' \rbrack=\emptyset$.} 
In this case, it follows from Proposition \ref{convex-quadrilateral-prop} that we have 
$\angle y'z'x'+\angle x'z'w' >\pi$ or $\angle y'x'z'+\angle z'x'w' >\pi$. 
We may assume without loss of generality that $\angle y'z'x'+\angle x'z'w' >\pi$. 
Then we have 
\begin{equation*}
d_X (y,z)+d_X (z,w)=\|y' -z' \|+\| z'-w' \|\leq\|y' -x' \|+\| x'-w' \| =d_X (y,x)+d_X (x,w)
\end{equation*}
by Proposition \ref{triangle-convexhull-prop}. 
Therefore, there exist $\tilde{x},\tilde{y},\tilde{w}\in\mathbb{R}^2$ such that 
\begin{equation*}
\|\tilde{y}-\tilde{w}\|=d_X (y,z)+d_X (z,w),\quad
\|\tilde{w}-\tilde{x}\|=d_X (w,x),\quad
\|\tilde{x}-\tilde{y}\|=d_X (x,y).
\end{equation*}
Then we clearly have $d_X (y,w)\leq\|\tilde{y}-\tilde{w}\|$, and we can take 
$\tilde{z}\in\lbrack\tilde{y},\tilde{w}\rbrack$ such that 
\begin{equation*}
\|\tilde{y}-\tilde{z}\|=d_X (y,z),\quad
\|\tilde{z}-\tilde{w}\|=d_X (z,w).
\end{equation*}
Because we have 
\begin{equation*}
\| y' -w' \|
\leq\| y' -z' \|+\| z' -w' \| =d_X (y,z)+d_X (z,w) 
=\|\tilde{y}-\tilde{w}\| ,
\end{equation*}
it follows from Lemma \ref{larger-larger-lemma} that 
\begin{equation*}
d_X (x,z)=\| x' -z' \|\leq\|\tilde{x}-\tilde{z}\| .
\end{equation*}
Thus the points $\tilde{x},\tilde{y},\tilde{z},\tilde{w}\in\mathbb{R}^2$ have the desired property.

The above three cases exhaust all possibilities, and thus 
$X$ is $\mathrm{Cycl}_4 (0)$. 
\end{proof}

\begin{remark}\label{proof-of-Cycl4-boxtimes-remark}
A proof of Theorem \ref{Cycl4-boxtimes-th} can also be found in \cite[Lemma 2.6]{KTU}. 
However, the case corresponding to \textsc{Case 2} in the above proof is omitted in the proof in \cite{KTU}. 
\end{remark}

\medskip
\begin{acknowledgement}
The author would like to thank Takefumi Kondo, Yu Kitabeppu and Toshimasa Kobayashi for 
helpful discussions and a number of valuable comments on the first version of this paper. 
The author also would like to thank Masato Mimura for helpful discussions, especially for noting that it follows from the result of \cite{EMN} 
that the $\mathrm{Cycl}_4 (0)$ condition do not imply the coarse embeddability into a $\mathrm{CAT}(0)$ space.
\end{acknowledgement}


\begin{thebibliography}{10}

\bibitem{AKP}
S.~Alexander, V.~Kapovitch, and A.~Petrunin. 
\newblock Alexandrov meets Kirszbraun. 
\newblock {\em Proceedings of the G\"{o}kova Geometry-Topology Conference 2010}. Int. Press, Somerville, MA, 2011, 88--109.

\bibitem{ANN}
A.~Andoni, A.~Naor, and O.~Neiman.
\newblock Snowflake universality of {W}asserstein spaces.
\newblock {\em Ann. Sci. \'{E}c. Norm. Sup\'{e}r. (4)}, 51(3):657--700, 2018.

\bibitem{B}
N. Ballmann.
\newblock {\em Lectures on spaces of nonpositive curvature}. 
\newblock {Vol. 25. DMV Sminar}. 
\newblock {With an appendix by Misha Brin.} 1995. 


\bibitem{BN}
I.~D. Berg and I.~G. Nikolaev.
\newblock Quasilinearization and curvature of {A}leksandrov spaces.
\newblock {\em Geom. Dedicata}, 133:195--218, 2008.

\bibitem{BH}
M.~R. Bridson and A.~Haefliger.
\newblock {\em Metric spaces of non-positive curvature}, volume 319 of {\em
  Grundlehren der Mathematischen Wissenschaften [Fundamental Principles of
  Mathematical Sciences]}.
\newblock Springer-Verlag, Berlin, 1999.

\bibitem{BBI}
D.~Burago, Y.~Burago, and S.~Ivanov.
\newblock {\em A course in metric geometry}, volume~33 of {\em Graduate Studies
  in Mathematics}.
\newblock American Mathematical Society, Providence, RI, 2001.


\bibitem{EMN}
A. ~Eskenazis, M. ~Mendel, and A. ~Naor. 
\newblock Nonpositive curvature is not coarsely universal. 
\newblock {\em Invent. Math.}, 217(3):833--886, 2019.

\bibitem{Gr1}
M.~Gromov.
\newblock {\em Metric structures for {R}iemannian and non-{R}iemannian spaces},
  volume 152 of {\em Progress in Mathematics}.
\newblock Birkh\"auser Boston, Inc., Boston, MA, 1999.
\newblock Based on the 1981 French original [ MR0682063 (85e:53051)], With 
  appendices by M. Katz, P. Pansu and S. Semmes, Translated from the French by Sean Michael Bates.

\bibitem{Gr2} 
M.~Gromov.
\newblock {${\rm CAT}(\kappa)$}-spaces: construction and concentration.
\newblock {\em Zap. Nauchn. Sem. S.-Peterburg. Otdel. Mat. Inst. Steklov.
  (POMI)}, 280(Geom. i Topol. 7):100--140, 299--300, 2001.

\bibitem{KTU} 
T. ~Kondo, T. ~Toyoda, and T. ~Uehara. 
\newblock On a question of Gromov about the Wirtinger inequalities. 
\newblock {\em Geom. Dedicata}, 195(1):203--214, 2018.

\bibitem{LPZ}
N.~Lebedeva, A.~Petrunin and V.~Zolotov.
\newblock Bipolar comparison.
\newblock {\em Geom. Funct. Anal.}, 29(1):258--282, 2019. 

\bibitem{Pe}
P.~Pech. 
\newblock Inequality between sides and diagonals of a space $n$-gon and its integral analog. 
\newblock {\em \v{C}asopis P\v{e}st. Mat.}, 115 (1990), no. 4, 343--350. 

\bibitem{R}
J. G. ~Re\v{s}etnjak. 
\newblock Non-expansive maps in a space of curvature no greater than $K$. 
\newblock {\em Sibirsk. Mat. \v{Z}}, 9:918--927, 1968. 

\bibitem{Sa}
T. ~Sato. 
\newblock An alternative proof of Berg and Nikolaev's characterization of $\rm CAT(0)$-spaces via quadrilateral inequality. 
\newblock {\em Arch. Math. (Basel)}, 93(5):487--490, 2009.

\bibitem{St}
K.-T. Sturm.
\newblock Probability measures on metric spaces of nonpositive curvature.
\newblock In {\em Heat kernels and analysis on manifolds, graphs, and metric
  spaces ({P}aris, 2002)}, volume 338 of {\em Contemp. Math.}, pages 357--390.
  Amer. Math. Soc., Providence, RI, 2003.

\bibitem{toyoda-five}
T. ~Toyoda. 
\newblock An intrinsic characterization of five points in a $\mathrm{CAT}(0)$ space. 
\newblock {\em Anal. Geom. Metr. Spaces}, 8(1):114--165, 2020.


\end{thebibliography}
\end{document}